\documentclass[11pt]{article}
\usepackage{mathrsfs}
\usepackage{amsmath,amssymb}
\usepackage{amsthm,bm}
\usepackage[colorlinks,linkcolor=red,anchorcolor=blue,citecolor=green]{hyperref}
\usepackage{graphics,graphicx}
\usepackage{color}
\usepackage{enumerate}
\usepackage[title]{appendix}
\usepackage{subfigure}
\usepackage{caption}
\usepackage{enumitem}
\setitemize[1]{itemsep=3pt,partopsep=0pt,parsep=\parskip,topsep=3pt}

\allowdisplaybreaks

\usepackage{tikz}
\usepackage{indentfirst}

\newtheorem{theorem}{Theorem}[section]
\newtheorem{lemma}[theorem]{Lemma}
\newtheorem{corollary}[theorem]{Corollary}
\newtheorem{proposition}[theorem]{Proposition}
\numberwithin{equation}{section}
\newtheorem{definition}[theorem]{Definition}
\theoremstyle{definition}
\newtheorem{remark}[theorem]{Remark}
\newtheorem{example}[theorem]{Example}
\newcommand{\runum}[1]{\romannumeral #1}

\def\R{{\mathbb{R}}}
\def\E{{\mathbb{E}}}
\def\P{{\mathbb{P}}}

\def\im{{\mathrm{i}}}

\def\Z{{\mathbb{Z}}}
\def\N{{\mathbb{N}}}
\def\C{{\mathbb{C}}}
\def\T{{\mathbb{T}}}
\def\d{{\mathrm{d}}}
\def\e{{\mathrm{e}}}
\def\bftheta{\mbox{\boldmath$\theta$\unboldmath}}
\def\bfomega{\mbox{\boldmath$\omega$\unboldmath}}
\newcommand{\supp}{\operatorname{supp}}

\title{ \textbf{Ergodicity for  eventually continuous Markov--Feller semigroups on Polish spaces}}

 {\author{ \bf{Fuzhou Gong$^a$,\ Yong Liu$^b$,\ Yuan Liu$^a$,\   Ziyu Liu$^{b},$} \\ 
 \small{\em  {\bf$^a$} Academy of Mathematics and Systems Science, Chinese Academy of Sciences, Beijing 100190, China}\\
\small{\em  {\bf$^b$}   LMAM, School of Mathematical Science, Peking University, 100871, Beijing, China}\\
\small{\em Email:\;fzgong@amt.ac.cn(F. Gong);\ liuyong@math.pku.edu.cn (Yong Liu);}\\
\small{\em liuyuan@amss.ac.cn(Yuan Liu);\ liuziyu@math.pku.edu.cn(Z. Liu)}
}}

\date{\today}

\begin{document}
    \maketitle

    \makeatother
	
    \begin{abstract} 
    This paper investigates the ergodicity of Markov--Feller semigroups on Polish spaces, focusing on  very weak regularity conditions, particularly the Ces\`aro eventual continuity. First, it is showed that the Ces\`aro average of such semigroups weakly converges to an ergodic measure when starting from its support. This leads to a characterization of the relationship between Ces\`aro eventual continuity, Ces\`aro e-property, and weak-* mean ergodicity. Next, serval criteria are provided for the existence and uniqueness of invariant measures via Ces\`aro eventual continuity and lower bound conditions, establishing an equivalence relation between weak-* mean ergodicity and a lower bound condition. Additionally, some refined properties of ergodic decomposition are derived. Finally, the results are applied to several non-trivial examples, including iterated function systems, Hopf's turbulence model with random forces, and Lorenz system with noisy perturbations,  either with or without Ces\`aro eventual continuity.
 
    \vspace{5mm}


    \vspace{5mm}
    \noindent{\bf 2020 MR Subject Classifiction:} 60J25; 37A30
 
    \end{abstract}

    \tableofcontents
   
    \section{Introduction}\label{Sec 1}

     In this paper, we focus on the ergodicity of Markov--Feller semigroups on Polish spaces. Ergodicity is a fundamental property of Markov semigroups, providing a perspective for analyzing the long-term behavior of Markov processes. Its applications span various fields, including statistical physics, fluid mechanics, financial time series, and Markov chain Monte Carlo methods, among others. There is a large literature dedicated to both the theoretical and applied aspects of ergodicity. For example, the monographs \cite{ChenMF2005, MT2009, DMPS2018, Kulik2018}  can serve as some basic introductions into this area. Key themes in this research include the existence and uniqueness of invariant probability measures, ergodic decomposition, and the rate of convergence.

    The regularity of Markov semigroups,  with different presentations such as the strong Feller property, the asymptotic strong Feller property, and (Ces\`aro) equicontinuity (also referred to as the (Ces\`aro) e-property, see Definition \ref{e-property}), plays a crucial role in establishing criteria for ergodicity. These  regular properties have been applied in numerous models with various degree of success. In the present paper, we investigate  very weak regularity conditions for Markov semigroups, specifically eventual continuity and Ces\`aro eventual continuity (see Definition \ref{EvC}), along with their novel contributions to the associated ergodic behaviors.

     \subsection{Comments on early literature} \label{sec 1.2}

    In this subsection, we briefly review some well known regularity conditions of Markov--Feller semigroups and their roles in ergodicity. We begin by looking back at two classic theorems. The Krylov-Bogoliubov theorem (see \cite[Theorem 3.1.1]{DPZ1996}) is a fundamental result that guarantees the existence of an invariant measure for Markov--Feller semigroups, with the Feller property playing a crucial role. Doob's theorem (see \cite[Theorem 4.2.1]{DPZ1996}) asserts that if a Markov semigroup is $t_0$-regular, then there exists at most one invariant measure. Indeed, Doob's theorem leads to further conclusions: the Markov semigroup converges to the unique invariant measure in total variation distance. Further analysis shows that verifying $t_0$-regularity can be decomposed into two components: the {\bf strong Feller property} and {\bf irreducibility} (see \cite[Proposition 4.1.1]{DPZ1996}).

    From a modern perspective, the following points are worth noting:
    \begin{itemize}
        \item [$(\runum{1})$]  According to \cite[Section 4]{Kup2011} and \cite[Proposition 1.1]{DongYL2016}, the strong Feller property can be used to establish the EMDS property,  i.e. a property that {\em ergodic measures are disjointly supported} for a Markov semigroup, see \cite[p80]{Z2005} and \cite[p47]{Z2014}.
        \item[$(\runum{2})$] Irreducibility ensures that for any invariant measure $\mu$, there exists a state $x\in \supp\mu$. Then, by the EMDS property, there can be at most one invariant probability measure.    
    \end{itemize} 

    However, the strong Feller property is often less useful in infinite-dimensional settings, such as SPDEs with degenerate noise, which may fail to satisfy this property. In 2006, Hairer and Mattingly introduced a new regularity condition for Markov semigroups: the asymptotic strong Feller property (see \cite[Definition 3.8]{HM2006}). They showed that this property guarantees the EMDS property (see \cite[Theorem 3.16]{HM2006}; see also \cite[Theorem 1.2]{HM2011}), and further proved that the support of any invariant measure contains the common element zero, which implies unique ergodicity (see \cite[Corollary 4.2]{HM2006}) for the 2D Navier--Stokes equation with highly degenerate stochastic noise.  They also introduced the concept of weak topological irreducibility (see \cite[Definition 1.3]{HM2011}) and  then demonstrated the unique ergodicity of asymptotic strong Feller semigroups (see \cite[Corollary 1.4]{HM2011}).

    \vspace{3mm}

    An alternative approach to verifying the existence and ergodicity of invariant measures for Markov processes  of infinite dimensions is through the e-property. In 2006, Szarek \cite{S2006} and Lasota--Szarek \cite{LS2006} introduced the e-property and lower bound technique in nonlocally compact spaces, showing that Feller processes satisfying the e-property and a lower bound condition admit an invariant probability measure (see \cite[Proposition 2.1]{S2006} and \cite[Theorem 3.1]{LS2006}). In \cite{KPS2010}, the authors provided a criterion, under a stronger lower bound condition, for the unique ergodicity of Feller processes with the e-property. Furthermore, in 2012, Kapica, Szarek and Sleczka explicitly stated that the e-property implies the EMDS property (see \cite[Theorem 1]{KSS2012}) and that any weakly irreducible Feller process with the e-property admits at most one invariant probability measure (see \cite[Theorem 2]{KSS2012}).

    The concept of equicontinuous Markov operators on compact state spaces can be traced back to the work of Jamison \cite{J1964, J1965} and Rosenblatt \cite{R1964}. The ergodic theory of equicontinuous Markov--Feller operators on locally compact state spaces is systematically described in \cite{Z2005, Z2014}. The e-property can also be viewed as a generalization of the non-expansiveness of Markov operators, particularly in relation to the Fortet--Mourier distance (FM distance) between finite measures. Lasota and Yorke \cite{LY1994} introduced the concept of non-expansiveness of Markov operators with respect to the FM distance and employed a lower bound technique to establish the existence of invariant probability measures and the asymptotic stability of Markov semigroups on locally compact and $\sigma$-compact state spaces. Szarek extended these results to Polish spaces in \cite{S2000, S2003, S2003+}.

    \vspace{3mm}

    In general, the e-property serves a two-fold purpose: verifying both the existence and ergodicity of an invariant probability measure. Moreover, it provides simpler criteria compared to the strong Feller or asymptotic strong Feller properties, such as the gradient estimate of a Markov semigroup (see \cite[Section 1]{GL2015}). As the lower bound conditions introduced by Szarek \cite{S2006, LS2006, KPS2010} are framed as probabilistic estimates on open sets, which are not precompact in nonlocally compact spaces, they are easier to verify for certain stochastic systems of infinite dimensions. This makes the e-property a valuable tool in ergodic theory. There are rather complete theoretical results on the ergodicity of equicontinuous Markov processes (see \cite{S2006, LS2006, KPS2010, W2010, SW2012, WH2010, WH2011, WH2011-0, CH2014, C2017, KW2021, KW2024} and references therein), as well as numerous applications in practical models (e.g., \cite{LS2006, SSU2010, BKS2014, CH2014, Zhai2024, Zhai2024+, PZZ2024}).

    \subsection{Obstructions and motivations} 

    Though these early works made a great deal of progress, it should be noted that both the asymptotic strong Feller property and the e-property require some form of  {\bf uniform estimation} in a neighborhood of  initial states. This can   be delicate and stressful in specific cases, and even if these properties  are essentially true, the verification is usually a heavy job. For instance, in \cite{GHM2017, KS2018}, the authors  provided examples of highly nonlinear SPDEs for which the  asymptotic strong Feller property and e-property are difficult to confirm. Additionally, some non-equicontinuous Markov semigroups  were presented in \cite[Example 2.1]{HSZ2017}, \cite[Section 4]{GLLL2024}, and in Examples \ref{Ex 1}, \ref{Ex 2}, and \ref{Ex 3} of the present paper, where  validation of these properties  posed significant challenges.

    To overcome these difficulties  or streamline nowadays arguments, a strictly weaker regularity condition, namely the concept of  {\bf asymptotic equicontinuity} (introduced by Jaroszewska \cite{J2013}) or  {\bf eventual continuity} (introduced by Gong and Liu \cite{GL2015}), was proposed to prove ergodicity. These two notions were formulated almost simultaneously and are mathematically equivalent. In this paper, we adopt the notion of eventual continuity, which  involves some possible sensitivity of non-equicontinuous Markov semigroups to initial  datum.

    In \cite{J2013}, Jaroszewska provided strict sufficient conditions for asymptotic stability for eventually continuous Markov semigroups.  The first and third authors  of the present paper \cite{GL2015}  derived sufficient and almost necessary criteria for the existence of invariant probability measures based on eventual continuity and  the same lower bound condition as \cite{S2006} without any uniform restriction on spatial variable and time parameter, and they characterized asymptotic stability under weaker assumptions. Both studies focused on Markov semigroups with a discrete time parameter. In \cite{LL2024}, the second and fourth authors of the present paper explored the relationship between the e-property and eventual continuity. Furthermore, in \cite{GLLL2024}, we demonstrated that asymptotic stability is equivalent to eventual continuity and a uniformly lower bound condition (see \cite[Theorem 1]{GLLL2024}). These results apply to both discrete and continuous-time Markov semigroups.

    We also mention that the asymptotic strong Feller property and the e-property do not imply each other \cite{J2013+}. Moreover, the asymptotic strong Feller property does not imply eventual continuity and stochastic continuity\cite{LL2024}.
    
     Now, the main motivations behind our analysis are summarized as the following:
     
    \begin{itemize} 
    \item[$(\runum{1})$] There exist Markov--Feller processes that are (Ces\`aro) eventually continuous, but do not satisfy the strong Feller property,  asymptotic strong Feller property, or e-property. For instance, examples of non-equicontinuous Markov semigroups are provided in \cite[Example 2.1]{HSZ2017}, \cite[Section 4]{GLLL2024}, as well as in Examples \ref{Ex 1}, \ref{Ex 2}, and \ref{Ex 3} of this manuscript.
    
    \item[$(\runum{2})$] The (Ces\`aro)  e-property, when combined with other  assumptions such as lower bound condition and  weak irreducibility, serves as a powerful sufficient condition for verifying the ergodicity of Markov--Feller semigroups. In contrast, (Ces\`aro) eventual continuity is not only strictly weaker but also necessary for ergodic process starting from its ergodic component. Moreover respectively, eventual continuity is necessary for global asymptotic stability (see \cite{J2013,LL2024}), while Ces\`aro eventual continuity is  necessary for weak-* mean ergodicity (see Definition~\ref{Def 3} and Proposition \ref{Prop 1}). From a theoretical standpoint,  it deserves a careful consideration to push forward the ergodic theory of Markov--Feller semigroups under the framework of (Ces\`aro) eventual continuity.
     
    \item[$(\runum{3})$] From a practical perspective, several established methods, e.g. gradient estimates, coupling techniques, and functional inequalities (see \cite{Chenmf1989, DPZ1996, HM2006, Wangfy2014, KS2018} and references therein), have been developed to verify the asymptotic strong Feller property or equicontinuity. Given that (Ces\`aro) eventual continuity is a relatively weaker regularity condition, we aim to provide simpler or more easily verifiable criteria for the existence and ergodicity of invariant probability measures via (Ces\`aro) eventual continuity, which can be applied to more complex  systems.
    \end{itemize}

    \vspace{3mm}
    
    \subsection{Our main results}  

    To describe our results, let $\{P_t\}_{t\ge 0}$ be a Markov--Feller semigroup on Polish space $\mathcal{X}$,  $Q_tf(x):=\frac{1}{t}\int_{0}^{t}P_sf(x)\d s$, $Q_t\mu:=\frac{1}{t}\int_{0}^{t}P_s\mu\, \d s$ for $t>0$ and $\mu \in \mathcal{P}(\mathcal{X})$, the collection of probability measure on $\mathcal{X}$.  Denote $Q_t(x, \cdot)=Q_t\delta_x$.  For convenience, when we refer to invariant measures and ergodic measure in this paper, we mean invariant probability measures and  ergodic probability measure, respectively.

    In the context of ergodic theory, the ergodicity of $\{P_t\}_{t\ge 0}$ can be determined by the asymptotic behaviours of  $\{Q_t\mu\}_{t\ge 0}$, therefore it is natural to analyze the set 
    \begin{equation}\label{Def T}
	\mathcal{T}:=\big\{x\in \mathcal{X}:\{Q_t(x,\cdot)\}_{t\geq 0}\text{ is tight}\big\},
    \end{equation}
    through Ces\`aro eventual continuity.   

    \begin{theorem} \label{Main Thm1}\textsl{(Proposition \ref{Prop T2*}, Theorem \ref{Prop T3*} and Theorem \ref{Thm 1} below)} \ Let $\{P_t\}_{t\geq 0}$ be Ces\`aro eventually continuous on $\mathcal{X}$. 
 
    \begin{itemize}
    \item[(\runum{1})] If $\{P_t\}_{t\geq 0}$ admits an invariant measure $\mu$, then  $\supp\mu\subset\mathcal{T}.$
    \item[(\runum{2})]  If $\{P_t\}_{t\geq 0}$ admits an ergodic invariant measure $\mu$, then for any $x\in\supp\mu$, $\{Q_t(x,\cdot)\}_{t\geq 0}$ weakly converges to $\mu$  as $t\rightarrow\infty$.
 
    \item[(\runum{3})] If $\{P_t\}_{t\geq 0}$ admits an ergodic invariant measure $\mu$ and  is stochastically continuous $\mathcal{X}$, then $\{P_t\}_{t\geq 0}$ satisfies the Ces\`aro e-property on  the interior of $\supp\mu$ in $\mathcal{X}$.
    \end{itemize} 
 
    \end{theorem}
 
    The part {\it (\runum{3}) } of Theorem \ref{Main Thm1} implies an interesting corollary that if $\{P_t\}_{t\ge 0}$ is restricted on ${\rm supp}\, \mu$, then the following three statements are equivalent: $(1)$ $\{P_t\}_{t\ge 0}$ is weakly-* mean ergodic on ${\rm supp}\, \mu$; $(2)$ $\{P_t\}_{t\ge 0}$ satisfies Ces\`aro e-property on ${\rm supp}\, \mu$;  $(3)$ $\{P_t\}_{t\ge 0}$ is Ces\`aro eventually continuous on ${\rm supp}\, \mu$ (see Proposition \ref{Prop QEvC=Q-e}).

     \vspace{3mm}
 
     Next, we provide serval criteria of the existence and uniqueness of invariant measures of Ces\`aro eventually continuous Markov--Feller semigroups via lower bound conditions. 
     \begin{theorem} \label{Main Thm2} \textsl{(Proposition \ref{Prop Exist*}, Proposition \ref{Thm 2} below)}  \ Let $\{P_t\}_{t\geq 0}$ be Ces\`aro eventually continuous on $\mathcal{X}$.
     \begin{itemize}
     \item[(\runum{1})]   If there exists some $z\in \mathcal{X}$ such that for any $\epsilon>0$, 
            \begin{equation}\label{eq 3.4}\tag{$\mathcal{C}_1$}
		\limsup\limits_{t\rightarrow\infty}Q_t(z,B(z,\epsilon))>0,
		\end{equation}
    then $\mathcal{T}\neq\emptyset$, i.e.,  $\{P_t\}_{t\geq 0}$ admits an invariant measure.
	
    \item[(\runum{2})]   If there exists $z\in \mathcal{X}$ such that for any $x\in \mathcal{X}$ and $\epsilon>0$,
    \begin{equation}\label{eq 3.5}\tag{$\mathcal{C}_2$}
    \limsup\limits_{t\rightarrow\infty}Q_t(x,B(z,\epsilon))>0,
    \end{equation}
    then $\{P_t\}_{t\geq 0}$ admits a unique invariant measure $\mu_*$.  Moreover, $z\in\supp\mu_*$ and $\{Q_t\nu\}_{t\geq 0}$ weakly converges to $\mu_*$  as $t\rightarrow\infty$ for any $\nu\in\mathcal{P}(\mathcal{X})$ with $\supp\nu\subset\mathcal{T}.$ 
		
    \end{itemize}
    \end{theorem}      

    Let us mention that although a Markov--Feller semigroup with the Ces\`aro e-property satisfies properties similar to those in Theorem \ref{Main Thm1} $(\runum{1}), (\runum{2})$ and Theorem \ref{Main Thm2}, our results extend to non-equicontinuous models and provide a theoretical framework for deriving more easily verifiable ergodicity criteria. Moreover, Theorem \ref{Main Thm1} $(\runum{3}),$ and its corollary characterize the relationship between Ces\`aro eventual continuity and the Ces\`aro e-property for an ergodic measure.

    We also adopt a proof by contradiction to establish these results. However, the main challenge we encounter is the absence of uniform convergence in the Ces\`aro e-property, which would typically allow an application of the Arzel\`a--Ascoli theorem to construct a compact set. This compactness would enable the construction of a bounded Lipschitz function on the compact set, which could then be extended to $\mathcal{X}$ to separate two distinct invariant measures. This step is crucial in the proof by contradiction, as it leads to a violation of the Ces\`aro e-property. To overcome this difficulty, we directly construct a bounded Lipschitz function to separate invariant measures on $\mathcal{X}$ using the completeness and separability of $\mathcal{X}$ and the pointwise convergence guaranteed by Ces\`aro eventual continuity.

    \vspace{3mm}

   The third result implies that the weak-* mean ergodicity is equivalent to Ces\`aro eventual continuity and a uniform lower bound condition.      
   
   \begin{theorem}\label{Main Thm3}  \textsl{(Theorem  \ref{Thm 3} below)} The following two statements are equivalent: 
		\begin{itemize}
		    \item [(\runum{1})] $\{P_t\}_{t\geq 0}$ is weakly-* mean ergodic with unique invariant measure $\mu$.
		     \item[(\runum{2})]  $\{P_t\}_{t\geq 0}$ is Ces\`aro eventually continuous on $\mathcal{X}$ and there exists some $z\in\mathcal{X}$ such that for any $\epsilon>0$, 
	 \begin{equation}\label{eq 3.7}\tag{$\mathcal{C}_3$}
		\inf\limits_{x\in\mathcal{X}}\limsup\limits_{t\rightarrow\infty}Q_t(x,B(z,\epsilon))>0.
		\end{equation}
		
			\end{itemize}
   
   \end{theorem}
   
    Theorem \ref{Main Thm3} provides a necessary and sufficient condition for weak-* mean ergodicity through Ces\`aro eventual continuity and a lower bound condition \eqref{eq 3.7}, both of which are relatively easy to verify in practical examples. The key observation in the proof is simple yet effective: the long-term behavior of a Markov semigroup is primarily determined by its behavior over sufficiently large times. In fact, Ces\`aro eventual continuity is sufficient to guarantee this observation. In addition, eventual continuity itself plays a crucial role in connecting asymptotic stability, and it has been extensively studied in \cite{GL2015, GLLL2024}. The relationship between eventual continuity and equicontinuity is explored in \cite{LL2024}.

    \vspace{3mm}

    We next focus on the ergodicity of iterated function systems (IFS), a classic yet still active research topic. Various IFS are closely related to fractals, learning models, recursive estimations, gene expression, and population dynamics \cite{Lasota1995,S2000+,CHW-S2020}.   They can even be viewed as solutions to certain stochastic differential equations perturbed by Poisson random measures \cite{Lasota1995,  H2002, LT2003}. For more on IFS, see \cite{HMS2005, HMS2006, BKS2014, C2017} and the references therein.

    Specifically, we present an IFS with jumps that is Ces\`aro eventually continuous but not Ces\`aro equicontinuous, and has internal randomness (see Example \ref{Ex 3}). The asymptotic stability and the Ces\`aro weak-* mean ergodicity of one of its variants are established (see Example \ref{Ex 5}). Notably, to the best of our knowledge, all known examples of (Ces\`aro) eventually continuous Markov--Feller semigroups are deterministic dynamical systems (see \cite{GL2015, HSZ2017}). This example is inspired by a class of IFS considered in \cite{BKS2014}.

    In addition, we explore a place-dependent IFS and its variant, for which proving their (Ces\`aro) e-property seems difficult (see Examples \ref{Ex 6} and \ref{Ex 7}). Nevertheless, we demonstrate their ergodicity by verifying their (Ces\`aro) eventual continuity (see Section \ref{Sec 4.3}).

    \vspace{3mm}

    Finally, it should be noted that while we have systematically discussed the ergodic properties of Ces\`aro eventually continuous semigroups on Polish spaces and obtained relatively complete results, the ergodicity of Markov semigroups remains far more complex than is currently understood. We present two Markov processes without Ces\`aro eventual continuity (see Examples~\ref{hopf model} and \ref{lorenz system}), which illustrate different ergodic behaviors of non-Ces\`aro eventually continuous semigroups.
             
    \subsection{Outline and notation}
 
    The paper is organized as follows. In Section \ref{Sec 2}, we introduce serval definitions and basic properties. In Section \ref{Sec 3}, we present the theoretical results on the ergodic properties of Ces\`aro eventually continuous Markov--Feller semigroups, containing complete statements of Theorems~\ref{Main Thm1}-\ref{Main Thm3}. In Section \ref{Sec 4}, we apply these results to investigate the ergodicity of certain iterated function systems with jumps. Section \ref{Sec 5} examines two Markov processes that lack Ces\`aro eventual continuity: the stochastic Hopf model and the stochastic Lorenz system. Finally, the complete proofs of the results discussed in the previous sections are collected in Section \ref{Sec 6}.

    \vspace{3mm}
    Let $\mathcal{X}$ be a Polish (i.e., a complete separable metric) space endowed with metric $\rho$ and Borel $\sigma$-algebra $\mathcal{B}(\mathcal{X})$.  We use the following notations:
    
    \vspace{1mm}
    $\begin{aligned}
        \mathcal{M}(\mathcal{X})&=\text{ the family of finite Borel measures on } \mathcal{X},\\
	\mathcal{P}(\mathcal{X})&=\text{ the family of probability measures on } \mathcal{X},\\
	B_b(\mathcal{X})&=\text{ the space of bounded, Borel real-valued functions defined on }\mathcal{X},\\
	&\quad\;\text { endowed with the supremum norm: } \|f\|_{\infty}=\text{sup}_{x\in \mathcal{X}}\,|f(x)|,f\in B_b(\mathcal{X}),\\
	C_b(\mathcal{X})&=\text{ the subspace of }  B_b(\mathcal{X}) \text{ consisting of bounded continuous functions},\\
	L_b(\mathcal{X})&=\text{ the subspace of } C_b(\mathcal{X}) \text{ consisting of bounded Lipschitz functions},\\
	B(x,r)&=\{y\in \mathcal{X}:\rho(x,y)<r\} \text{ for } x\in \mathcal{X} \text{ and }r>0,\\
	\partial A,\overline{A},\text{Int}_{\mathcal{X}} (A)&=  \text{ the boundary, closure, interior of } A \text{ in } \mathcal{X}, \text{ respectively}, \\
        A^{r}&=\{y\in \mathcal{X}:\rho(x,y)<r,\;x\in A\} \text{ for } A\in \mathcal{B}(\mathcal{X}) \text{ and }r>0,\\
	\supp\mu&=\{x\in \mathcal{X}:\mu(B(x,\epsilon))>0 \text{ for any } \epsilon>0\}, \text{ for }\mu\in\mathcal{M}(\mathcal{X}), \\               &\quad\text{ i.e. the topological  support of the measure } \mu,\\
        \N,\Z,\R,\R_+&= \text{natural numbers, integers, real numbers, nonnegative numbers, respectively}. 
	\end{aligned}$	
    
    \vspace{1mm}
    For brevity, we use the notation $\langle f,\mu\rangle=\int_{\mathcal{X}}f(x)\mu(\d x)$  for $f\in B_b(\mathcal{X})$ and $\mu\in\mathcal{M}(\mathcal{X})$.	For a random variable $\xi$, we denote its law by $\mathcal{D}(\xi)$.

    \section{Preliminaries}\label{Sec 2}

    In this section, we recall some necessary notions of the Markov semigroups, and further introduce serval useful properties for (Ces\`aro) eventually continous semigroups.

    \begin{definition} \label{Ma-oprt} 
    An operator $P:\mathcal{M}(\mathcal{X})\rightarrow\mathcal{M}(\mathcal{X})$ is a Markov operator on $\mathcal{X}$ if it satisfies the following conditions
    \begin{itemize}
        \item[$(\runum{1})$] {\rm(}Positive linearity{\rm)} $P(\lambda_1\mu_1+\lambda_2\mu_2)=\lambda_1P\mu_1+\lambda_2P\mu_2$ for $\lambda_1,\lambda_2\geq 0$, $\mu_1,\mu_2\in\mathcal{M}(\mathcal{X})$;
        \item[$(\runum{2})$]  {\rm(}Preservation of the norm{\rm)} $P\mu(\mathcal{X})=\mu(\mathcal{X})$ for $\mu\in\mathcal{M}(\mathcal{X})$.
       \end{itemize}
 \end{definition}       
 
\begin{definition}\label{Ma-regul}  
    A Markov operator $P$ is regular if there exists a linear operator $P^*:B_b(\mathcal{X})\rightarrow B_b(\mathcal{X})$ such that
    \begin{equation*}
        \langle f, P\mu\rangle = \langle P^*f, \mu\rangle\quad\forall\,f\in B_b(\mathcal{X}),\;\mu\in\mathcal{M}(\mathcal{X}).
    \end{equation*} 
\end{definition}    
    For the ease of notation, we simply rewrite $P^*$ as $P$.   A {\it Markov semigroup}  $\{P_t\}_{t\geq 0}$ on  $\mathcal{X}$ is a semigroup of Markov operators on $\mathcal{M}(\mathcal{X})$.    
\begin{definition}    \label{Ma-Feller}
 \begin{itemize}     
   \item[$(\runum{1})$]   A Markov operator $P$ is a Markov--Feller operator if it is regular and $P$ leaves $C_b(\mathcal{X})$ invariant, i.e., $P(C_b(\mathcal{X}))\subset C_b(\mathcal{X})$; 
   
   \item[$(\runum{2})$]   A Markov semigroup $\{P_t\}_{t\geq 0}$ is a Markov--Feller semigroup if $P_t$ is a Markov--Feller operator for any $t\geq 0$.  
   
   \item[$(\runum{3})$] A Markov semigroup $\{P_t\}_{t\geq 0}$ is  stochastically continuous if 
			\begin{equation*}
				\lim\limits_{t\rightarrow 0}P_tf(x)=f(x)\quad\forall\,f\in C_b(\mathcal{X}),\;x\in\mathcal{X}.
			\end{equation*}
 \end{itemize}
 
\end{definition}

   Recall that $\mu\in\mathcal{P}(\mathcal{X})$ is \emph{invariant} for the semigroup $\{P_t\}_{t\geq 0}$ if $P_t\mu=\mu$ for any $t\geq 0$.  For $\mu\in\mathcal{P}(\mathcal{X}),$ define 
     \begin{equation*}
        Q_t\mu:=\frac{1}{t}\int_{0}^{t}P_s\mu\, \d s\quad\text{for }t>0,
     \end{equation*}
     and denote $Q_t(x,\cdot)=Q_t\delta_x$.  

    \vspace{3mm}
  
    Before defining (Ces\`aro) eventually continuity, recall the notion of e-property as follows.

    \begin{definition}\label{e-property} A Markov semigroup $\{P_t\}_{t\geq 0}$ satisfies the e-property at $z\in \mathcal{X}$, if $\forall\,f\in L_b(\mathcal{X})$,
		\begin{equation*}
		\limsup\limits_{x\rightarrow z}\sup\limits_{t\geq 0}|P_tf(x)-P_tf(z)|=0.
		\end{equation*}
		That is, $\forall\,\epsilon>0$, $\exists\,\delta>0$ such that $\forall\,x\in B(z,\delta)$, 
		\begin{equation*}
			|P_tf(x)-P_tf(z)|\leq \epsilon,\quad\forall\, t\geq0.
		\end{equation*}		
	Similarly, 	$\{P_t\}_{t\geq 0}$ satisfies the Ces\`aro e-property at $z\in \mathcal{X}$, if $\forall\,f\in L_b(\mathcal{X})$,
		\begin{equation*}
		\limsup\limits_{x\rightarrow z}\sup\limits_{t\geq 0}|Q_tf(x)-Q_tf(z)|=0.
		\end{equation*}    
    \end{definition}

    \begin{definition}\label{EvC}
		A Markov semigroup 
		$\{P_t\}_{t\geq 0}$ is eventually continuous at $z\in \mathcal{X},$ if for any Lipschitz bounded function $f,$
		\begin{equation}\label{def1}
		\limsup\limits_{x\rightarrow z}\limsup\limits_{t\rightarrow\infty}|P_tf(x)-P_tf(z)|=0,
		\end{equation}
		that is, $\forall\,\epsilon>0$, $\exists\,\delta>0$ such that $\forall\,x\in B(z,\delta)$, there exists some $T_x\geq 0$ satisfying
		\begin{equation*}
			|P_tf(x)-P_tf(z)|\leq \epsilon\quad\forall\, t\geq T_x.
		\end{equation*}		
	$\{P_t\}_{t\geq 0}$ is Ces\`aro eventually continuous at $z\in \mathcal{X},$ if for any Lipschitz bounded function $f,$
		\begin{equation}\label{def2}
		\limsup\limits_{x\rightarrow z}\limsup\limits_{t\rightarrow\infty}|Q_tf(x)-Q_tf(z)|=0,
		\end{equation} 
	where $Q_tf(x)=\frac{1}{t}\int_{0}^{t}P_sf(x)\d s$ for $t>0$.
    \end{definition}

    By definition,  eventual continuity implies Ces\`aro eventual continuity.

    \begin{proposition}\label{Prop ECQEC}
    Let $\{P_t\}_{t\geq 0}$ be  eventually continuous at $z\in\mathcal{X}$. Then $\{P_t\}_{t\geq 0}$  is Ces\`aro eventually continuous at $z$.
    \end{proposition}
    \begin{proof}
        By \eqref{def1}, $\forall\,\epsilon>0$ and $f\in L_b(\mathcal{X})$, $\exists\,\delta>0$ such that $\forall\,x\in B(z,\delta)$, there exists some $T_x\geq 0$ satisfying
		\begin{equation*}
			|P_tf(x)-P_tf(z)|\leq \epsilon/2\quad\forall\, t\geq T_x.
		\end{equation*}	
        Therefore, choosing $T_x':=\max\{T_x,4T_x\|f\|_{\infty}/\epsilon\}$, then for any $t\geq T_x'$, one gets
        \begin{align*}
            |Q_tf(x)-Q_tf(z)|&\leq \frac{1}{t}\int_{0}^{T_x}|P_sf(x)-P_sf(z)|\d s+\frac{1}{t}\int_{T_x}^{t}|P_sf(x)-P_sf(z)|\d s\\
            &\leq 2T_x\|f\|_{\infty}t^{-1}+\epsilon/2\leq \epsilon,
        \end{align*}
        which completes the proof of Ces\`aro eventual continuity at $z$.
    \end{proof}

    Throughout this paper, we assume that $\{P_t\}_{t\geq 0}$ is a Markov--Feller semigroup. For the set $\mathcal{T}$ defined in \eqref{Def T}, we have  the following results. 

    \begin{lemma}\label{Prop T4}
    Let $\{P_t\}_{t\geq 0}$ be Ces\`aro eventually continuous on $\mathcal{X}$. Then $\mathcal{T}$ is a closed set in $\mathcal{X}$.
    \end{lemma}
		
    \begin{lemma}\label{Prop T1*}
    Let $\{P_t\}_{t\geq 0}$ be Ces\`aro eventually continuous on $\mathcal{X}$. Then for any $x\in\mathcal{T}$, the set $\{Q_t(x,\cdot)\}_{t\geq 0}$ weakly converges to an invariant measure $\varepsilon_x\in\mathcal{P}(\mathcal{X})$ as $t\rightarrow\infty$.  Moreover, for any  $\mu\in\mathcal{P}(\mathcal{X})$ with $\supp\mu\subset\mathcal{T}$, $\{Q_t\mu\}_{t\geq 0}$ weakly converges to $\int_{x\in \mathcal{T}}\varepsilon_x\mu(\d x)$, where the integral is in the Bochner sense as follows,
    \begin{equation*}
		    \int_{x\in \mathcal{T}}\varepsilon_x\mu(\d x)(A)=\int_{x\in \mathcal{T}}\varepsilon_x(A)\mu(\d x)\quad\text{for }A\in\mathcal{B}(\mathcal{X}).
    \end{equation*}
    \end{lemma} 
    
    The proofs of Lemma \ref{Prop T4} and Lemma \ref{Prop T1*} are collected in Section \ref{prf sec2}. Recall the following two types of stability of Markov semigroups. \par 
    
    \begin{definition}\label{Def 2}

    $\{P_t\}_{t\geq 0}$ is asymptotically stable if there exists a unique invariant measure $\mu_*\in\mathcal{P}(\mathcal{X})$ such that $\{P_t\mu\}_{t\geq 0} $ weakly converges to $\mu_*$  as $t\rightarrow\infty$ for any $\mu\in\mathcal{P}(\mathcal{X})$, i.e., $\langle f,P_t\mu\rangle\rightarrow\langle f,\mu_*\rangle$ as $t\rightarrow\infty$ for any $f\in C_b(\mathcal{X})$.
    \end{definition}

    \begin{definition}\label{Def 3}

    $\{P_t\}_{t\geq 0}$ is weakly-* mean ergodic if there exists a unique invariant measure $\mu_*\in\mathcal{P}(\mathcal{X})$ such that $\{Q_t\mu\}_{t\geq 0} $ weakly converges to $\mu_*$  as $t\rightarrow\infty$ for any $\mu\in\mathcal{P}(\mathcal{X})$. 
    \end{definition}

    \vspace{3mm}
    
    The following proposition indicates that the eventual continuity is a necessary condition for asymptotic stability, while the Ces\`ro eventual continuity is a necessary for weak-* mean ergodicity.

    \begin{proposition}\label{Prop 1}
    If $\{P_t\}_{t\geq 0}$ is asymptotically stable, then it is eventually continuous on $\mathcal{X}$. Similarly, if $\{P_t\}_{t\geq 0}$ is weakly-* mean ergodic, then it is Ces\`aro eventually continuous on $\mathcal{X}.$
    \end{proposition}
    
    \begin{proof}
    If $\{P_t\}_{t\geq 0}$ is asymptotically stable, then for any $f\in L_b(\mathcal{X})$ and $x,y\in\mathcal{X}$,
    \begin{equation*}
    \limsup_{t\rightarrow\infty}|P_tf(x)-P_tf(y)|\leq \limsup_{t\rightarrow\infty}|P_tf(x)-\langle f,\mu_*\rangle|+\limsup_{t\rightarrow\infty}|P_tf(y)-\langle f,\mu_*\rangle|=0.    
    \end{equation*}
    Thus, the eventual continuity at any $x\in\mathcal{X}$ follows from
    \begin{equation*}
    \lim_{y\rightarrow x}\limsup_{t\rightarrow\infty}|P_tf(x)-P_tf(y)|=0.    
    \end{equation*}
    
    Similarly, if $\{P_t\}_{t\geq 0}$ is weakly-* mean ergodic, then for any $f\in L_b(\mathcal{X})$ and $x,y\in\mathcal{X}$,
    \begin{equation*}
    \limsup_{t\rightarrow\infty}|Q_tf(x)-Q_tf(y)|\leq \limsup_{t\rightarrow\infty}|Q_tf(x)-\langle f,\mu_*\rangle|+\limsup_{t\rightarrow\infty}|Q_tf(y)-\langle f,\mu_*\rangle|=0,  
    \end{equation*}
    which ensures the Ces\`aro eventual continuity at $x\in\mathcal{X}$ as
    \begin{equation*}
    \lim_{y\rightarrow x}\limsup_{t\rightarrow\infty}|Q_tf(x)-Q_tf(y)|=0.    
    \end{equation*}      
    \end{proof}

    We end this section with an example of a continuous-time Markov process, for which the Ces\`aro e-property fails, but (Ces\`aro) eventual continuity holds.

    \begin{example}\label{Ex 1}
    Let $\mathcal{X}=\{0\}\cup\{\tfrac{1}{n}:n\geq 2\}\cup\{n:n\geq 2\}.$  Define a continuous-time Markov process $\{\Phi_t\}_{t\geq 0}$ with a transition rate matrix $\{q_{i,j}\}_{i,j\in \mathcal{X}}$, where $q_{n,n}=q_{n^{-1},n^{-1}}=-n,\;q_{n,0}=q_{n^{-1},n}=n$, for $n\geq 2,$ and $q_{i,j}=0$ otherwise. The transition probabilities $\{p_{i,j}(t)\}_{i,j\in \mathcal{X}}$ are then given by: $p_{0,0}(t)=1$ and for $n\geq 2$, 
    
    \begin{align*}
	&p_{n^{-1},n^{-1}}(t)=p_{n,n}(t)=\e^{-t/n},  &p_{n^{-1},n}(t)=\tfrac{t}{n}\e^{-t/n}, \\ &p_{n^{-1},0}(t)=1-p_{n^{-1},n^{-1}}(t)-p_{n^{-1},n}(t), &p_{n,0}(t)=1-\e^{-t/n}.
    \end{align*}
    
    It can be  checked that $\{P_t\}_{t\geq 0}$ generated by $\{\Phi_t\}_{t\geq 0}$ is a Markov--Feller semigroup. Furthermore, since $\Phi_t$ converges to $0$ almost surely as $t\rightarrow\infty$, $\{P_t\}_{t\geq 0}$ is asymptotically stable, hence $\{P_t\}_{t\geq 0}$ is eventually continuous on $\mathcal{X}$ by Proposition \ref{Prop 1}.  However, the Ces\`aro e-property fails at $0$. Indeed, let us take $f(x)=\min\{x,1\} \in C_b(\mathcal{X})$, then $P_tf(0)=0$,  $\forall\,t\geq 0$. On the other hand, 
    \begin{align*}
        Q_tf(\tfrac{1}{n})&=\frac{1}{t}\int_{0}^{t}P_sf(\tfrac{1}{n})\d s\\
        &=\frac{1}{t}\left[f(0)\int_{0}^{t}p_{\frac{1}{n},0}(s)\d s+f(\tfrac{1}{n})\int_{0}^{t}p_{\frac{1}{n},\frac{1}{n}}(s)\d s+f(n)\int_{0}^{t}p_{\frac{1}{n},n}(s)\d s\right]\\
        &\geq \frac{1}{t}\int_{0}^{t}p_{\frac{1}{n},n}(s)\d s=\frac{n}{t}-\e^{-t/n}\left(1+\frac{n}{t}\right).
    \end{align*}
    
    In particular, one has 
    \begin{equation*}
        Q_nf(\tfrac{1}{n})\geq 1-2\e^{-1}>0=Q_nf(0),
    \end{equation*}
    which conflicts with the Ces\`aro e-property.  
    \end{example}

    \section{Ergodicity for eventually continuous Markov--Feller semigroup}\label{Sec 3}

    This section is devoted to providing a comprehensive analysis of the ergodic properties of Ces\`aro eventually continuous semigroups on Polish spaces, deriving some comparatively complete results. Specifically, we in Section \ref{Sec 3.1new} explore  the ergodic properties related to the support of $\mathcal{T}$. Section \ref{Sec 3.2} establishes the existence and uniqueness of invariant measures via Ces\`aro eventual continuity and lower bounded conditions, and also the asymptotic behavior of $\{P_t\}_{t>0}$ starting outside $\mathcal{T}$. Section \ref{Sec 3.3} provides refined results on the ergodic decomposition of Ces\`aro eventually continuous Markov--Feller semigroups.

    \subsection{Support properties of invariant measure inside $\mathcal{T}$} \label{Sec 3.1new} 
    
    The starting point of our investigation of the Cea\`aro averages of Markov--Feller semigroups is form the following observation:  $\mathcal{T}\neq \emptyset$ ensures the existence of an invariant measure, but it does not imply the convergence of Ces\`aro averages. Indeed, we shall justify such convergence for the Ces\`aro eventually continuous semigroups.  Let us begin with the following property of the support of invariant measures.  
	
    \begin{proposition}\label{Prop T2*}
    Let $\{P_t\}_{t\geq 0}$ be Ces\`aro eventually continuous on $\mathcal{X}$ and admit an invariant measure $\mu.$ Then $\supp\mu\subset\mathcal{T}.$
    \end{proposition}
	
    Furthermore, if $\mu$ is an ergodic measure, then the following convergence property holds.
	
    \begin{theorem}\label{Prop T3*} 
    Let  $\{P_t\}_{t\geq 0}$ be  Ces\`aro eventually continuous on $\mathcal{X}$ and  $\{P_t\}_{t\geq 0}$ admit an ergodic invariant measure $\mu$. Then for any $x\in\supp\mu$, $\{Q_t(x,\cdot)\}_{t\geq 0}$ weakly converges to $\mu$  as $t\rightarrow\infty$.
    \end{theorem}

    We postpone the proof of Propositions \ref{Prop T2*} in Section \ref{prf sec3.1}, and present the proof of Theorem~\ref{Prop T3*} by Propositions \ref{Prop T2*}  and Lemma  \ref{Prop T1*}  as follows. 
	
    \begin{proof} 
    Given an ergodic measure $\mu$, for any $x\in\supp\mu$, by Lemma \ref{Prop T1*} and Proposition \ref{Prop T2*},  $\{Q_t(x,\cdot)\}_{t\geq 0}$ weakly converges to an invariant measure $\varepsilon_x$ as $t\rightarrow\infty$. Assume that, contrary to our theorem, there exists some $x\in\supp\mu$ such that $\varepsilon_x\neq\mu$. Then there exist $f\in L_b(\mathcal{X})$ and $\epsilon>0$ such that
    \begin{equation}\label{eq 3.0}
        |\langle f,\mu\rangle-\langle f,\varepsilon_x\rangle|>\epsilon.
    \end{equation}
    By \cite[Theorem 4.4]{WH2011}, there exists $A\subset \mathcal{X}$ with $\mu(A)=1$, such that 
    \begin{equation}\label{eq 3.1}
    \lim\limits_{t\rightarrow\infty}Q_tf(y)=\langle f,\mu\rangle,\quad\,\forall\,y\in A.		
    \end{equation}
        
    As $x\in\supp\mu,$ there exists a sequence $y_n\in A$ such that $y_n\rightarrow x$ as $n\rightarrow\infty$. Then by the Ces\`aro eventual continuity of $\{P_t\}_{t\geq 0}$, there exists some $N\in\mathbb{N}$ sufficiently large such that for $n\geq N$,
    
    \begin{equation}\label{eq 3.2}
    \limsup\limits_{t\rightarrow\infty}|Q_tf(x)-Q_tf(y_n)|\leq\epsilon/2.
    \end{equation}\par 
    On the other hand,  $\{Q_t(x,\cdot)\}_{t\geq 0}$ weakly converges to $\varepsilon_x$ as $t\rightarrow\infty$. Consequently,
    
    \begin{equation}\label{eq 3.3}
	\lim\limits_{t\rightarrow\infty}Q_tf(x)=\lim\limits_{t\rightarrow\infty}\langle f,Q_t(x,\cdot)\rangle=\langle f,\varepsilon_x\rangle.
    \end{equation}
    Collecting (\ref{eq 3.1})-(\ref{eq 3.3}), one has $|\langle f,\mu\rangle-\langle f,\varepsilon_x\rangle|\leq\epsilon/2$. This contradicts \eqref{eq 3.0}.  
    \end{proof}
    
    \begin{remark} 
    These properties are also obtained for continuous-time e-processes. In \cite[Lemma 1]{KSS2012}, it is shown that $\forall\,x\in\mathcal{T},$ $\{Q_t(x,\cdot)\}_{t\geq 0}$ weakly converges to some invariant measures as $t\rightarrow\infty$. While \cite[Proposition 1]{KPS2010} proves that each invariant measure is supported inside $\mathcal{T}$. Furthermore,  \cite[Theorem 5.13]{WH2010} implies that  $\{Q_t(x,\cdot)\}_{t\geq 0}$ starting inside the support of some ergodic measure will converge to this measure. 
    \end{remark}

    Now, we give  two  interesting applications of Theorem \ref{Prop T3*} as follows.

    \begin{theorem}\label{Thm 1}
    Let  $\{P_t\}_{t\geq 0}$ be  stochastically continuous and Ces\`aro eventually continuous on $\mathcal{X}$. Assume that $\{P_t\}_{t\geq 0}$ admits an  ergodic measure $\mu$ such that ${\rm Int}_{\mathcal{X}}(\supp\mu)\neq\emptyset$. Then $\{P_t\}_{t\geq 0}$ satisfies the Ces\`aro e-property on ${\rm Int}_{\mathcal{X}}(\supp\mu)$.
    \end{theorem}
    The proof of Theorem \ref{Thm 1} is placed in Section \ref{prf sec3.1}.

    \vspace{3mm}
    Given an ergodic measure $\mu$, denote $\mathcal{X}_{\mu}:=\supp\mu$. As $\mathcal{X}_{\mu}$ is a closed set in $\mathcal{X}$, the subspace $(\mathcal{X}_{\mu},\rho)$ is still a  Polish space. Moreover, by \cite[Lemma 4.1]{GL2015}, $\mathcal{X}_{\mu}$ is an invariant set for $\{P_t\}_{t\geq 0}$ in the sense that $P_t(x,\mathcal{X}_\mu)=1$ for any $x\in\mathcal{X}_\mu$ and $t\geq 0$. In particular, it allows one to consider the semigroup  $\{P_t\}_{t\geq 0}$ to be restricted on the subspace $(\mathcal{X}_{\mu},\rho)$. Then, we obtain an immediate corollary of Theorem \ref{Thm 1} as follows.

    \begin{proposition}\label{Prop QEvC=Q-e}  
     Let $\{P_t\}_{t\geq 0}$  be stochastically continuous and admit an ergodic measure $\mu$. Then the following there statements are equivalent:
    \begin{itemize}
        \item [$(\runum{1})$] $\{P_t\}_{t\geq 0}$ is weakly-* mean ergodic on $\mathcal{X}_\mu$, i.e., $\{Q_t(x,\cdot)\}_{t\geq 0}$ weakly converges  to $\mu$ as $t\rightarrow\infty$ for any $x\in\mathcal{X}_\mu$.
        \item [$(\runum{2})$] $\{P_t\}_{t\geq 0}$ satisfies the Ces\`aro e-property on $\mathcal{X}_\mu$.
        \item [$(\runum{3})$] $\{P_t\}_{t\geq 0}$ is Ces\`aro eventually continuous on $\mathcal{X}_\mu$.
    \end{itemize}
    \end{proposition}

    \begin{proof}
    The direction $(\runum{2})\Rightarrow(\runum{3})$ is trivial,  $(\runum{3})\Rightarrow(\runum{1})$ follows straightforwardly from Theorem~\ref{Prop T3*}, and $(\runum{2})\Rightarrow(\runum{1})$ is ensured by Theorem \ref{Thm 1}. 
    \end{proof}
    
    \begin{remark} 
    \cite[Theorem 3.3]{LL2024} shows that if a stochastic continuous and eventually continuous Markov--Feller semigroup $\{P_t\}_{t\ge 0}$ admits an ergodic measure $\mu$, and  $\text{Int}_{\mathcal{X}}({\supp}\;\mu)\neq \emptyset$, then $\{P_t\}_{t\ge 0}$ satisfies e-property on $\text{Int}_{\mathcal{X}}({\supp}\;\mu)$, and a direct corollary  is  that the e-property and eventual continuity are equivalent on the subspace $\supp\mu$  (see \cite[Corollary 3.4]{LL2024}).   
    \end{remark}
	
    The second application is that we get a direct and simple proof of the EMDS-property by Theorem \ref{Prop T3*}. In \cite[Proposition 1.10]{GL2015}, the EMDS-property of eventually continuous semigroups is proved  by means of  Birkhoff's ergodic theorem.  \par

    \begin{corollary}\label{Coro EMDS}
    Let  $\{P_t\}_{t\geq 0}$ be  Ces\`aro eventually continuous on $\mathcal{X}$. Then $\{P_t\}_{t\geq 0}$ satisfies the EMDS-property, i.e., for any two distinct ergodic measures $\mu$ and $\nu$  for $\{P_t\}_{t\geq 0}$, 
    \begin{equation*}
    \supp\mu\cap\supp\nu=\emptyset.
    \end{equation*}
    \end{corollary}

    As aforementioned in Section \ref{sec 1.2}, the uniqueness of invariant measures can be derived from some regularity property of the Markov--Feller semigroup combined with some kind of irreducibility. In fact, the various regularity properties are essential to establish the EMDS-property, i.e., the EMDS-property and irreducibility ensure the uniqueness of the invariant measure.

    \subsection{Lower bound condition, unique ergodicity and weak-* mean ergodicity}\label{Sec 3.2}

    In this subsection, we investigate the relation between the Ces\`aro eventually continuous condition and the existence-uniqueness of invariant measures. As illustrated in Section \ref{sec 1.2}, the Krylov-Bogoliubov Theorem shows that the existence of an invariant measure of the Markov--Feller semigroups is implied by the fact that $\mathcal{T}\neq\emptyset$. However, tightness is difficult to verify for infinite-dimensional systems. In 2006, Lasota and Szarek developed the lower bound technique to provide a tightness criterion for the Markov--Feller semigroup with e-property \cite{LS2006,S2006}. For eventually continuous semigroups with discrete time parameters, Gong and Liu also provide tightness criteria by means of some lower bound conditions \cite[Theorem 1.6, Theorem 1.8]{GL2015}.  We first extend the criterion of the existence of invariant measures for eventually continuous semigroups in \cite[Theorem 1.6]{GL2015} to the continuous-time setting.

    \begin{proposition}\label{Prop Exist*}
    Let $\{P_t\}_{t\geq 0}$ be Ces\`aro eventually continuous on $\mathcal{X}$. Assume that there exists some $z\in \mathcal{X}$ such that for any $\epsilon>0$, 
    \begin{equation}\label{eq 3.4*}\tag{$\mathcal{C}_1$}
	\limsup\limits_{t\rightarrow\infty}Q_t(z,B(z,\epsilon))>0.
    \end{equation}	
    Then $\mathcal{T}\neq\emptyset$, which implies the existence of invariant measures of  $\{P_t\}_{t\geq 0}$.
    \end{proposition}
    
    See the proof of Proposition \ref{Prop Exist*} in Section \ref{prf sec3.2}.
    
    \begin{remark} 
    It seems that Proposition \ref{Prop Exist*} is a plain extension of \cite[Theorem 1.6]{GL2015} to the continuous time setting. However, there are some technical difficulties such that the approach in \cite[Theorem 1.6]{GL2015} cannot be directly applied to the continuous-time setting.  Therefore, we prove it by a contradiction argument, invoking the methods similar in \cite{LS2006}.
    \end{remark}

    Combining the existence criterion of invariant measures and the formulation of $\mathcal{T},$ we are now able to provide a sufficient condition for the unique ergodicity.
	
    \begin{proposition}\label{Thm 2} 	Let  $\{P_t\}_{t\geq 0}$ be Ces\`aro eventually continuous on $\mathcal{X}$. Assume that there exists $z\in \mathcal{X}$ such that for any $x\in \mathcal{X}$ and $\epsilon>0$,
    \begin{equation}\label{eq 3.5*}\tag{$\mathcal{C}_2$}
    \limsup\limits_{t\rightarrow\infty}Q_t(x,B(z,\epsilon))>0.
    \end{equation}
    Then there exists a unique invariant measure $\mu_*.$ Moreover, $z\in\supp\mu_*$ and $\{Q_t\nu\}_{t\geq 0}$ weakly converges to $\mu_*$  as $t\rightarrow\infty$ for any $\nu\in\mathcal{P}(\mathcal{X})$ with $\supp\nu\subset\mathcal{T}.$
    \end{proposition}
	
    \begin{proof}  
    By Proposition \ref{Prop Exist*}, $\mathcal{T}\neq\emptyset$. Therefore,  there exists an ergodic measure $\mu$ for $\{P_t\}_{t\geq 0}$. By Theorem \ref{Prop T3*} and \eqref{eq 3.5}, we have  $\mu(B(z,\epsilon))>0$ for any $\epsilon>0$. Thus we obtain that $z\in\text{supp\;}\mu$ for any ergodic measure $\mu$. Hence by the EMDS-property, we conclude that the semigroup admits a unique  invariant  measure $\mu_*$. Moreover, by Lemma \ref{Prop T1*}, $\{Q_t\nu\}_{t\geq 0}$ weakly converges to $\mu_*$  as $t\rightarrow\infty$ for any $\nu\in\mathcal{P}(\mathcal{X})$ with $\supp\nu\subset\mathcal{T}$.
    \end{proof} 
	
    \begin{remark} 
    \cite[Theorem 1]{KPS2010} shows that if $\{P_t\}_{t\geq 0}$ has e-property and satisfies \eqref{eq 3.5}, then $\{P_t\}_{t\geq 0}$ admits a unique invariant measure $\mu_*$. Moreover, $\{Q_t\nu\}_{t\geq 0}$ weakly converges to $\mu_*$  as $t\rightarrow\infty$ for any $\nu\in\mathcal{P}(\mathcal{X})$ with $\supp\nu\subset\mathcal{T}$. The condition of Ces\`aro eventual continuity of $\{P_t\}_{t\geq 0}$ in Proposition \ref{Thm 2} is weaker than e-property in \cite[Theorem 1]{KPS2010}. 
    \end{remark}

    Note that \eqref{eq 3.5} is stronger than \eqref{eq 3.4} in the sense that it requires all points to have a positive probability of reaching neighborhoods of some fixed states. In fact, \eqref{eq 3.5} implies not only tightness, but also some type of irreducibility, thereby ensuring uniqueness. However, \eqref{eq 3.5} cannot guarantee that the set $\mathcal{T}$ is the whole space. Instead, we prove that $\mathcal{T}$ is the whole space under a stronger version of the lower bound condition \eqref{eq 3.7}, which correspondingly implies the weak-* mean ergodicity as follows.

    \begin{theorem}\label{Thm 3}
    The following two statements are equivalent: 
    \begin{itemize}
        \item [$(\runum{1})$] $\{P_t\}_{t\geq 0}$ is weakly-* mean ergodic with unique invariant measure $\mu$.
        \item[$(\runum{2})$] $\{P_t\}_{t\geq 0}$ is  Ces\`aro eventually continuous on $\mathcal{X}$ and there exists some $z\in\mathcal{X}$ such that for any $\epsilon>0$, 
        \begin{equation}\label{eq 3.7*}\tag{$\mathcal{C}_3$}
        \inf\limits_{x\in\mathcal{X}}\limsup\limits_{t\rightarrow\infty}Q_t(x,B(z,\epsilon))>0.
        \end{equation}
    \end{itemize}
    \end{theorem}
    
    We prove Theorem \ref{Thm 3} in Section \ref{prf sec3.2}.

    \begin{remark}
    The lower bound condition (\ref{eq 3.7}) seems to be a rather strong condition at first sight, while we will show that it is not able to ensure the convergence of $\{P_t\}_{t\geq0}$ even if we replace it by an even stronger condition as follows:
    \begin{equation}\label{eq lwp}
    \inf\limits_{x\in \mathcal{X}}\limsup\limits_{t\rightarrow\infty}P_t(x,B(z,\epsilon))>0,
    \end{equation}
    see the next example.
    \end{remark}
    
    \begin{example}\label{Ex 4}
    Let $\mathcal{X}=\T=\R/2\pi\Z$ and Markov semigroup $\{P_t\}_{t\geq 0}$ on $\mathcal{X}$ be given by $P_t\delta_x=\delta_{x+t\;(\text{mod}\;2\pi)},\;x\in \mathcal{X}$.  $\{P_t\}_{t\geq 0}$ is Feller, eventually continuous on $\mathcal{X}$ and (\ref{eq lwp}) holds for any $z\in\mathcal{X}$ and $\epsilon>0$. However $\{P_t\delta_x\}_{t\geq 0}$ does not converge for any $x\in \mathcal{X}.$ 
    \end{example}

     Next, it is natural to consider how to describe the long-time behavior of $\{P_t\}_{t\geq 0}$ or $\{Q_t\}_{t\geq 0}$ starting from the points outside $\mathcal{T}$. We have the following result, considering  the sweep property introduced in \cite{SSU2010}. In the context of  e-processes, the sweep property is first given in \cite[Proposition 3]{SSU2010} and then extended under a weaker lower bound condition in \cite[Proposition 3.9]{SW2012}.

    \begin{proposition}\label{Prop Sweep} 
    Let $\{P_t\}_{t\geq 0}$ be Ces\`aro eventually continuous on $\mathcal{X}$. Assume that there exists $z\in \mathcal{X}$ such that for any $x\in \mathcal{X}$ and $\epsilon>0$ such that \eqref{eq 3.5} holds. Then $\{P_t\}_{t\geq 0}$ is sweeping from compact sets disjoint from $\mathcal{T}$ in the following sense: for any compact set $K\subset\mathcal{X}$ such that $K\cap\mathcal{T}=\emptyset$, then 
    \begin{equation*}
    \lim\limits_{t\rightarrow\infty}P_t\mu(K)=0\quad\forall\,\mu\in\mathcal{P}(\mathcal{X}).
    \end{equation*}
  
    \end{proposition}

    The proof of Proposition \ref{Prop Sweep} is given in Section \ref{prf sec3.2}.

    \vspace{3mm}
    In addition to Ces\`aro averages, the eventually continuous semigroup $\{P_t\}_{t\geq 0}$ itself also converges under some appropriate assumptions. Recall that some connection between eventual continuity and asymptotic stability has been presented in Proposition \ref{Prop 1}. A sufficient condition for asymptotic stability is also given in \cite[Theorem 1.15]{GL2015}, which is relatively difficult to verify and thus serves more as a theoretical tool for studying eventual continuity. 
    
    In what follows, we formulate an asymptotic stability criterion for the eventually continuous semigroups, which can be applied to a jump process connected to the place-dependent iterated function systems provided in the next section.    More precisely, we prove that the asymptotic stability can be guaranteed under the eventual continuity and a lower bound condition \eqref{eq 4.1} as follows.

    \begin{theorem}\label{Thm 4} 
    The following three statements are equivalent:
    \begin{itemize}
        \item[$(\runum{1})$] $\{P_t\}_{t\ge 0}$ is asymptotically stable with unique invariant measure $\mu$.
        \item[$(\runum{2})$] $\{P_t\}_{t\geq 0}$ is eventually continuous on $\mathcal{X}$, and there exists $z\in\mathcal{X}$ such that for any $\epsilon>0$,  
            \begin{equation}\label{eq 4.1}\tag{$\mathcal{C}_4$}
		  	   \inf\limits_{x\in \mathcal{X}}\liminf\limits_{t\rightarrow\infty}P_t(x,B(z,\epsilon))>0.
		\end{equation}
        \item[$(\runum{3})$]  There exists $z\in\mathcal{X}$ such that $\{P_t\}_{t\geq 0}$ is eventually continuous at $z$ and for any $\epsilon>0$, \eqref{eq 4.1} holds. 
    \end{itemize} 		
    \end{theorem}
    
    \begin{remark}
    Theorem \ref{Thm 4} slightly improves  \cite[Theorem 1]{GLLL2024}. Actually, Condition $(\runum{3})$ is straightforward to apply, as it only requires verification of the eventual continuity and relation \eqref{eq 4.1} holding at a single point $z$.  
    \end{remark}
 
    \begin{proof}[Proof of Theorem \ref{Thm 4}] 
    The implication that $(\runum{1})\Rightarrow(\runum{2})$ follows from \cite[Theorem 1]{GLLL2024}. $(\runum{2})\Rightarrow(\runum{3})$ is trivial. It remains to show that $(\runum{3})\Rightarrow(\runum{1})$. The existence of an unique invariant measure follows directly from Proposition \ref{Thm 2}. Thus it suffices to show that 
    \begin{equation*}
    \lim\limits_{t\rightarrow\infty}|P_tf(x)-P_tf(y)|=0,\quad\forall\,x,y\in \mathcal{X},\;f\in L_b(\mathcal{X}).
    \end{equation*}
        
    Assume, contrary to our claim, that there exist  some  $x_1,x_2\in \mathcal{X}$, $f\in L_b(\mathcal{X})$ and $\epsilon>0$ such that $\limsup\limits_{t\rightarrow\infty}|P_tf(x_1)-P_tf(x_2)|\geq 3\epsilon$. As $\{P_t\}_{t\geq 0}$ is eventually continuous at   $z$, there exists $\delta>0$ such that
    
    \begin{equation*}
    \limsup\limits_{t\rightarrow\infty}|P_tf(x)-P_tf(z)|<\epsilon/2,\quad\forall\,x\in B(z,\delta). 
    \end{equation*}
    For such $\delta>0$, condition (\ref{eq 4.1}) gives some $\alpha \in (0,\frac12)$ such that $\liminf\limits_{t\rightarrow\infty}P_t(x,B(z,\delta))>\alpha$ for any $x\in\mathcal{X}$.	Then Fatou's lemma gives, for any $\nu\in\mathcal{P}(\mathcal{X})$ with $\text{supp }\nu\subset B(z,\delta),$
    
    \begin{equation*}
    \liminf\limits_{t\rightarrow\infty}P_t\nu(B(z,\delta))\geq\int_{\mathcal{X}}\liminf\limits_{t\rightarrow\infty}P_t(y,B(z,\delta))\nu(\d y)>\alpha.
    \end{equation*}
    
    Let $k\geq 1$ be such that $2(1-\alpha)^k\|f\|_{\infty}<\epsilon.$ By induction we will define four sequences of measures $\{\nu_i^{x_1}\}_{i=1}^k,\{\mu_i^{x_1}\}_{i=1}^k,\{\nu_i^{x_2}\}_{i=1}^k,\{\mu_i^{x_2}\}_{i=1}^k,$ and a sequence of positive numbers $\{t_{i}\}_{i=1}^k$ in the following way: let $t_1>0$ be such that $P_{t}(x_j,B(z,\delta))>\alpha$ for any $t\geq t_1$, $j=1,2$. Set     
    \begin{center}
    $\nu_1^{x_j}(\cdot) = \dfrac{P_{t_1}\delta_{x_j}(\cdot \cap B(z,\delta))}{P_{t_1}\delta_{x_j}(B(z,\delta))},\quad$
    $\mu_1^{x_j}(\cdot) = \dfrac{1}{1-\alpha}(P_{t_1}\delta_{x_j}(\cdot)-\alpha\nu_1^{x_j}(\cdot)),\quad j=1,2.$
    \end{center}
		
    Assume that we have done it for $i = 1,\dots , l,$ for some $l < k.$ Now let $t_{l+1}\geq t_l$ be such that $P_{t}\mu_l^{x_j}(B(z,\delta))>\alpha$ for any $t\geq t_{l+1}$, $j=1,2$. Set 
    \begin{center}
    $\nu_{l+1}^{x_j}(\cdot) = \dfrac{P_{t_{l+1}}\mu_l^{x_j}(\cdot \cap B(z,\delta))}{	P_{t_{l+1}}\mu_l^{x_j}(B(z,\delta))},\quad$ $\mu_{l+1}^{x_j}(\cdot) = \dfrac{1}{1-\alpha}(P_{t_{l+1}}\mu_l^{x_j}(\cdot)-\alpha\nu_{l+1}^{x_j}(\cdot)),\quad j=1,2.$
    \end{center}
    Then it follows that
    \begin{align*}
    P_{t_1+\dots+t_k}\delta_{x_j}(\cdot)&=\alpha P_{t_2+\dots+t_k}\nu_1^{x_j}(\cdot)+\alpha(1-\alpha) P_{t_3+\dots+t_k}\nu_2^{x_j}(\cdot)+ \dots \\
    &\quad +\alpha(1-\alpha)^{k-1}\nu_k^{x_j}(\cdot)+(1-\alpha)^k \mu_k^{x_j}(\cdot),
    \end{align*}
    with $\text{supp }{{\nu}_i^{x_j}}\subset B(z,\delta),\,i=1,\dots,k,\,j=1,2$. Thus  from the Fatou's lemma and eventual continuity,  we have
    \begin{equation*}
    \begin{aligned}
    \limsup\limits_{t\rightarrow\infty}|\langle P_{t}f, \nu_i^{x_1} \rangle-\langle P_{t}f, \nu_i^{x_2} \rangle| \leq \limsup\limits_{t\rightarrow\infty}|\langle P_{t}f-P_{t}f(z), \nu_i^{x_1} \rangle| +|\langle P_{t}f-P_{t}f(z), \nu_i^{x_2} \rangle|\leq\epsilon.
    \end{aligned}
    \end{equation*}
    
    Finally, using the measure decomposition, we get
    \begin{equation*}
    \begin{aligned}
    3\epsilon&\leq\limsup\limits_{t\rightarrow\infty}|P_tf(x_1)-P_tf(x_2)|\\
    &=\limsup\limits_{t \to \infty}|\langle P_tf,P_{t_1+\cdots+t_k} \delta_{x_1}\rangle-\langle P_tf,P_{t_1+\cdots+t_k} \delta_{x_2}\rangle| \\
    &\leq \alpha\limsup\limits_{t \to \infty}|\langle P_{t}f, \nu_1^{x_1} \rangle-\langle P_{t}f, \nu_1^{x_2} \rangle|+\cdots+\alpha(1-\alpha)^{k-1}\limsup\limits_{t \to \infty}|\langle P_{t}f, \nu_k^{x_1}\rangle-\langle P_{t}f, \nu_k^{x_2} \rangle|+\epsilon\\
    &\leq (\alpha+\cdots+\alpha(1-\alpha)^{k-1})\epsilon+\epsilon< 2\epsilon,
    \end{aligned}
    \end{equation*}
    which is impossible. This completes the proof.
    \end{proof}

    \subsection{Some properties of ergodic decomposition}\label{Sec 3.3}
    The primary objective of the ergodic decomposition theory of Markov operators or semigroups is to present an integral decomposition of invariant  measures into ergodic measures.   In practice, however, it is not convenient to study the ergodic measure either through the canonical dynamical system or the equivalent characterization as extreme points in the convex set of the invariant measure. A reasonable and feasible approach is to characterize ergodic measures in terms of a measurable subset of the state space, and to perform an integral decomposition over this subset of any invariant measure in terms of the ergodic measures. 
    
    Following this perspective, Worm and Hille \cite{WH2011-0,W2010,WH2011} give a detailed survey on this topic and establish the ergodic decomposition theorem of regular Markov operators and regular jointly measurable  Markov semigroups.  Specifically, in \cite{WH2010}, for Markov operators and semigroups with the (Ces\`aro) e-property on Polish spaces, Worm and Hille give some refined depictions of the decomposition that consists of the closedness and invariance of the sets in the decomposition, and obtain a continuous surjective function from one of these sets to the ergodic invariant  measures.   In \cite{SW2012}, Szarek and Worm introduce a weak concentrating condition around a compact set for Markov semigroups with e-property, providing a bijection between the a Borel subset of the compact set and the set of ergodic measures. 
    
    In particular, as Worm and Hille state on Page 33 in \cite{WH2011} that, a Markov process generates a regular jointly measurable Markov semigroup. Consequently, the results of the ergodic decomposition in \cite{WH2011-0,W2010,WH2011} can be directly applied when considering the ergodic behaviors of SPDEs, IFS and other Markov processes. 
    
    Therefore, in this subsection,  we shall assume that $\{P_t\}_{t\ge 0}$ is a regular jointly measurable Markov--Feller semigroup.  Similar to \cite{WH2011,SW2012}, we only consider some improved  properties, such as the closedness and the compactness, of  the subset in the state space related to the ergodic decomposition of  the eventually continuous Markov--Feller semigroups. To this end, we use an approach analogous to \cite{W2010,WH2011,SW2012}. More precisely, a continuous surjective map is established from the subsets of state space in ergodic decomposition to the ergodic measures for (Ces\`aro) eventually continuous Markov--Feller semigroups by Lemma \ref{Prop Phi} below.

    \vspace{3mm}
    
    Let us begin with some notions and results from \cite{W2010,WH2010,WH2011,SW2012}.  Let us consider the weak topology on $\mathcal{P}(\mathcal{X})$, which is metricized by the dual-Lipschitz distance $\|\cdot\|^*_{L}$:
	\begin{equation*}
		\|\mu-\nu\|_{L}^*:=\sup\{|\langle f,\mu \rangle-\langle f,\nu\rangle|:f\in L_b(\mathcal{X}),\|f\|_{L}:=\|f\|_{\infty}+\|f\|_{\rm Lip}\leq 1\}.
	\end{equation*}  
    Recall that by Lemma \ref{Prop T1*}, under the assumption of Ces\`aro eventual continuity,  $\{Q_t(x,\cdot)\}_{t\geq 0}$ weakly converges to an invariant measure for any $x\in\mathcal{T}$. We denote such measure by $\varepsilon_x$. Meanwhile, the set $\mathcal{T}$ is closed according to Lemma \ref{Prop T4}. This allows to define an equivalence relation $\sim$ on $\mathcal{T}$:  $x\sim y$ if and only if $\varepsilon_x=\varepsilon_y$, and such equivalence class is denoted  by $[x]$. Let $\Phi:\mathcal{T}\rightarrow\mathcal{P}(\mathcal{X})$ be defined by $\Phi(x)=\varepsilon_x$ for any $x\in\mathcal{T}$.	We begin with the following property of the map $\Phi$.
 	\begin{lemma}\label{Prop Phi}
		Let $\{P_t\}_{t\geq 0}$ be  Ces\`aro eventually continuous on $\mathcal{X}$. Then the map $\Phi$ is continuous from $\mathcal{T}$ to $\mathcal{P}(\mathcal{X}).$   Conversely, if $\Phi$ is continuous from $\mathcal{T}$ to $\mathcal{P}(\mathcal{X})$, then $\{P_t\}_{t\geq 0}$ is Ces\`aro eventually continuous on ${\rm Int}_{\mathcal{X}}(\mathcal{T})$. 
	\end{lemma}
   \begin{proof}  $\Rightarrow:$   Assume that $\{P_t\}_{t\geq 0}$ is Ces\`aro eventually continuous at $x\in\mathcal{T}.$ Fix $f\in L_b(\mathcal{X})$ and $\epsilon>0.$ There exists $\delta>0,$ for $y\in B(x,\delta)$
	\begin{equation*}
	\limsup\limits_{t\rightarrow\infty}|Q_tf(y)-Q_tf(x)|\leq\epsilon,
	\end{equation*}
	Let $y\in B(x,\delta)\cap\mathcal{T},$ then we have
	\begin{equation*}
	|\langle f,\varepsilon_x\rangle-\langle f,\varepsilon_y\rangle|=\lim\limits_{t\rightarrow\infty}|Q_tf(x)-Q_tf(y)|\leq\epsilon.
	\end{equation*}
	Thus $x\mapsto\langle f,\varepsilon_x\rangle$ is continuous. Now let $\{x_n\}_{n\geq 1}\subset\mathcal{T}$ converging to $x\in \mathcal{X}$. Then $x\in\mathcal{T}$ since $\mathcal{T}$ is closed, and $\langle f,\varepsilon_{x_n} \rangle\rightarrow\langle f,\varepsilon_x\rangle$ for any $f\in L_b(\mathcal{X}).$ Using \cite[Theorem 2.3.24]{W2010}, $$\lim_{n\rightarrow\infty}\|\varepsilon_{x_n}-\varepsilon_x\|_L^*=0,$$ hence $\Phi$ is continuous.
 
     \noindent  $\Leftarrow:$ Assume that $\Phi$ is continuous at $x\in\text{Int}_{\mathcal{X}}(\mathcal{T}).$ Then for any $\epsilon>0,$ there exists $\delta>0,$ for any $y\in B(x,\delta)\subset\mathcal{T}$ such that
	\begin{equation*}
	\|\varepsilon_x-\varepsilon_y\|_{L}^*<\epsilon.
	\end{equation*}
	Then for any $f\in L_b(\mathcal{X}),$ 
	\begin{equation*}
	\limsup\limits_{t\rightarrow\infty}|Q_tf(x)-Q_tf(y)|=\lim\limits_{t\rightarrow\infty}|\langle f,Q_t(x,\cdot)\rangle-\langle f,Q_t(y,\cdot)\rangle|=|\langle f,\varepsilon_x\rangle-\langle f,\varepsilon_y\rangle|\leq \epsilon\|f\|_{\infty},
	\end{equation*}
	which implies that $\{P_t\}_{t\geq 0}$ is Ces\`aro eventually continuous at $x$
       \end{proof}

   Using Lemma \ref{Prop Phi},  the following properties are satisfied, which are similar in the Ces\`aro e-property setting. Let $\mathcal{P}_{\rm erg}$ be the set of  ergodic probability measures of $\{P_t\}_{t\ge 0}$, and set
	\begin{equation*}
	 \mathcal{T}_{\rm erg}=\{x\in \mathcal{T}:\varepsilon_x\in\mathcal{P}_{\rm erg}\}.	    
        \end{equation*}

    \begin{proposition}\label{Prop decompose}
         Let  $\{P_t\}_{t\geq 0}$ be  Ces\`aro eventually continuous on $\mathcal{X}$. Then the following statements hold.
         \begin{itemize}
             \item[$(\runum{1})$] $\mathcal{P}_{{\rm erg}}$ is closed in $\mathcal{P}(\mathcal{X})$, and  $\mathcal{T}_{\rm erg}$ is closed in $\mathcal{X}$.
             \item[$(\runum{2})$]  For any $x\in\mathcal{T}_{\rm erg}$, the set $[x]$ is closed, and $\supp\varepsilon_x\subset[x]$.
             \item[$(\runum{3})$] The set  $D:=\cup_{\mu\in \mathcal{P}_{{\rm erg}}}\supp\mu$ is a $G_\delta$ set.
         \end{itemize}
    \end{proposition}

	Recall the weak concentration condition \eqref{eq ergodic composition} introduced in \cite{SW2012}, which is a stronger condition than the condition \eqref{eq 3.4}:
	
	\emph{
	There exists a compact set $K\subset \mathcal{X}$ such that for any $\epsilon>0$ and $x\in\mathcal{X}$}
    \begin{equation}\label{eq ergodic composition}\tag{$\mathcal{C}$}
		\limsup\limits_{t\rightarrow\infty}Q_t(x,K^\epsilon)>0.
    \end{equation}		
    This is a stronger condition than the condition \eqref{eq 3.4}, and further implies the existence of invariant measures for the eventually continuous Markov--Feller semigroups. Invoking \eqref{eq ergodic composition}, we deduce the following result.
	
	  \begin{proposition}\label{Prop 3.11}
          Let  $\{P_t\}_{t\geq 0}$ be   Ces\`aro eventually continuous on $\mathcal{X}$.  Assume that \eqref{eq ergodic composition} holds, then $\mathcal{P}_{\rm erg}$ is compact in $\mathcal{P}(\mathcal{X})$.
	\end{proposition}
        \begin{proof}[Proof.] 
            For any $\mu\in\mathcal{P}_{\rm erg}$, we claim that $\supp\mu\cap K\neq \emptyset$. Assume, on the contrary, $\supp\mu\cap K=\emptyset$. Let $x\in \mathcal{T}_{\rm erg}$ be such that $\mu=\varepsilon_x$. Since $\supp\varepsilon_x$ is closed and $K$ is compact, there exists an $\epsilon>0$ such that $K^\epsilon\cap \supp\varepsilon_x=\emptyset$. In particular, $\varepsilon_x(\overline{K^{\epsilon/2}})=0$. This leads that
            \begin{equation*}
                \limsup\limits_{t\rightarrow\infty}Q_t(y,\overline{K^{\epsilon/2}})\leq\varepsilon_x(\overline{K^{\epsilon/2}})=0,\quad\forall\,y\in\supp\varepsilon_x,            \end{equation*}            
           which contradicts \eqref{eq ergodic composition}.  Therefore, let $\hat{K}:=K\cap\mathcal{T}_{\rm erg}$, then $\hat{K}$ is compact. Moreover, $\Phi(\hat{K})=\mathcal{P}_{\rm erg}$, which completes the proof by the continuity of $\Phi$.
        \end{proof}

   The following result is a generalization of \cite[Theorem 3.8]{SW2012}.

	\begin{theorem}\label{Thm 3.3}
		  Let  $\{P_t\}_{t\geq 0}$ be   Ces\`aro eventually continuous on $\mathcal{X}$.  Assume that \eqref{eq ergodic composition} holds, then there exists a Borel set $K_0\subset K$ such that
			\begin{itemize}
				\item[$(\runum{1})$] for any ergodic measure $\mu$, there exists $x\in K_0$, such that $\mu=\varepsilon_x$.
				\item[$(\runum{2})$] if $x,y\in K_0,\,x\neq y$, then $\varepsilon_x\neq \varepsilon_y$.
				\item[$(\runum{3})$] $x\in{\rm supp }\;\varepsilon_x$ for any $x\in K_0$. In particular, $K_0\subset\mathcal{T}_{\rm erg}$.
			\end{itemize}
	\end{theorem}
    \begin{proof}
    Let $X:=\bigcup\limits_{\mu\in \mathcal{P}_{\rm erg}}\supp\mu \cap K$, then $X$ is a $G_\delta$ set by Proposition \ref{Prop decompose}. Moreover, $X\subset \mathcal{T}_{\rm erg}$ is Polish space in its relative topology by \cite[Theorem 3.1.2]{A1976}. Now the rest of the proof can be argued in the same way as that of \cite[Theorem 3.8]{SW2012}.
    \end{proof}

     \begin{remark}  It should be noted that the results in  Sections \ref{Sec 3.1new}-\ref{Sec 3.3} can also be applied to discrete-time semigroups accordingly. Some proofs of the continuous-time setting are more difficult and complex than those of the discrete-time.
    \end{remark}

    Lastly, as the concepts in this subsection are rather involved, we provide an example which is devoted to illustrating the structure of ergodic decompositions for Ces\`aro eventually continuous semigroups.

    \begin{example}\label{Ex 2}  Let  $\mathcal{X}=\Z$ and  let $P:\mathcal{X}\times\mathcal{B}(\mathcal{X})\rightarrow[0,1]$ be the transition function deﬁned by 
		\begin{align*}  
            &P(0,0)=P(1,3)=P(3,1)=1,\\
		&P(-n,0)=1-P(-n,-n-1)=\e^{-1/n^2},\quad\text{for }n\geq 1,\\ 
            &P(2n+1,0)=1-P(2n+1,1)=1/n,\quad\text{for }n\geq 2,\\
            &P(2n,2n-2)=1\quad\text{for }n\geq 1.
		\end{align*}
	Then let the Markov operator $P:\mathcal{B}(\mathcal{X})\rightarrow\mathcal{B}(\mathcal{X})$ be defined by $Pf(n):=\sum_{m\in\mathcal{X}}f(m)P(n,m)$ and the corresponding Markov--Feller semigroup  
	\begin{equation*}
	P_tf:=\sum_{n=0}^{\infty}\e^{-t}\frac{t^n}{n!}P^nf\quad\text{for }t\geq 0,\;f\in\mathcal{B}(\mathcal{X}).
	\end{equation*}

        By construction, $\{P_t\}_{t\geq 0}$ is Ces\`aro eventually continuous on $\mathcal{X}$. We will check that $\{P_t\}_{t\geq 0}$ verifies the assumptions of Theorem \ref{Thm 3.3} and that $\mathcal{T}=\N$. It is clear that $\N\subset\mathcal{T}$ by construction. Moreover, for any $n\geq 1$ and $N\geq 1$, one has
        \begin{equation*}	\liminf\limits_{t\rightarrow\infty}P_t(-n,\{k:k\leq -N\})\geq \exp(-\sum_{j=0}^{\infty}(n+j)^{-2})>0,
	\end{equation*}
	hence $-n\notin\mathcal{T}.$ In particular, this implies that $\mathcal{T}=\N$.

        Next, we verify the weak concentration condition \eqref{eq ergodic composition}. Note that, for $K=\{0,1,3\}$, 
        \begin{equation*}
            \lim\limits_{t\rightarrow\infty}Q_t(n,K)=1,\quad\forall\,n\in\N
        \end{equation*}
        Meanwhile, it follows that 
        \begin{equation*}
	\liminf\limits_{t\rightarrow\infty}P_t(-n,0)\geq 1-\exp(-n^{-2})>0,\quad\forall\,n\geq 1.
	\end{equation*}

       Collecting these observations, Theorem \ref{Thm 3.3} is applicable to $\{P_t\}_{t\geq 0}$. In fact, we may take $K_0=\{0,1\}\subset K$ such that the ergodic measures $\mathcal{P}_{\rm erg}$  are given by
        \begin{equation*}
            \mathcal{P}_{\rm erg}=\{\varepsilon_0,\varepsilon_1\}=\{\delta_0,(\delta_1+\delta_3)/2\}.
        \end{equation*}
        In addition, we have 
          \begin{align*}
            \supp\varepsilon_0=\{0\}\subset[0]=2\N,\quad\supp\varepsilon_1=\{1,3\}=[1]\quad\text{and} \quad\mathcal{T}_{\rm erg}=[0]\cup[1].
        \end{align*}
        Meanwhile, for any $2n+1\in\mathcal{T}\setminus\mathcal{T}_{\rm erg}$ with $n\geq 2$, $\{Q_t(2n+1,\cdot)\}_{t\geq 0}$ weakly converges to $\varepsilon_{2n+1}=\tfrac{1}{n}\varepsilon_0+(1-\tfrac{1}{n})\varepsilon_1$ as $t\rightarrow\infty$. The set $\mathcal{T}$ is thereby decomposed as
        \begin{equation*}
            \mathcal{T}=\bigcup_{n\geq 2}[2n+1]\cup\mathcal{T}_{\rm erg}=\bigcup_{n\geq 2}[2n+1]\cup[0]\cup[1].
        \end{equation*}
    Let us mention that the sweep property described in Proposition \ref{Prop Sweep} is also satisfied for $\{P_t\}_{t\geq 0}$: for any $N>M\geq 1$,
	\begin{equation*}
		\lim\limits_{t\rightarrow\infty}P_t(n,\{k:-N\leq k\leq -M\})=0\quad\forall n\in \mathcal{X},
	\end{equation*}
    which suggests that condition \eqref{eq 3.5} in Proposition \ref{Prop Sweep} could be replaced by \eqref{eq ergodic composition}.
    
    In summary, the ergodic decomposition of $\{P_t\}_{t\geq 0}$ can be roughly formulated as  Figure \ref{ergodic_dcmp} below. 

    \begin{figure}[h]
    \centering
    \includegraphics[width=0.7\textwidth]{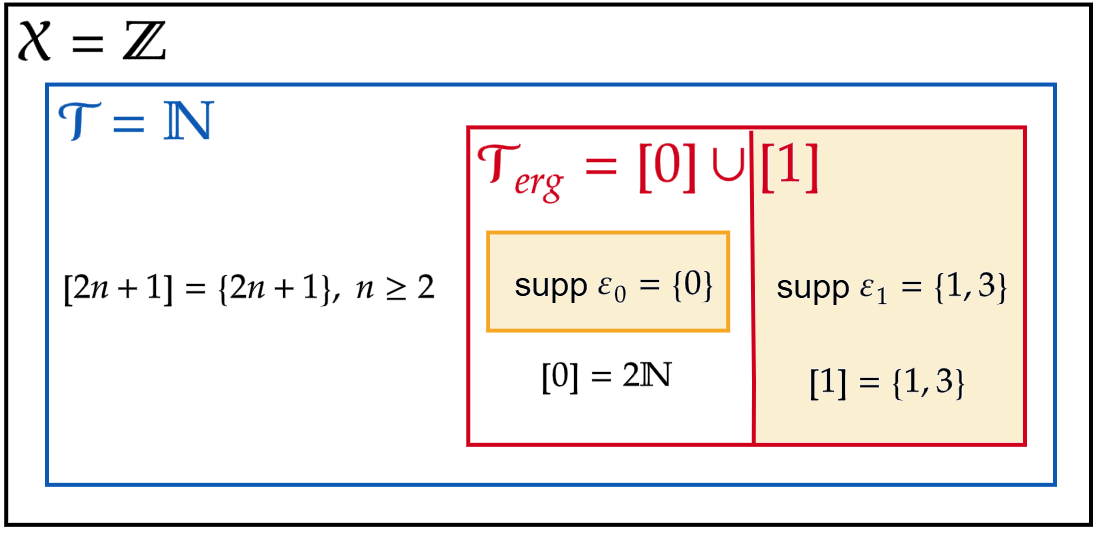}
    \caption{\small{Ergodic decomposition of Example \ref{Ex 2}}}
    \label{ergodic_dcmp}
    \end{figure}

    \end{example}

    \section{Ergodicity for iterated function systems with jumps}\label{Sec 4}
	
    In this section, we present a class of iterated function systems, for which the results established in the previous section can be applied. In particular, several specific cases are either non-equicontinuous or difficult to verify the (Ces\`aro) e-property. Before going through our models in Sections \ref{Sec 4.2} and \ref{Sec 4.3}, let us introduce some preliminaries on the IFS.  There is much literature on the asymptotic analysis of both discrete-time and continuous-time IFS with the e-property, see e.g.  \cite{HMS2005, CH2014, C2017,BKS2014}. In \cite{BKS2014}, it provides a continuous-time IFS and uses the notions of average bounded and concentrating at some point to establish the asymptotic stability. For comparison, we first introduce the continuous-time IFS provided in \cite{BKS2014} as follows. 
	
	\subsection{Preliminaries of iterated function systems}\label{Sec 4.1}
	Let us begin with the following notations:
	\begin{itemize}
	    \item $(\mathcal{X},\rho)$ is a Polish space, $I=\{1,\dots,N\}$ is a finite set;
	    \item $w_i:\mathcal{X}\rightarrow\mathcal{X},\;i\in I$ are transformations on $\mathcal{X}$;
	    \item $(p_1(x),\dots,p_N(x))$ is a probability vector $\forall\,x\in \mathcal{X},$ that is, $p_i(x)\geq 0$ for $i\in I,\;x\in\mathcal{X}$ and $\sum_{i\in I}p_i(x)=1$;
	    \item $\{\tau_n\}_{n\geq 1}$  is a sequence of random variables with $\tau_0=0,$ and $\triangle\tau_n=\tau_n-\tau_{n-1},n\geq 1$, are i.i.d. with density $\lambda \e^{-\lambda t},\;\lambda>0$;
	    \item $\{S(t)\}_{t\geq 0}$ is a continuous semigroup on $\mathcal{X}$.
	\end{itemize}\par 
	Now we define the $\mathcal{X}$-valued Markov process $\Phi=\{\Phi^x(t):x\in X\}_{t\geq 0}$ in the following way. 
	\begin{itemize}
	    \item[$\mathbf{Step\;1.}$]  Let $x\in\mathcal{X}$. Let $\Phi_0^x:=x$.
	    \item[$\mathbf{Step\;2.}$] 	 Let $\xi_1=S(\tau_1)(x).$  We randomly choose from the set $i_1\in I$ with probability $p_{i_1}(\xi_1),$ in other words,
	\begin{equation*}
	P(i_1=k)=p_{k}(\xi_1)=p_{k}(S(\tau_1)(x))\quad\text{for } k\in I.
	\end{equation*}
	Set $\Phi_1^x:=w_{i_1}(\xi_1).$
	    \item[$\mathbf{Step\;3.}$] Recursively, we assume that $\Phi_1^x,\dots,\Phi_{n-1}^x,n\geq 2$ are given. We define $\xi_n=S(\triangle\tau_n)(\Phi_{n-1}^x).$ Further, we randomly choose $i_n\in I$ with probability $p_{i_n}(\xi_n),$ i.e,
	\begin{equation*}
	P(i_n=k)=p_{k}(\xi_n)\quad\text{for } k\in I.
	\end{equation*}
	Then let $\Phi_n^x=w_{i_n}(\xi_n).$
	    \item[$\mathbf{Step\;4.}$] Finally, we set
	\begin{equation*}
	\Phi^x(t):=S(t-\tau_n)(\Phi_n^x)\quad\text{for  }\tau_n\leq t<\tau_{n+1},\;n\geq 0.
	\end{equation*}
	\end{itemize}\par 
	Now we can define semigroup $\{P_t\}_{t\geq 0}$ by
	\begin{equation*}
	P_tf(x)=\mathbb{E}f(\Phi^x(t))\quad\text{for } f\in B_b(\mathcal{X}).
	\end{equation*}\par

    For this model, in \cite{BKS2014}, they made the following assumptions:
    \begin{itemize}
        \item[$\mathbf{A_1.}$] $w_i,\,p_i$ are continuous for $i\in I$.
        \item[$\mathbf{A_2.}$]   There exists $r\in (0,1)$ such that
    \begin{equation}\label{IFS 1}
        \sum_{i\in I}p_i(x)\rho(w_i(x),w_i(y))\leq r\rho(x,y)\quad\text{for }x,y\in \mathcal{X}.
    \end{equation}
        \item[$\mathbf{A_3.}$]
    There exists a function $\omega:\;\mathbb{R}_+\rightarrow\mathbb{R}_+$ such that
	\begin{equation}\label{IFS 2}
	\sum_{i=1}^{N}|p_i(x)-p_i(y)|\leq\omega(\rho(x,y))\quad\text{for }x,y\in \mathcal{X},
	\end{equation}
	where $\omega$ is concave, non-decreasing and that
	\begin{equation*}
	    \int_{0}^{\epsilon}\frac{\omega(t)}{t}\d t<\infty\quad\text{for some }\epsilon>0.
	\end{equation*}
	 \item[$\mathbf{A_4.}$] There exists $\alpha\geq 0$ such that
    \begin{equation}\label{IFS 3}
        \rho(S(t)(x),S(t)(y))\leq \e^{\alpha t}\rho(x,y)\quad\text{for }x,y\in \mathcal{X}\text{ and } t\geq 0.
    \end{equation}
     \item[$\mathbf{A_5.}$] $\{S(t)\}_{t\geq 0}$ admits a global attractor, i.e., there exists a compact set $\mathcal{K}\subset\mathcal{X}$ such that $S(t)\mathcal{K}=\mathcal{K}$ for $t\geq 0$ and that for any bounded set $B\subset\mathcal{X}$ and open set $U,\,\mathcal{K}\subset U,$ there exists $t_*>0$ such that $S(t)B\subset U$ for $t\geq t_*$.
     \item[$\mathbf{A_6.}$]  $r$, $\alpha$ and $\lambda$ satisfy:
    \begin{equation}\label{IFS 4}
        r+\alpha/\lambda<1.
    \end{equation}
    \end{itemize}

    In \cite{BKS2014}, Bessaih, Kapica and Szarek prove that under assumptions $\mathbf{A_1}-\mathbf{A_6}$,  $\{P_t\}_{t\geq 0}$ has the e-property and  asymptotic stability.   Notice that the assumptions $\mathbf{A_1}, \mathbf{A_2}$ are essential for establishing the e-property, i.e., the continuity of  $w_i$'s and the uniform contraction together ensure the system to converge at the same rate.

	Thus, once the assumptions $\mathbf{A_1}$ and $\mathbf{A_2}$ are relaxed, $\{P_t\}_{t\geq 0}$ may not be an e-process. Based on the above intuitions, we present a specific and computable example (see in Section \ref{Sec 4.2}), where we drop the continuity assumption of $w_i$ and further allow the system not to contract uniformly. As a result, we still get an asymptotically stable semigroup, but it does not satisfy the Ces\`aro e-property. \par 
    
    Invoking these ideas further, in Section \ref{Sec 4.3} we introduce another place-dependent iterated function system, which can be seen as a generalization of some special cases in \cite{BKS2014} (see in Remark \ref{Rem IFS}). It is assumed to have a place-dependent contraction coefficient and not to be strictly contracted. It seems that the e-property is not satisfied for this model, and even if it is, it is still difficult to verify. Instead, we can directly prove the eventual continuity, and further obtain the asymptotic stability by Theorem \ref{Thm 4}.

    Let us mention that in the case of  complex systems, partial or the entire system may  exhibit periodicity or quasi-periodicity, thereby violating the asymptotic stability.  Conversely, the corresponding system  may  be stable in the sense of weak mean ergodicity defined as in Definition~\ref{Def 3}. For example, see Examples \ref{Ex 5} and \ref{Ex 7} below for illustration.

	\subsection{A non-equicontinuous example of iterated function systems}\label{Sec 4.2}
	As illustrated above, we modify the model in Section \ref{Sec 4.1} to provide a non-equicontinuous example. Actually, we assume that $w_i$ are piecewise continuous and that $p_i(\cdot)=0$ at the discontinuities of $w_i,$ for $i\in I$.
	\begin{example}\label{Ex 3}
		Let $\mathcal{X}=\mathbb{R}_+.\;$ Let $S(t)= \text{Id}_{X}$ for $t\geq 0,$ and
		\begin{equation*}
		w_1(x)=0,\quad w_2(x)=x,\quad w_3(x)=x^{-1}\mathbf{1}_{\{x\neq 0\}}.
		\end{equation*}
		Let $p_i\in C(\mathcal{X},[0,1]),i=1,2,3$ such that
		\begin{equation*}
		(p_1(x),p_2(x),p_3(x))=\begin{cases}
		(\frac{x}{2},1-x,\frac{x}{2}),&0\leq x< \frac{2}{3},\\
		(\frac{1}{3},\frac{1}{3},\frac{1}{3}),&\frac{2}{3}\leq x\leq \frac{3}{2},\\
		(\frac{1}{2x},1-x^{-1},\frac{1}{2x}),&x>\frac{3}{2}.
		\end{cases}
		\end{equation*}\par 
		This example differs from the IFS in \cite{BKS2014} in two ways. First, $\mathbf{A_1}$ is not satisfied because $w_3$ is not a continuous function.  On the other hand, we can calculate it for this example,
		\begin{equation*}
		   \sum_{i\in I}p_i(x)\rho (w_i(x),w_i(0))=(1-x)x+\tfrac{1}{2}\quad\text{for } x\in [0,\tfrac{2}{3}].
		\end{equation*}
		Hence the contraction coefficient $r(x)=1-x+\frac{1}{2x}$ for $x\in[0,\frac{2}{3}]$, and it tends to infinity as $x$ goes to zero, violating the assumption $\mathbf{A_2}$.\par 
		Notice that $0$ is the only absorbing state for this process, $\Phi^x_t$ almost surely converges to $0$ for any initial states $x\in \mathcal{X}$ as $t\rightarrow\infty.$ Therefore, $\{P_t\}_{t\geq 0}$ is asymptotically stable and eventually continuous on $\mathcal{X}$ by Proposition \ref{Prop 1}. However, it does not satisfy the Ces\`aro e-property at $0.$ Indeed, for $x_n=1/n$ and $s\geq 0$, 
		\begin{equation*}
		\begin{aligned}
		\mathbb{P}(\Phi_s^{x_n}=n)&\geq\sum_{k=1}^{\infty}\sum_{l=0}^{k-1}p_2(x_n)^lp_3(x_n)p_2(x_n^{-1})^{k-l-1}\mathbb{P}(\tau_k\leq s<\tau_{k+1})\\
		&=\frac{x_n}{2}\sum_{k=1}^{\infty}(1-x_n)^{k-1}\frac{(\lambda s)^k}{(k-1)!}\e^{-\lambda s}\\
		&=\frac{1}{2}x_n\lambda s\e^{-x_n\lambda s}.
		\end{aligned}
		\end{equation*}
		Now, choosing $f(x)=\min\{x,1\}\in C_b(\mathcal{X}),$ we have
		\begin{align*}
		Q_nf(x_n)&\geq f(n)\frac{1}{n}\int_{0}^{n}\mathbb{P}(\Phi_s^{x_n}=n)\d s\\
        &\geq\frac{\lambda}{2n^2}\int_0^ns\e^{-\lambda s/n}\d s= \frac{1}{2\lambda}(1-\lambda \e^{-\lambda }-\e^{-\lambda })>0=Q_nf(0),
		\end{align*}
		for any $\lambda>0$ and $n\geq 1$, which conflicts with the Ces\`aro e-property.
	\end{example}
	
     Below is a modification of Example \ref{Ex 3} to address the issue of periodicity, which leads to the fail of asymptotic stability.

    \begin{example}\label{Ex 5}
    Let $\mathcal{Y}=\R_+\times \T$. Define a Markov process on $\mathcal{Y}$ by
    \begin{equation*}
    F_{x,y}(t):=(\Phi^x_t,\Psi^y_t),
    \end{equation*}
    where $\Phi^x_t$ is given in Example \ref{Ex 3}, and 
    \begin{equation}\label{Ex 5-periodic}
    \Psi^y_t:=y+ t\;(\text{mod}\;2\pi),\quad\text{for }y\in\T,\;t\geq 0.
    \end{equation}

    Clearly, one can check that the associated Markov--Feller semigroup $\{P_t\}_{t\geq 0}$ satisfies the eventual continuity on $\mathcal{Y}$, while the Ces\`aro e-property fails at $0$. Moreover, it follows that the lower bound condition \eqref{eq 3.7} in Theorem \ref{Thm 3} is satisfied. Namely, for any $\epsilon>0$, one has
    \begin{equation*}
    \inf_{(x,y)\in\mathcal{Y}}\limsup_{t\rightarrow\infty}Q_t((x,y),B(0,\epsilon))=1.
    \end{equation*}
    Consequently, by Theorem \ref{Thm 3}, $\{P_t\}_{t\geq 0}$ is weakly-* mean ergodic.
    
    \end{example}

    \subsection{Ergodicity of a place-dependent iterated function system}\label{Sec 4.3}
    For brevity, we first introduce some notations. For $\mathbf{i}=(i_1,\dots,i_n)\in I^n,\, n\geq 1$, denote $\mathbf{i}[k]=(i_1,\dots,i_k)$ for $0\leq k\leq n$,
	\begin{equation*}
	\begin{aligned}
	&w_{\mathbf{i}[0]}(x)=x,\\
	&w_{\mathbf{i}[k]}(x)=w_{i_1}\circ\cdots\circ w_{i_k}(x)\quad\text{for }1\leq k\leq n.\\
	&w_{\mathbf{i}}(x)=w_{\mathbf{i}[n]}(x),\\
	&J_n(x):=\max\limits_{\mathbf{i}\in I^n}\Pi_{j=0}^{n-1}r(w_{\mathbf{i}[j]}(x)).
	\end{aligned}
	\end{equation*}

    Now we can introduce an IFS with a place-dependent contraction coefficient. The assumptions are as follows.
	\begin{itemize}
	     \item[$\mathbf{B_1.}$] $w_i,\,p_i$ are continuous for $i\in I$.
	    \item[$\mathbf{B_2.}$] There exists $z\in \mathcal{X}$ and a continuous function $r:\;\mathcal{X}\rightarrow (0,1]$ such that
	\begin{equation}\label{eq IFS 1}
		\sum_{i\in I}p_i(x)\rho(w_i(x),z)\leq r(x)\rho(x,z)\quad\text{ for }x\in \mathcal{X}.
	\end{equation}
	    \item[$\mathbf{B_3.}$] There exists a function $\omega:\;\mathbb{R}_+\rightarrow\mathbb{R}_+$ such that
	\begin{equation}\label{eq IFS 2}
	\sum_{i\in I}|p_i(x)-p_i(z)|\leq\omega(\rho(x,z))\quad\text{ for }x\in \mathcal{X},
	\end{equation}
	where $\omega$ is concave, non-decreasing and $\omega(0)=0$.
	\item[$\mathbf{B_4.}$] There exists $\alpha\geq 0$ such that
	\begin{equation}\label{eq IFS 3}
	\rho(S(t)(x),S(t)(y))\leq \e^{\alpha t}\rho(x,y)\quad\text{ for }x,y\in \mathcal{X},\;t\geq 0.
	\end{equation}
	\item[$\mathbf{B_5.}$] 	 
	 There exist $M\in\mathbb{N}$ and $\eta>0$ such that
	\begin{equation}\label{eq IFS 4}
	\sum_{n=M}^{\infty}\omega\left(J_n(x)\rho(x,z)\frac{\lambda^n}{(\lambda-\alpha)^n}\right)<1\quad\text{ for }x\in B(z,\eta).
	\end{equation}
	\end{itemize}

    We emphasize that  the contraction coefficient $r$  depends on the location and is not assumed to be strictly less than $1$. In other words, although the system may still be contracting, the rate of convergence may vary from place to place. Consequently, it seems rather complicated to verify the e-property. Instead, it is straightforward to establish the eventual continuity.

    \begin{proposition}\label{IFS Prop 1}
    Assume that $\mathbf{B_1}-\mathbf{B_5}$ hold, then $\{P_t\}_{t\geq 0}$ is eventually continuous at $z.$
    \end{proposition}

    It can be checked that the Markov semigroup generated by $\{\Phi_t^x\}_{t\geq 0}$ is Feller, see in Proposition \ref{IFS Prop Feller}. While $\mathbf{B_1}-\mathbf{B_5}$ are not enough to ensure the lower bound condition (\ref{eq 4.1}), we further have the following assumptions.
    \begin{itemize}
       \item[$\mathbf{C_1.}$]  There exists $\gamma\in (0,1),\;M\in\mathbb{N}$ and $\eta>0$ such that
	\begin{equation}\label{eq IFS 5}
    \sum_{n=M}^{\infty}\omega\left(J_n(x)\rho(x,z)\frac{\lambda^n}{(\lambda-\alpha)^n}\right)\leq 1- \gamma\quad\text{ for }x\in B(z,\eta).
	\end{equation}
    	\item[$\mathbf{C_2.}$]  $S(t)(z)=z$ for $t\geq 0.$
    	\item[$\mathbf{C_3.}$] 	For any $\epsilon>0,$ there exists $\beta=\beta(\epsilon)>0,$ and that for any $x\in \mathcal{X}$ there exists some $t_{x}=t(x,\epsilon)\geq 0$ such that
	\begin{equation}\label{eq IFS 6}
	  P_{t_{x}}(x,B(z,\epsilon))>\beta.
	\end{equation}
    \end{itemize}

    \begin{theorem}\label{Thm IFS}
    Assume that $\mathbf{B_1}-\mathbf{B_4}$ and $\mathbf{C_1}-\mathbf{C_3}$ hold, then $\{P_t\}_{t\geq 0}$ is asymptotically stable.
    \end{theorem}
    \begin{remark} \label{Rem IFS}
    In particular, taking $r(x)=r\in (0,1)$ for any $x\in\mathcal{X},$ then $\mathbf{A_2}$ is automatically satisfied. Thus, this model can be seen as a generalization of the IFS in \cite{BKS2014} in the sense that $\mathcal{K}=\{z\}$ (see $\mathbf{A_5}$).
    \end{remark}

    A simple and concrete example is given at the end of this subsection.  It satisfies the assumptions in Theorem \ref{Thm IFS}, but is not included in the model in \cite{BKS2014} because assumption $\mathbf{A_2}$ fails. As a result, it is difficult to check the e-property. However, we can use Theorem \ref{Thm IFS} to prove asymptotic stability immediately. 
    
    \begin{example}\label{Ex 6}
    Let $\mathcal{X}=\mathbb{R}_+.\;$ Let $S(t)= \text{Id}_{\mathcal{X}}\;$ for $t\geq 0,$ and
    \begin{equation*}
    w_1(x)=x/2,\quad w_2(x)=x.
    \end{equation*}
    Let $p_1(x)=\e^{-x}$ and $p_2(x)=1-\e^{-x}$ for $x\geq 0.$ 
    
    Then one may check that for  $z=0$,  $r(x)=1-\e^{-x}/2$,  $\omega(x)=2x$, $J_n(x)=(1-\e^{-x}/2)^n$, $M=0$,  $\eta = 1/8$,  $\gamma = 1-\e^{1/8}/4$,	all the assumptions $\mathbf{B_1}-\mathbf{B_4}$ and $\mathbf{C_1}-\mathbf{C_2}$ are satisfied. To check $\mathbf{C_3},$ note that $\forall\,x>0$
    \begin{align*}
    P_t(x,(x/2,x])&=P_t(x,\{x\})=\sum_{n=0}^{\infty}p_2(x)^n\frac{(\lambda t)^n}{n!}\e^{-\lambda t}=\exp(- \e^{-x}\lambda t )\rightarrow 0 \text{ as } t\rightarrow\infty.
    \end{align*}
    Thus there exists some $t_x^1$ such that
    \begin{equation*}
    P_{t_x^1}(x,[0,x/2])\geq 1-2^{-2}.
    \end{equation*}
    Inductively, for any $k\geq 1$ there exists some $t_x^k$ such that
    \begin{equation*}
    P_{t_x^k}(x,[0,x/2^k])\geq \Pi_{i=1}^{k}(1-(i+1)^{-2})\geq \Pi_{i=1}^{\infty}(1-(i+1)^{-2}):=\beta >0.
    \end{equation*}
    
    Thus this semigroup is eventually continuous at $0$ by Proposition \ref{IFS Prop 1}, and asymptotically stable by Theorem \ref{Thm IFS}.
    \end{example}
    
    \begin{example}\label{Ex 7} 
    Similar to Example \ref{Ex 5},  one can  construct a Markov process on $\mathcal{Y}$ by $\tilde{F}_{x,y}(t):=(\tilde{\Phi}_t^x,\Psi_t^y)$, where $\tilde{\Phi}_t^x$ is determined by Example \ref{Ex 6}, and $\Psi_t^y$ is given by \eqref{Ex 5-periodic}. It is not hard to verify that the Markov--Feller semigroup of $\tilde{F}_{x,y}(t)$ satisfies the Ces\`aro eventually continuity and weak-* mean ergodicity.
    \end{example}

    \section{Beyond  Ces\`aro eventual continuity} \label{Sec 5}
	
    In Section \ref{Sec 3}, we have established some comparatively complete results on the ergodic properties of Ces\`aro eventually continuous semigroups on Polish spaces. Nevertheless, the ergodicity of Markov semigroups is a much more complicated issue than the  previous results would suggest. To illustrate this, we present two Markov processes that do not possess Ces\`aro eventual continuity, exhibiting some interesting ergodic phenomena.  

    \subsection{Stochastic Hopf model} \label{hopf model}

    In \cite{Hopf1948} and \cite{Hopf1956}, E. Hopf investigates the  mechanism of the route to turbulence. Despite their lack of physical relevance, these two articles provide some interesting perspectives in mathematics, such as in dynamical systems and ergodic theory.  We consider Hopf's second model (see \cite{Hopf1956}).  The model is formulated on the torus $\T=\R/2\pi\Z$,  which reads
    \begin{equation}\label{hopf}
    \begin{cases}
        \partial_tw-\nu\partial_{xx}w+I(w)=L(w),\\
         w(0,x)=u(x)\in L^2(\T;\C),
    \end{cases}       
    \end{equation}
    where $\nu>0$ is the viscosity, and
    \begin{align*}
        (I(w))(t,x)&=\frac{1}{4\pi^2}\int_{\T^2}w(t,y)w(t,y')\overline{w}(t,y+y'-x)\d y \d y',\\
        (L(w))(t,x)&=\frac{1}{2\pi}\int_{\T}w(t,x+y)F(t,y)\d y.
    \end{align*}
    Here $\overline{w}$ demotes the complex conjugate of $w$ and $F\in L^2(\T;\C)$ is a given external force. Using the Fourier transform, Equation \eqref{hopf} can be written as
    \begin{equation*}
        \partial_tw_n=-\nu n^2w_n-w_n^2\overline{w}_n+F_nw_n\quad\text{for }n\in\Z,
    \end{equation*}
    with
    \begin{equation}\label{hopf0}
        w(t,x)=\sum_{n\in\Z}w_n(t)\e^{\im nx},\quad u(x)=\sum_{n\in\Z}u_n\e^{\im nx},\quad F(x)=\sum_{n\in\Z}F_n\e^{\im nx}.
    \end{equation}

    With these preparations, let us now formulate the random Hopf’s  model considered in the present paper. Instead of a deterministic external force $F$ along, we assume that Equation \eqref{hopf} is further perturbed by an additional random force, i.e., we replace $F(x)$ by $F(x)+\xi(t,x)$, where $\xi(t,x)$ is the random noise. More precisely, for simplicity, we assume that $\xi$ is the form of $\xi=\frac{\d W}{\d t}$. $W$ is a Wiener process in the conanical probability space $(\Omega,\mathcal{F},\P)$ given by 
    \begin{equation*}
        W(t)=\sum_{n\in\Z}b_n\beta_n(t) e_n.
    \end{equation*}
    Here $b_n\in\R$ are constants such that $\sum_{n\in\Z}b_n^2<\infty$, $e_n(x)=\e^{\im nx}$ and $\{\beta_n\}_{n\in\Z}$ is a sequence of independent standard real-valued Brownian motions.

    The model on $\mathcal{X}:=L^2(\T;\C)$ under consideration is therefore formulated by \eqref{hopf0} and 
    \begin{equation}\label{hopf1}
    \begin{cases}
      \d w_n=(-\nu n^2w_n-w_n^2\overline{w}_n+F_nw_n)\d t+b_nw_n\d \beta_n(t),\\
      w_n(0)=u_n\in\C,
    \end{cases} n\in \Z.
    \end{equation}

    By standard arguments, Equation \eqref{hopf1} admits a unique solution
    \begin{equation*}
        w(\,\cdot\,)= w(\,\cdot\,;u)\in L^2(\Omega;C([0,\infty);\mathcal{X}))
    \end{equation*}
    which depends continuously in $\mathcal{X}$ on $u$ for any $t\geq 0$. Moreover, Equation \eqref{hopf1} defines a Markov--Feller semigroup on $\mathcal{X}$ with transition functions
    \begin{equation*}
        P_t(u,A)=\P(w(t)\in A),\quad\text{for }A\in\mathcal{B}(\mathcal{X}),\;u\in\mathcal{X},\;t\geq 0.
    \end{equation*} 

    \begin{example}\label{Ex Hopf1}
    Given  that $F\in L^2(\T;\C)$ and $\sum_{n\in\Z}b_n^2<\infty$, without loss of generality, assume that there exist $N\geq 1$ such that
      \begin{equation*}
      a_n:=-\nu n^2+{\rm Re}\;F_n>0\quad\text{and}\quad b_n\neq 0\quad\text{ for }|n|\leq N,
    \end{equation*}
    and
    \begin{equation*}
       a_n\leq 0\quad\text{ for }|n|>N.
    \end{equation*}
   
    Let $\mathcal{X}_0\subset\mathcal{X}$ be defined by 
    \begin{equation*}
        \mathcal{X}_0:=\left\{u\in\mathcal{X}:u=\sum_{n\in\Z}u_ne_n,\;u_n\neq 0\;\text{for any }|n|\leq N\right\}.
    \end{equation*}

    \begin{proposition}
        Under the above settings, for any $z\in\mathcal{X}$, the semigroup $\{P_t\}_{t\geq 0}$ associated with Equation \eqref{hopf1} is Ces\`aro eventually continuous at $z$ if and only if $z\in\mathcal{X}_0$.
    \end{proposition}

    \begin{proof}
    Making use of polar coordinates 
    \begin{equation*}
        w_n(t)=r_n(t)\exp(\im \theta_n(t)),\quad u_n= r_n^u\exp(\im \theta_n^u),
    \end{equation*}
    Equation \eqref{hopf1} can be further written as
    \begin{equation}\label{hopf2}
        \begin{cases}
            \d r_n=(a_n+\tfrac{1}{2}b_n^2)r_n\d t-r_n^3\d t+b_nr_n\d \beta_n(t),\\
            \d\theta_n={\rm Im}\,F_n\d t,\\
            (r_n(0),\theta_n(0))=(r_n^u,\theta_n^u)\in\R_+\times\T,
        \end{cases} n\in \Z.
    \end{equation}
    Performing calculations similar to those in \cite[Section 2.3.7]{A1998}, Equation \eqref{hopf2} can be solved explicitly with solutions 
    \begin{equation}\label{hopf3}
        \begin{cases}
            \displaystyle r_n(t)=\frac{r_n^u\e^{a_nt+b_n\beta_n(t)}}{\left(1+2(r_n^u)^2\int_{0}^{t}\e^{2a_ns+2b_n\beta_n(s)}\d s\right)^{1/2}},\\
            \theta_n(t)=\theta_n^u+{\rm Im}\,F_n\,t\quad({\rm mod}\;2\pi),
        \end{cases} n\in \Z.
    \end{equation}
    Moreover, by Lemma \ref{Lemma r-convergence}, it follows that 
         \begin{align}\label{hopf4}
        \lim_{t\rightarrow\infty}r_n(t)\overset{a.s.}{=}0\quad&\text{ for } |n|>N,\;r_n^u>0,\\        \lim_{t\rightarrow\infty}r_n(t)\overset{d}{=}\lambda_n\quad&\text{ for }|n|\leq N,\;r_n^u>0,\label{hopf4.1}
    \end{align}    
    where the first convergence holds almost surely, while the second holds in the sense of weak convergence. Moreover, $\{\lambda_n\}_{|n|\leq N}$ are independent random variables given by
    \begin{equation*}
        \lambda_n(\omega)=\left(2\int_{-\infty}^{0}\e^{2a_nt+2b_n\hat{\beta}_n(t,\omega)}\d t\right)^{-1/2},
    \end{equation*}
    where $\{\hat{\beta}_n\}_{|n|\leq N}$ are independent standard (two-sided time) Brownian motions.

    Now consider another solution $\hat{w}$ of Equation \eqref{hopf1} with initial condition $\hat{u}=\sum_{n\in\Z}\hat{u}_ne_n$, and the corresponding solutions are denoted by $\hat{w}_n$, $\hat{r}_n$, $\hat{\theta}_n$.

    \begin{itemize}
        \item[$(\runum{1})$] {\it $\{P_t\}_{t\geq 0}$ is Ces\`aro eventually continuous on $\mathcal{X}_0$.}  For any $f\in L_b(\mathcal{X})$ with $\|f\|_{L}\leq 1$, $u,\hat{u}\in\mathcal{X}_0$,
    \begin{equation}\label{hopf5}
    \begin{aligned}
        |P_tf(u)-P_tf(\hat{u})|
        &\leq (\E\|w(t)-\hat{w}(t)\|^2\wedge 4)^{1/2}\\
        &\leq \left(\sum_{n\in\Z}(\E|r_n-\hat{r}_n|^2\wedge 4+ \E r_n\hat{r}_n|\theta_n-\hat{\theta}_n|^2)\right)^{1/2}\\
        &\leq \left(\sum_{n\in\Z}(\E|r_n-\hat{r}_n|^2\wedge 4+ |\theta^u_n-\hat{\theta}_n^u|^2\E (r_n^2+\hat{r}_n^2))\right)^{1/2}\\
        &\leq \left(\sum_{n\in\Z}\E|r_n-\hat{r}_n|^2\wedge 4+ \|u-\hat{u}\|^4\sum_{n\in\Z}\E (r_n^4+\hat{r}_n^4)\right)^{1/2}.   
    \end{aligned}
    \end{equation}

    It remains to control the two terms on the right-hand side of Equation \eqref{hopf5}. Using \eqref{hopf2} and It\^o's lemma, we can calculate that there exist constants $C_1,C_2,C_3>0$, $N'\geq N$, independent of $r_n^u>0$ and $t\geq 0$, such that
    \begin{align}
        \E r_n^4(t)\leq &\  C_1\e^{-C_2t}(r_n^u)^4,\quad \forall\,t\geq 0,\;|n|>N',\label{hopf6}\\
        \E r_n^4(t)\leq &\ C_3(1+r_n^u)^4,\quad\forall\,t\geq 0,\;|n|\leq N'.\label{hopf7}
    \end{align}

    Therefore,  the (Ces\`aro) eventual continuity at $u\in\mathcal{X}_0$ follows by first plugging \eqref{hopf4}, \eqref{hopf4.1}, \eqref{hopf6} and \eqref{hopf7} into \eqref{hopf5}, then letting $t$ go to infinity and $\hat{u}$ converge to $u$.

    \item[$(\runum{2})$]{\it $\{P_t\}_{t\geq 0}$ is not Ces\`aro eventually continuous on $\mathcal{X}\setminus\mathcal{X}_0$.} Let $u=\sum_{n\in\Z}u_n e_n\in\mathcal{X}\setminus\mathcal{X}_0$ be fixed. There exists $k$ with $|k|\leq N$ such that $u_{k}=0$. For any $R>0$, let $f_R\in L_b(\mathcal{X})$ be given by 
    \begin{equation*}
            f_R(u)=\min\{\|u_k\|,R\}.
    \end{equation*}
    Then for $\hat{u}:=u+\alpha \e_k$ with $\alpha>0$, it follows form \eqref{hopf4.1} that 
    \begin{align*}
        P_tf(\hat{u})-P_tf(u)=\langle f_R,\mathcal{D}(\hat{r}_n(t))\rangle\rightarrow\langle f_R,\mathcal{D}(\lambda_n)\rangle=\E\min\{\lambda_n,R\}\quad \text{as }t\rightarrow\infty.
    \end{align*}
    Here the notation $\mathcal{D}(\xi)$ denotes the law of random variable $\xi$.  Thus taking $R$ be sufficiently large, there would exist a constant $c>0$ independent of $\alpha$ such that
    \begin{equation*}
        \limsup_{t\rightarrow\infty}|Q_tf(\hat{u})-Q_tf(u)|\geq c>0,
    \end{equation*}
    which implies that the Ces\`aro eventual continuity fails at $u\in\mathcal{X}\setminus\mathcal{X}_0$.
    \end{itemize}

    \end{proof}
    Notice that $\theta_n(t)=\theta_n^u + \textrm{Im}\,F_n\, t\;({\rm mod}\,2\pi)$, $|n|\le N$ determine a conditionally periodic motion  (see \cite[Section 51, A]{ArndV}) or called a quasi-periodic motion  (see \cite{Moser}) on $\mathbb{T}^{2N+1}$. If the frequencies $\textrm{Im}\,F_n$, $|n|\le N$ are independent, i.e., they are linearly independent over the field of rational numbers: if $k_n\in \Z$, $|n|\le N$ and  $\sum_{|n|\le N}k_n \textrm{Im}\, F_n=0$, then $k_n=0$, for any $|n|\le N$. By Theorem on the averages (see \cite[p286]{ArndV}), if $g$ is a Riemann integrable function on $\mathbb{T}^{2N+1}$, then 
    \begin{equation}\label{hopf8}
    \lim_{t\to \infty}\frac{1}{t}\int_0^t g(\bftheta_N^u+\bfomega_N s)\d s=(2\pi)^{-(2N+1)}\int_{\mathbb{T}^{2N+1}}g(\mathbf{x})m(\d \mathbf{x}),
    \end{equation}
    where $\bftheta_N^u=(\theta_n^u, |n|\le N)$, $\bfomega_N=(\textrm{Im}\,F_n, |n|\le N)$ and $m(\d \mathbf{x})$ denotes the Lebesgue measure on $\mathbb{T}^{2N+1}$. 

    Recall that $Q_t(u,\cdot)=\frac{1}{t}\int_0^tP_s\delta_u \d s$, let $\{\eta_n\}_{|n|\leq N}$ be independent uniform random variables on $\T$. Then, combining \eqref{hopf4}, \eqref{hopf4.1} and \eqref{hopf8}, we obtain

    \begin{corollary}\label{hopf-cor3} 
    Assume that the ${\rm Im}\,F_n, |n|\le N$ are independent, then for any $u\in \mathcal{X}_0$, $Q_t(u,\cdot)$ weakly converges to $\mathcal{D}\left(\sum_{|n|\le N}\lambda_n\e^{\im \eta_n}e_n\right)$ and $\mathcal{D}\left(\sum_{|n|\le N}\lambda_n\e^{\im \eta_n}e_n\right)$ is an ergodic measure. Moreover, $\supp\mathcal{D}\,\left(\sum_{|n|\le N}\lambda_n\e^{\im \eta_n}e_n\right)= {\rm span}\,\{e_n:|n|\leq N\}$.
    \end{corollary}    

    Note that if the initial data of Equation \eqref{hopf2} $u=0$, then for any $t\ge 0$ and $n\in \Z$, $( r_n(t), \theta_n(t))=0$, i.e., $\delta_0$ is also an ergodic measure of $\{P_t\}_{t\ge 0}$. In particular,  this fact violates the conclusion of Theorem \ref{Prop T3*} and the EMDS-property, since the Ces\`aro eventual continuity fails at $0$. Furthermore, we have   

    \begin{proposition}
    Assume that  ${\rm Im}\,F_n, |n|\le N$ are linearly independent, then the semigroup $\{P_t\}_{t\geq 0}$ associated with Equation \eqref{hopf1} does not satisfy the EMDS-property on $\mathcal{X}$.
    \end{proposition} 
    \begin{proof}
        Using \eqref{hopf3}, \eqref{hopf4}, \eqref{hopf4.1} and Theorem on the averages in \cite{ArndV}, we derive that the ergodic measures for $\{P_t\}_{t\geq 0}$ are
        \begin{equation*}
            \mathcal{P}_{\rm erg}=\left\{\mu_l=\mathcal{D}\left(\sum_{|n|\leq N}l_n\lambda_n\e^{\im \eta_n}e_n\right):l=(l_{-N},\cdots,l_N)\in\{0,1\}^{2N+1}\right\},
        \end{equation*}
        In particular, one has
         \begin{equation}\label{EMDS fail}
            \supp\mu_0\cap  \supp\mu_l=\{0\}\neq \emptyset\quad\forall\;l\in\{0,1\}^{2N+1},
        \end{equation}
        which conflicts the EMDS-property. Indeed, to prove \eqref{EMDS fail}, it suffices to check that for any $\varepsilon>0$ and $|n|\leq N$,
        \begin{equation*}
            \P(\lambda_n<\epsilon)>0.
        \end{equation*}    
    Note that $\mathcal{D}(\lambda_n):=\nu_n\in\mathcal{P}(\R_+)$ is an invariant measure of $\{r_n(t)\}_{t\ge 0}$ in Equation \eqref{hopf3} and denote the corresponding Markov semigroup by $p_t^{n}$. For any $R>0$ and $t>0$, we have 
    
    \begin{align*}
    \P(\lambda_n<\epsilon)&=\nu_n((0,\epsilon))=p_t^{n}\nu_n((0,\epsilon))=\int_{\R_+}p_t^{n}(x,(0,\epsilon))\nu_n(\d x)\\
            &\geq \nu_n((0,R))\inf_{0<x\leq R}p_t^{n}(x,(0,\epsilon)).
    \end{align*}    
    Fixing $R>0$ sufficiently large such that $ \nu_n((0,R))\geq 1/2$, it remains to estimate the infimum of the above inequality. Recall that Equation \eqref{hopf3} has solution \eqref{hopf4}. Moreover, for $f_{t,R}(s)=-\tfrac{a_n}{b_n}s\in C^1([0,t];\R)$ with $t>4\epsilon^{-2}$, we have
    \begin{equation*}
        \frac{x\e^{a_nt+b_nf_{t,R}(t)}}{\left(1+2x^2\int_{0}^{t}\e^{2a_ns+2b_nf_{t,R}(s)}\d s\right)^{1/2}}<\frac{1}{2}\epsilon\quad\forall\, x\in(0,R].
    \end{equation*}
    In particular, in view of the continuity, it implies that there exists $\delta>0$ such that for any $f\in C([0,t];\R)$ with $f(0)=0$ and $\|f-f_{t,R}\|_{C([0,t];\R)}\leq \delta$, 
    \begin{equation*}
        \frac{x\e^{a_nt+b_nf(t)}}{\left(1+2x^2\int_{0}^{t}\e^{2a_ns+2b_nf(s)}\d s\right)^{1/2}}<\epsilon\quad\forall\, x\in(0,R].
    \end{equation*}
    Therefore, the proof is completed by noticing that
    \begin{equation*}
        \inf_{0<x\leq R}p_t^{n}(x,(0,\epsilon))\geq \P(\|\beta_n-f_{t,R}\|_{C([0,t];\R)}\leq \delta)>0.
    \end{equation*}
    \end{proof}

    Now, we consider Hopf's original model \eqref{hopf}, i.e., $b_n=0, n\in\Z$ in Equation \eqref{hopf1} and \eqref{hopf2}, 
        \begin{equation}\label{hopf9}
        \begin{cases}
            \frac{\d r_n}{\d t}=a_nr_n-r_n^3,\\
            \frac{\d\theta_n}{\d t}={\rm Im}\,F_n,\\
            (r_n(0),\theta_n(0))=(r_n^u,\theta_n^u)\in\R_+\times\T,
        \end{cases} n\in \Z,
    \end{equation}
    and assume that $a_n:=-\nu n^2+{\rm Re}F_n>0$, for $|n|\le N$ and $a_n\le 0$ for $|n|>N$, they by \cite[Theorem in p52]{Hopf1956}, one has the following.     
    
    \begin{proposition} \label{hopf-prp5} Assume that  ${\rm Im}\,F_n, |n|\le N$ are independent. Let  $\{P_t\}_{t\ge 0}$ be the semigruop associated with Equation \eqref{hopf9}.       
    \begin{itemize}
    \item[$(\runum{1})$] If $N=0$, i.e., ${\rm Re}\, F_0>0$, $a_n<0$, $|n|\neq 0$, then  $\{P_t\}_{t\ge 0}$  is not (Ces\`aro) eventually continuous at $0\in \mathcal{X}$, but (Ces\`aro) eventually continuous at $\mathcal{X}\setminus \{0\}$.  $\{P_t\}_{t\ge 0}$ has only two ergodic measures, $\delta_0$ and $\mathcal{D}(\sqrt{{\rm Re}F_0}\e^{\im \eta_0}e_0)$. Moreover, $\{P_t\}_{t\ge 0}$ satisfies the EMDS-property.
    
    \item[$(\runum{2})$] If $N>0$, then $\{P_t\}_{t\ge0}$ is Ces\`aro eventually continuous at $z$ if and only if $z\in\mathcal{X}_0$. The ergodic measures for  $\{P_t\}_{t\ge0}$ are 
    
    \begin{equation*}
        \mathcal{P}_{\rm erg}=\left\{\mu_l=\mathcal{D}\left(\sum_{|n|\leq N}l_n\sqrt{a_n}\e^{\im \eta_n}e_n\right):l=(l_{-N},\cdots,l_N)\in\{0,1\}^{2N+1}\right\}.
    \end{equation*} 
    Moreover, set ${\bf 1}:=l$, if for any $|n|\le N$, $l_n=1$ and ${\bf 0}:=l$, if for any $|n|\le N$, $l_n=0$.  If $l\neq {\bf 1}$ and $l\neq {\bf 0}$, then ${\rm supp}\,\mu_l\subsetneq {\rm supp}\, \mu_{\bf 1}$, and ${\rm supp}\,\delta_0\cap {\rm supp}\, \mu_{\bf 1}=\emptyset$.
    \end{itemize}
    \end{proposition}

    \end{example}

    \begin{remark}
    Indeed, more refined estimates imply that $\{P_t\}_{t\ge 0}$ associated with Equation \eqref{hopf2} and \eqref{hopf9} satisfies the Ces\`aro e-property at $z\in \mathcal{X}_0$. Nevertheless, this enhancement is not necessary for the purposes of this example, which concerns non-Ces\`aro eventual continuity. 

    By now we have considered the case where ${\rm Im}\,F_n, |n|\le N$ are  independent. If ${\rm Im}\,F_n, |n|\le N$ are not independent, the ergodic properties for the semigroup $\{P_t\}_{t\ge 0}$ are similar to Equation \eqref{hopf2} and \eqref{hopf9} respectively, while slightly more complex analysis and techniques are required.       
    \end{remark}

    \begin{example}\label{Ex Hopf2}
    In this example, we are interested in phase transition phenomena in the ergodicity for   $\{P_t^{\nu}\}_{t\ge 0}$ of  Equation \eqref{hopf1} induced by viscosity $\nu$.  To illustrate the essence of the relation between the Ces\`aro eventual continuity/non-Ces\`aro eventual continuity and $\nu$, we simplify the assumption of the parameters as follows. Given that $F\in L^2(\T;\C)$, there exist an integer  $\tilde{N}\ge 1$ such that 
    \begin{equation*}
      {\rm Re}\;F_n=0, \ \text{for any}\  |n|>\tilde{N}
       \ \text{and} \ n=0,
    \end{equation*}
    \begin{equation*}
      b_n\neq 0,\ {\rm Re}\; F_n=c>0,\  \text{for any}\ 0<|n|\leq\tilde{N}, \text{where} \ c \ \text{is a constant},
    \end{equation*}    
    and  $\{{\rm Im}\; F_n, 0<|n|\leq \tilde{N}\}$ are independent, then $a_n<0$ for any $|n|>\tilde{N}$ and $a_0=0$. Thus, Lemma \ref{Lemma r-convergence} implies that 
    \begin{itemize}
        \item[$(\runum{1})$] $\{P_t^{\nu}\}_{t\ge 0}$ is asymptotic stability and Ces\`aro eventually continuous if $\nu\ge c$,
        \item[$(\runum{2})$]  $\{P_t^{\nu}\}_{t\ge 0}$ admits two ergodic measures and does not satisfy  non-Ces\`aro eventual continuity  if $0<\nu<c$.
    \end{itemize}
    Furthermore, by Corollary \ref{hopf-cor3}, if $\frac{c}{(\tilde{N}-i)^2}\le \nu< \frac{c}{(\tilde{N}-(i+1))^2}$,  then besides $\delta_0$, the other ergodic measures of  $\{P_t^{\nu}\}_{t\ge 0}$ are $\mathcal{D}(\sum_{1\le |n|\le \tilde{N}-(i+1)}\sqrt{-\nu n^2+c}\,\e^{\im \eta_n}e_n)$ for $i=0,\dots,\tilde{N}-2$;  and if $0<\nu<\frac{c}{\tilde{N}^2}$,  the other ergodic measure is $\mathcal{D}(\sum_{1\le |n|\le \tilde{N}}\sqrt{-\nu n^2+c}\,\e^{\im \eta_n}e_n)$.  
    
    \end{example}

    \subsection{Stochastic Lorenz system} \label{lorenz system}

    In \cite{CotiHairer}, Coti Zetali and Hairer consider some ergodic behaviour of a stochastic perturbation of the classical  Lorenz system as follows.
    \begin{equation}\label{eq lorenz}
       \left\{\begin{aligned} 
       \d X_t=&\,\sigma (Y_t-X_t)\d t,\\
       \d Y_t=&\,X_t(\varrho-Z_t)\d t-Y_t\d t,\\
       \d Z_t=&\,-(\beta Z_t+X_tY_t)\d t+ \hat{\alpha}\d W_t,       
       \end{aligned}
       \right.
    \end{equation}
    where $\sigma,\beta,\hat{\alpha}>0$ and $\varrho<1$ are constants, $W$ denotes a  standard Brownian motion.

    Let $\nu_0:=\delta_0\times \delta_0\times \mathcal{N}(0,\tfrac{\hat{\alpha}^2}{2\beta})\in\mathcal{P}(\R^3)$, where $\delta_0$ denotes the dirac measure at $0$ and $\mathcal{N}(0,\tfrac{\hat{\alpha}^2}{2\beta})$  denotes the zero-mean normal distribution  with variance $\tfrac{\hat{\alpha}^2}{2\beta}$. By definition, $\supp\nu_0=H:=\{(0,0,z):z\in\R\}$.

    \begin{theorem}[Theorem 1.1 in \cite{CotiHairer}]
    For any $\sigma,\beta>0$ and $\rho<1$, there exist  constants $0<\alpha_*\leq \alpha^*<\infty$ such that
    \begin{itemize}
    \item [$(\runum{1})$] For $0\leq \hat{\alpha}<\alpha_*$, Equation \eqref{eq lorenz} admits $\nu_0$ as its unique invariant measure.
    \item [$(\runum{2})$] For $\hat{\alpha}>\alpha^*$, Equation \eqref{eq lorenz} admits exactly two ergodic invariant measures: $\nu_0$ and another measure $\nu_*\in\mathcal{P}(\R^3)$. 
    \end{itemize}

    \end{theorem}

    \begin{proposition}[Proposition 3.3 in \cite{CotiHairer}]\label{Prop lorenz}
    For any $\boldsymbol{x}_0,\hat{\boldsymbol{x}}_0\in\R^3\setminus H$ and $\epsilon>0$, there exists $T=T(\|\boldsymbol{x}_0\|,\|\hat{\boldsymbol{x}}_0\|,\epsilon)>0$ and a function $h\in C^1([0,T];\R)$ such that      the equation 
    \begin{equation}  \Dot{x}=\sigma(x-y),\quad   \Dot{y}=x(\varrho-z)-y, \quad  \Dot{z}=-(\beta z+xy)+ h,\label{eq lorenz2}
     \end{equation}
    with initial condition $(x(0),y(0),z(0))=\boldsymbol{x}_0$ is well-posed, and satisfies that
    \begin{equation*}
    \|(x(T),y(T),z(T))-\hat{\boldsymbol{x}}_0\|<\epsilon.
    \end{equation*}
    \end{proposition}

    Invoking Proposition \ref{Prop lorenz},  we conclude that the EMDS-property also fails for the Lorenz system \eqref{eq lorenz} as the following corollary, whose proof is given  in Section  \ref{pr of beyond}. 

    \begin{corollary}\label{Coro lorenz}
     For $\hat{\alpha}>\alpha^*$, the ergodic measure $\nu_*$ for Equation \eqref{eq lorenz} satisfies  $\supp\nu_*=\R^3$. In particular, $\supp \nu_0\cap \supp \nu_*=H$, and the Ces\`aro eventual continuity fails. 
     \end{corollary}

     As previously noted in \cite[Remark 1.2]{CotiHairer}, one would naturally expect  $\alpha_*=\alpha^*$. Inspired by Example \ref{Ex Hopf2}, we suspect that there is a critical point $\tilde{\alpha}$ at which the Ces\`aro eventual continuity/non-Ces\`aro eventual continuity phase transition occurs and furthermore $\tilde{\alpha}=\alpha^*=\alpha_*$.

    \vspace{3mm}

    Another interesting problem is what new ergodic phenomena occur when the state space is a metric space, but not necessarily complete. For instance, in  \cite{WW2018}, the authors  formulate a criterion for the existence of invariant measures on a metric space for Feller semigroups that satisfy the e-property for bounded continuous functions, and further use it to prove the asymptotic stability via lower bound conditions. 

    The  examples above illustrate some different ergodic behaviours of non-Ces\`aro eventually continuous semigroups.  It is postulated that analogous phenomena may frequently occur in chaotic or turbulent models. These represent promising avenues for further investigation.

    \section{Proofs}\label{Sec 6}

    We in this section collect the technical proofs used in the main text.

    \subsection{Proofs of  Section \ref{Sec 2} }\label{prf sec2}

    \paragraph{Proof of Lemma \ref{Prop T4}:}
    
    \begin{proof}

    Assume, contrary to our claim, that $\mathcal{T}$ is not closed. Then there exists a sequence $\{x_n\}_{n\geq 1}\subset\mathcal{T},\;x_n\rightarrow x$  converging to $x$ such that $\{Q_t(x,\cdot)\}_{t\geq 0}$ is not tight. Thus there exists a strictly increasing sequence of positive numbers $t_i$ going to infinity,   $\epsilon>0$ and a sequence of compact sets $\{K_i\}_{i\geq 1}$ such that
		\begin{equation*}\label{close1}
		Q_{t_i}(x,K_i)\geq\epsilon,\quad\forall\,i\geq 1,
		\end{equation*}
		and
		\begin{equation*}\label{close2}
		\min\{\rho(x,y):x\in K_i,y\in K_j\}\geq\epsilon,\quad\forall\,i\neq j.
		\end{equation*} 
		We will derive the assertion from the claim that there exists a sequence of Lipschitz functions $\{\bar{f}_n\}_{n\geq 1}\subset L_b(\mathcal{X})$ and an increasing sequence of integers $\{m_n\}_{n\geq 1}$ such that 
		\begin{equation}\label{close3}
		\mathbf{1}_{K_{m_n}}\leq\bar{f}_n\leq\mathbf{1}_{K_{m_n}^{\epsilon/4}}\quad\text{and}\quad \|\bar{f}_n\|_{\rm Lip}\leq 4/\epsilon\quad \forall\,n\geq 1.
		\end{equation}
		Moreover,
		\begin{equation}\label{close4}
		Q_t(x_n,\cup_{i=n}^{\infty}K_{m_i}^{\epsilon/4})\leq\epsilon/4\quad\forall\,t\geq 0,
		\end{equation}
		and
		\begin{equation}\label{close5}
		\limsup\limits_{t \to \infty}|Q_{t}f_n(x)- Q_{t}f_n(x_n)|\leq\epsilon/4,
		\end{equation}
		where $f_1:=0,f_n:=\sum_{i=1}^{n-1}\bar{f}_i,n\geq 2.$ It can be easily checked that $f:=\sum_{i=1}^{\infty}\bar{f}_i\in L_b(\mathcal{X})$ with $\|f\|_{\infty}=1.$ Then it follows that
		\begin{equation}\label{close6}
		\begin{aligned}
		Q_tf(x)-Q_tf(x_n)&\geq Q_{t}(x,\cup_{i=n}^\infty K
		_{m_i})-|Q_{t}f_n(x)-Q_{t}f_n(x_n)|-Q_{t}(x_n,\cup_{i=n}^{\infty}K_{m_i}^{\epsilon/4}).
		\end{aligned}
		\end{equation}
		From (\ref{close4})-(\ref{close6})  it follows that
		\begin{equation*}
		\begin{aligned}
		\limsup\limits_{t \to \infty}[Q_tf(x)-Q_tf(x_n)]&\geq\limsup\limits_{i\to \infty}[Q_{t_{m_i}}f(x)-Q_{t_{m_i}}f(x_n)]\\		
		&\geq\limsup\limits_{i\to \infty}Q_{t_{m_i}}(x,\cup_{i=n}^\infty K
		_{m_i})-\liminf\limits_{i\to \infty}|Q_{t_{m_i}}f_n(x)-Q_{t_{m_i}}f_n(x_n)|-\tfrac{\epsilon}{4}\\
		&\geq\limsup\limits_{i\to \infty}Q_{t_{m_i}}(x,\cup_{i=n}^\infty K
		_{m_i})-\limsup\limits_{t\to \infty}|Q_{t}f_n(x)-Q_{t}f_n(x_n)|-\tfrac{\epsilon}{4}\\
		&\geq\tfrac{\epsilon}{2}
		\end{aligned}
		\end{equation*}
		for any $n\geq 1.$ This clearly contradicts the Ces\`aro eventual continuity of $\{P_t\}_{t\geq 0}$ at $x.$\par 
		
	\noindent 	$Proof\;of\;the\;claim.$ We accomplish this by induction on $n$. Let $n = 1$. Given $x_1\in\mathcal{T},$ there exists some compact $K$ such that $Q_t(x_1,K)\geq 1-\epsilon/4,$ for $t\geq 0.$ $\mu(B(x,\delta))>0$ for any $\delta>0$. Consequently, there exists an integer $m_1$ such that
		\begin{equation*}
		Q_t(x_1,\cup_{i=1}^{\infty}K_{m_1}^{\epsilon/4})\leq\epsilon/4\quad\forall\,t\geq 0,
		\end{equation*}
		Let $\bar{f}_1$ be an arbitrary Lipschitz function satisfying
		\begin{equation*}
		\mathbf{1}_{K_{m_1}}\leq\bar{f}_1\leq\mathbf{1}_{K_{m_1}^{\epsilon/4}}\quad\text{and}\quad \|\bar{f}_1\|_{\rm Lip}\leq 4/\epsilon.
		\end{equation*}
		Assume, now, that for a given $n\geq 1$, we have already constructed $\bar{f}_1,\dots,\bar{f}_n,m_1,\dots,m_n$ satisfying the claim. Since $\{Q_tf_{n+1}\}_{t\geq 0}$ is eventually continuous, we can choose $x_k\in\{x_n\}_{n\geq 1}$ such that 
		\begin{equation*}
		\limsup\limits_{t \to \infty}|Q_tf_{n+1}(x)-Q_tf_{n+1}(x_k)|<\epsilon/4.
		\end{equation*}
		Without loss of generality, assume that $k=n+1.$ Finally, we let $\bar{f}_{n+1}$ be an arbitrary continuous function satisfying (\ref{close3}).
		
	\end{proof}
        
	\paragraph{Proof of Lemma \ref{Prop T1*}:}
	
	\begin{proof} It only needs to show that this theorem holds for $\delta_x$ with $x\in \mathcal{T}.$ Since if there exist invariant measures $\{\varepsilon_x\}_{x\in\mathcal{T}}$ such that $Q_t(x,\cdot)$ weakly converges  to $\varepsilon_x$ as $t\rightarrow\infty$ for any $x\in\mathcal{T},$ then for $\mu\in\mathcal{P}(\mathcal{X})$ with $\supp\mu\subset\mathcal{T}, $ it implies that
		\begin{equation*} 
		\lim\limits_{t\rightarrow\infty}Q_t\mu(\cdot)=\int_\mathcal{X}\varepsilon_x(\cdot)\mu(\d x),
		\end{equation*}
		and in particular, $\int_\mathcal{T}\varepsilon_x(\cdot)\mu(\d x)$ is an invariant measure for $\{P_t\}_{t\geq 0}.$\par 
		
		Fix $x\in\mathcal{T}$ and assume, contrary to our claim, that the sequence $\{Q_t(x,\cdot)\}_{t\geq 0}$ does not converge.  Then by tightness of $\{Q_t(x,\cdot)\}_{t\geq 0}$ and  Prokhorov theorem, we can find at least two different probability measures $\mu_1,\mu_2$ and two sequences of positive numbers $\{s_n\}_{n\geq 1}$ and $\{t_n\}_{n\geq 1}$ increasing to infinity such that $\{Q_{s_n}(x,\cdot)\}_{n\geq 1},\;\{Q_{t_n}(x,\cdot)\}_{n\geq 1}$ weakly converges to $\mu_1,\mu_2$ as $n\rightarrow\infty$, respectively.\par 
		Choose $f\in L_b(\mathcal{X})$ and $\epsilon>0$ such that $|\langle f,\mu_1\rangle-\langle f,\mu_2\rangle|>\epsilon.$  Let $D=\{x_k\}_{k\geq 1}$ be a countable dense set of $\mathcal{X}$. Passing to a subsequence if necessary, we may assume that $\lim\limits_{n\rightarrow\infty}Q_{s_n}f(x_k)$ exists for any $k\geq 1.$ Now let $\overline{g}(x):=\limsup\limits_{n\rightarrow\infty}Q_{s_n}f(x)$ and $\underline{g}(x):=\liminf\limits_{n\rightarrow\infty}Q_{s_n}f(x)$ for $x\in \mathcal{X}.$ We claim that $\overline{g}=\underline{g}\in C_b(\mathcal{X}),$ and we shall denote $g=\overline{g}=\underline{g}.$ Indeed, using the Ces\`aro eventual continuity of $\{P_t\}_{t\geq 0},$ for any $x\in \mathcal{X}$ and $\eta>0,$ there exists some $x_k$ such that
		\begin{equation*}
		\limsup\limits_{n\rightarrow\infty}|Q_{s_n}f(x)-Q_{s_n}f(x_k)|\leq\eta/2.
		\end{equation*}
		Then one has
		\begin{equation*}
		|\overline{g}(x)-\underline{g}(x_k)|\leq\eta/2\quad\text{and }\quad|\underline{g}(x)-\overline{g}(x_k)|\leq\eta/2,
		\end{equation*}
		hence
		\begin{equation*}
		|\overline{g}(x)-\underline{g}(x)|\leq\eta\quad\forall\,x\in \mathcal{X},\;\eta>0.
		\end{equation*}
		Thus we conclude that $\{Q_{s_n}f\}_{n\geq 1}$ converges to $g\in C_b(\mathcal{X})$ pointwisely. \par 
		
		By the bounded convergence theorem and invariance, we have
		\begin{equation*}
		\langle f,\mu_2\rangle =\lim\limits_{n\rightarrow\infty}\langle Q_{s_n}f,\mu_2 \rangle=\langle g,\mu_2\rangle.
		\end{equation*}
		Since $\{Q_{t_n}(x,\cdot)\}_{n\geq 1}$ weakly converges to $\mu_{2}$  as $n\rightarrow\infty$, we can fix $N\in\mathbb{N}$ so that
		\begin{equation}\label{eq 5.1}
		|\langle g,\mu_{2}\rangle-\langle g,Q_{t_N}(x,\cdot)\rangle|\leq \epsilon/5.
		\end{equation}
		For such $N,$ we can choose $n$ sufficiently large such that
		\begin{equation}\label{eq 5.2}
		|\langle g,Q_{t_N}(x,\cdot)\rangle-\langle Q_{s_n}f,Q_{t_N}(x,\cdot)\rangle|\leq \epsilon/5
		\end{equation}
		Furthermore, by \cite[Lemma 2]{KPS2010}, we have
		\begin{equation*}
		\lim\limits_{n\rightarrow\infty}\|Q_{s_n,t_N}(x,\cdot)-Q_{s_n}(x,\cdot)\|_{\rm TV}=0.
		\end{equation*}
 Thus we fix $n$ sufficiently large such that
		\begin{equation}\label{eq 5.3}
		|\langle f,Q_{s_n,t_N}(x,\cdot)\rangle-\langle f,Q_{s_n}(x,\cdot)\rangle|\leq \epsilon/5
		\end{equation}
		Finally, noting that $\{Q_{s_n}(x,\cdot)\}_{n\geq 1}$ weakly converges to $\mu_{1}$ as $n\rightarrow\infty$, there exists $n$ sufficiently large such that
		\begin{equation}\label{eq 5.4}
		|\langle f,Q_{s_n}(x,\cdot)\rangle-\langle f,\mu_{1}\rangle|\leq \epsilon/5
		\end{equation}\par 
		Collecting (\ref{eq 5.1})-(\ref{eq 5.4}) we arrive at $|\langle f,\mu_1\rangle-\langle f,\mu_2\rangle|\leq\frac{4}{5}\epsilon,$ contrary to the definition of $\epsilon$, which completes the proof.
	\end{proof}

	\subsection{Proofs of  Section \ref{Sec 3.1new} }\label{prf sec3.1}

    \paragraph{Proof of Proposition \ref{Prop T2*}:}
    \begin{proof}
		Assume, contrary to our claim, that $\{Q_t(x,\cdot)\}_{t\geq 0}$ is not tight for some $x\in\text{supp\;}\mu$. Then by \cite[Lemma 1]{KPS2010}, there exists a strictly increasing sequence of positive numbers $t_i$ going to infinity, a positive number $\epsilon$ and a sequence of compact sets $\{K_i\}$ such that
		\begin{align}
		Q_{t_i}(x,K_i)&\geq\epsilon,\quad\forall\,i\geq 1,\label{eq 5.5}\\ 
		\min\{\rho(x,y):x\in K_i,y\in K_j\}&\geq\epsilon,\quad\forall\, i\neq j. \label{eq 5.6}
		\end{align} 
		We will derive the assertion from the claim that there exist sequences $\{\bar{f}_n\}_{n\geq 1}\subset L_b(\mathcal{X})$, $\{\nu_n\}_{n\geq 1}\subset\mathcal{P}(\mathcal{X})$ and an increasing sequence of integers $\{m_n\}_{n\geq 1}$ such that supp $\nu_n\subset B(x,1/n)$, 
		\begin{align}
		\mathbf{1}_{K_{m_n}}\leq\bar{f}_n\leq\mathbf{1}_{K_{m_n}^{\epsilon/4}}\quad\text{and}\quad \|\bar{f}_n\|_{\rm Lip}&\leq 4/\epsilon\quad \forall\,n\geq 1,\label{eq 5.7}\\
        Q_t\nu_n(\cup_{i=n}^{\infty}K_{m_i}^{\epsilon/4})&\leq\epsilon/4\quad\forall\,t\geq 0,\label{eq 5.8}\\
        \limsup\limits_{t \to \infty}|\langle f_n,Q_t(x,\cdot)\rangle-\langle f_n,Q_t\nu_n\rangle|&\leq\epsilon/4,\label{eq 5.9}
		\end{align}
	where $f_1:=0,f_n:=\sum_{i=1}^{n-1}\bar{f}_i,n\geq 2.$  Let $f:=\sum_{i=1}^{\infty}\bar{f}_i,$ then by (\ref{eq 5.6}) and (\ref{eq 5.7}) $f$ is uniformly bounded with $\|f\|_{\infty}=1.$ Further, note that for any $x,y\in \mathcal{X}$ with $\rho(x,y)<\epsilon/8,$ we have $\bar{f}_i(x)\neq0,$ or $\bar{f}_i(y)\neq0$ for at most one $i.$  Thus
		\begin{equation*}
		|f(x)-f(y)|\leq 16\epsilon^{-1} \rho(x,y)
		\end{equation*} 
		and $f\in L_b(\mathcal{X}).$ Then it follows that
		\begin{equation}\label{eq 5.10}
		\begin{aligned}
		\langle f,Q_{t}(x,\cdot)\rangle -\langle f,Q_{t}\nu_n\rangle\geq Q_{t}(x,\cup_{i=n}^{\infty}K_{m_i})+\langle f_n,Q_{t}(x,\cdot)\rangle
		-\langle f_n,Q_{t}\nu_n\rangle-Q_{t}\nu_n(\cup_{i=n}^{\infty}K_{m_i}^{\epsilon/4}).
		\end{aligned}
		\end{equation}
		By (\ref{eq 5.5})
		\begin{equation}\label{eq 5.11}
		\limsup\limits_{t \to \infty}Q_t(x,\cup_{i=n}^{\infty}K_{m_i})\geq	\limsup\limits_{k \to \infty}Q_{t_{m_k}}(x,\cup_{i=n}^{\infty}K_{m_i})\geq\limsup\limits_{k \to \infty}Q_{t_{m_k}}(x,K_{m_k})\geq\epsilon.
		\end{equation}
		From (\ref{eq 5.8})-(\ref{eq 5.11}),  it follows that
		\begin{equation*}
		\limsup\limits_{t \to \infty}[\langle f,Q_{t}(x,\cdot)\rangle -\langle f,Q_{t}\nu_n\rangle]\geq\epsilon-\epsilon/4-\epsilon/4=\epsilon/2.
		\end{equation*}
		Hence there must be a sequence $y_n\in\text{supp\;}\nu_n$ such that
		\begin{equation*}
		\limsup\limits_{t \to \infty}|Q_tf(x)-Q_tf(y_n)|\geq\epsilon/2,
		\end{equation*}
		which contradicts the Ces\`aro eventual continuity of $\{P_t\}_{t\geq 0}$ at $x.$ This completes the proof.\par 
		\noindent $Proof\;of\;the\;claim.$ We accomplish this by induction on $n$. Let $n = 1$. Given $x\in\text{supp\;}\mu$, we have $\mu(B(x,\delta))>0$ for any $\delta>0$. Let $\nu_1\in\mathcal{P}(\mathcal{X})$ be defined by the formula
		\begin{equation*}
		\nu_1(B)=\mu(B|B(x,1)):=\frac{\mu(B\cap B(x,1))}{\mu(B(x,1))},\quad B\in\mathcal{B}(\mathcal{X}).
		\end{equation*}
		Since $\nu_1\leq\mu^{-1}(B(x,1))\mu$, from the fact that $\mu$ is invariant, it follows that the family $\{Q_t\nu_1\}_{t\geq 0}$ is tight. Then there exists some compact set $K$ such that 
		\begin{equation*}
		Q_t\nu_1(\mathcal{X}\setminus K)\leq\epsilon/4\quad\forall\,t\geq 0.
		\end{equation*} 
		Note, however, that $K\cap K_{i}^{\epsilon/4}\neq\emptyset$ for only finitely many $i'$s. As a result, there exists an integer $m_1$ such that
		\begin{equation*}
		Q_t\nu_1(\cup_{i=1}^{\infty}K_{m_1}^{\epsilon/4})\leq\epsilon/4\quad\forall\,t\geq 0,
		\end{equation*}
		Let $\bar{f}_1$ be an arbitrary Lipschitz function satisfying
		\begin{equation*}
		\mathbf{1}_{K_{m_1}}\leq\bar{f}_1\leq\mathbf{1}_{K_{m_1}^{\epsilon/4}}\quad\text{and}\quad \|\bar{f}_1\|_{\rm Lip}\leq 4/\epsilon.
		\end{equation*}
		Assume, now, that for a given $n\geq 1$, we have already constructed $\bar{f}_1,\dots,\bar{f}_n$, $\nu_1,\dots,\nu_n$ and $m_1,\dots,m_n$ satisfying the claim. In view of the Ces\`aro eventual continuity of $\{P_t\}_{t\geq 0}$, we can choose $\delta\in(0,1/(n+1))$ such that 
		\begin{equation*}
		\sup\limits_{y\in B(x,\delta)}\limsup\limits_{t \to \infty}|Q_tf_{n+1}(x)-Q_tf_{n+1}(y)|<\epsilon/4.
		\end{equation*}
		Further, let $\nu_{n+1}(\cdot):=\mu(\cdot|B(x,\delta)).$ Therefore, by the dominate convergence theorem,
		\begin{equation*}
		\limsup\limits_{t \to \infty}|\langle f_{n+1},Q_t(x,\cdot)\rangle-\langle f_{n+1},Q_t\nu_{n+1}\rangle|\leq\epsilon/4.
		\end{equation*}
		Finally, we let $\bar{f}_{n+1}$ be an arbitrary bounded, globally Lipschitz function satisfying (\ref{eq 5.7}).
    \end{proof}

    \paragraph{Proof of Theorem \ref{Thm 1}:}

    \begin{proof}  
    The proof is divided into four steps.

    \noindent $\mathbf{Step\;1.}$   Recall that $\{P_t\}_{t\geq 0}$ satisfies the  Ces\`aro e-property at  $x_0\in\mathcal{X}$ if for any $f\in L_b(\mathcal{X})$,
    \begin{equation}\label{Q-e}
        \limsup_{x\rightarrow z}\sup_{t\geq 0}|Q_tf(x)-Q_tf(x_0)|=0.
    \end{equation}
    We claim that \eqref{Q-e} holds if and only if 
      \begin{equation}\label{Q-E-e}
        \limsup_{x\rightarrow z,\,t\rightarrow\infty}|Q_tf(x)-Q_tf(x_0)|=0.
    \end{equation}
    Clearly, \eqref{Q-e} implies \eqref{Q-E-e}. For the opposite direction, we argue by contradiction. Otherwise, there exists $f\in L_b(\mathcal{X})$, $\epsilon>0$, a nonnegative sequence $\{t_n\}_{n\geq 1}$ bounded by some $T\geq 0$, and $\{x_n\}_{n\geq 1}$ converging to $x_0$ such that
    \begin{equation*}
        \limsup_{n\rightarrow\infty}|Q_{t_n}f(x_n)-Q_{t_n}f(x_0)|\geq \epsilon.
    \end{equation*}
    Taking the stochastic continuity into account, by \cite[Lemma 4]{KW2024}, it follows that
    \begin{align*}
        0<\epsilon\leq \limsup_{n\rightarrow\infty}|Q_{t_n}f(x_n)-Q_{t_n}f(x_0)|\leq\limsup_{n\rightarrow\infty}\sup_{t\in[0,T]}|P_{t}f(x_n)-P_{t}f(x_0)|=0,
    \end{align*}
    which is impossible.

    \noindent $\mathbf{Step\;2.}$ It remains to verify \eqref{Q-E-e} holds for any $x\in{\rm Int}_{\mathcal{X}}(\supp\mu)$ and $f\in L_b(\mathcal{X})$.  Assume that, contrary to our claim, $\{P_t\}_{t\geq 0}$ does not satisfy the Ces\`aro e-property on $\mathcal{X}$. Then there would exist $x_0\in\mathcal{X}$, $\epsilon>0$, $f\in L_b(\mathcal{X})$, and a sequence $\{x_n\}_{n\geq 1}$ converging to $x_0$ such that
		\begin{equation*}
		\limsup\limits_{n\rightarrow\infty}|Q_{t_n}f(x_n)-Q_{t_n}f(x_0)|>3\epsilon>0.
		\end{equation*}		

    Meanwhile, by the Ces\`aro eventual continuity on $\mathcal{X}$, there exists $\delta>0$ such that
    \begin{equation*}
        \limsup\limits_{t\rightarrow\infty}|Q_tf(x)-Q_tf(x_0)\rangle|\leq \epsilon/2\quad\forall\,x\in B(x_0,\delta).
    \end{equation*}
    In addition, $B(x_0,\delta)\subset\supp\mu=\mathcal{X}$.  Let $Y:=\overline{B(x_0,\delta)}$ and set 
    \begin{equation*}
        Y_n:=\{x\in Y:|Q_tf(x)-Q_tf(x_0)|\leq\epsilon/2\quad\forall\,t\geq n\}\quad\text{for }n\in\N.
    \end{equation*}
    Then $Y=\cup_{n\geq 1}Y_n$. By the Baire category theorem, there exists $T_*\in\N$ such that $\text{Int}_{\mathcal{X}}(Y_{T_*})\neq \emptyset$. Thus there exits an open ball $B(z,2r)\subset Y_{T_*}$ such that
    \begin{equation}\label{T*}
       |Q_tf(x)-Q_tf(x_0)|\leq\epsilon/2\quad\forall\,x\in B(z,2r),\;t\geq T_*.
    \end{equation}

	\noindent $\mathbf{Step\;3.}$ By Theorem \ref{Prop T3*}, it follows that there exists $\alpha>0$ such that for any $\nu\in\mathcal{P}(\mathcal{X})$ satisfying $\supp\nu\subset \supp\mu$,

        \begin{equation}\label{alpha}
            \liminf\limits_{t\rightarrow\infty}(Q_t\nu)(B(z,r))>\alpha.
        \end{equation} 
 
        Let $k\geq 1$ be such that $2(1-\alpha)^k\|f\|_{\infty}<\epsilon.$ By induction we are going to define two sequences of measures $\{\nu_i^{x_0}\}_{i=1}^k,\{\mu_i^{x_0}\}_{i=1}^k,$ and a sequence of positive numbers $\{s_{i}\}_{i=1}^k$ in the following way: using \eqref{alpha}, let
		$s_1>0$ be such that	
		\begin{equation*}
		P_{s_1}\delta_{x_0}(B(z,r))>\alpha.
		\end{equation*}
        
        By the Feller property, take $r_1<r$ such that $P_{s_1}\delta_{x_0}(B(z,r_1))>\alpha$ and $P_{s_1}\delta_{x_0}(\partial B(z,r_1))=0$ and set
		\begin{center}
			$\nu_1^{x_0}(\cdot) = \dfrac{P_{s_1}\delta_{x_0}(\cdot \cap B(z,r_1))}{P_{s_1}\delta_{x_0}(B(z,r_1))},\quad$
			$\mu_1^{x_0}(\cdot) = \dfrac{1}{1-\alpha}(P_{s_1}\delta_{x_0}(\cdot)-\alpha\nu_1^{x_0}(\cdot)).$
		\end{center}
        Notice that $\supp\nu_1^{x_0}\subset B(z,r)\subset\supp\mu$, and $\supp\mu_1^{x_0}\subset\supp\mu$ by \cite[Lemma 3.7]{LL2024}.
  
		Assume that we have done it for $i = 1,\dots , l,$ for some $l < k.$   Again using \eqref{alpha}, now let $s_{l+1}$ be such that
		$P_{s_{l+1}}\mu_l^{x_0}(B(z,r))>\alpha$. 	Choose $r_{l+1}<r$ such that $P_{s_{l+1}}\mu_l^{x_0}(B(z,r_{l+1}))>\alpha$, $P_{s_{l+1}}\mu_l^{x_0}(\partial B(z,r_{l+1}))=0$ and set
		
		\begin{center}
			$\nu_{l+1}^{x_0}(\cdot) = \dfrac{P_{s_{l+1}}\mu_l^{x_0}(\cdot \cap B(z,r_{l+1}))}{	P_{s_{l+1}}\mu_l^{x_0}(B(z,r_{l+1}))},\quad$
			$\mu_{l+1}^{x_0}(\cdot) = \dfrac{1}{1-\alpha}(P_{s_{l+1}}\mu_l^{x_0}(\cdot)-\alpha\nu_{l+1}^{x_0}(\cdot)).$
		\end{center}
		Then this implies that
		\begin{equation*}
		\begin{aligned}
		P_{s_1+\dots+s_k}\delta_{x_0}(\cdot) =&\,\alpha P_{s_2+\dots+s_k}\nu_1^{x_0}(\cdot)+\alpha(1-\alpha) P_{s_3+\dots+s_k}\nu_2^{x_0}(\cdot)+ \cdots \\
		&\,+\alpha(1-\alpha)^{k-1}\nu_k^{x_0}(\cdot)+(1-\alpha)^k \mu_k^{x_0}(\cdot),
		\end{aligned}
		\end{equation*}
        and that $\supp\nu_{i}^{x_0}\subset B(z,r)$, $\supp\mu_{i}^{x_0}\subset \supp\mu$ for any $1\leq i\leq k$ by construction. In addition, thanks to the Feller property, there exists $N_1$ sufficiently large such that 
       \begin{equation*}
		P_{s_1}\delta_{x_n}(B(z,r))>\alpha \quad\forall\,n\geq N_1. 
	\end{equation*}
      Therefore, choose $r_{1,n}\in(0,r)$ such that $P_{s_1}\delta_{x_n}(B(z,r_{1,n}))>\alpha$ and $P_{s_1}\delta_{x_n}(\partial B(z,r_{1,n}))=0$ and set $\nu_1^{x_n},\,\mu_1^{x_n}$  as
      \begin{center}
			$\nu_1^{x_n}(\cdot) = \dfrac{P_{s_1}\delta_{x_n}(\cdot \cap B(z,r_{1,n}))}{P_{s_1}\delta_{x_n}(B(z,r_{1,n}))},\quad$
			$\mu_1^{x_n}(\cdot) = \dfrac{1}{1-\alpha}(P_{s_1}\delta_{x_n}(\cdot)-\alpha\nu_1^{x_n}(\cdot))$  for $n\geq N_1$.
		\end{center}
  
    Furthermore, we continue this procedure to construct the sequence $\{\nu_i^{x_n}\}_{i=1}^{k},\{\mu_i^{x_n}\}_{i=1}^{k}$ for $n$ sufficiently large.   	Assume that we have done it for $i = 1,\dots , l,$ for some $l < k.$ By the Feller property, there exists $N_l$ sufficiently large such that
		$P_{s_{l+1}}\mu_l^{x_n}(B(z,r))>\alpha$ for $n\geq N_l$. 	Choose $r_{l+1,n}<r$ such that $P_{s_{l+1}}\mu_l^{x_n}(B(z,r_{l+1,n}))>\alpha$ and $P_{s_{l+1}}\mu_l^{x_n}(\partial B(z,r_{l+1,n}))=0$ and set
		\begin{center}
			$\nu_{l+1}^{x_n}(\cdot) = \dfrac{P_{s_{l+1}}\mu_l^{x_n}(\cdot \cap B(z,r_{l+1,n}))}{	P_{s_{l+1}}\mu_l^{x_n}(B(z,r_{l+1,n}))},\quad$
			$\mu_{l+1}^{x_n}(\cdot) = \dfrac{1}{1-\alpha}(P_{s_{l+1}}\mu_l^{x_n}(\cdot)-\alpha\nu_{l+1}^{x_n}(\cdot))$  $\forall\,n\geq N_l$.
		\end{center} 
     
    In summary, for any $n\geq N_*:=\max\{N_l:1\leq l\leq k\}$ and  $t\geq s_1+\cdots+s_k:=S_*$, we have
		\begin{equation}\label{P-decompose}
		\begin{aligned}
		P_t\delta_{x_n}(\cdot)&= \alpha P_{t-s_1}\nu_1^{x_n}(\cdot)+\alpha(1-\alpha) P_{t-s_1-s_2}\nu_2^{x_n}(\cdot)+ \cdots \\
		&\quad+\alpha(1-\alpha)^{k-1}P_{t-s_1-\cdots-s_k}\nu_k^{x_n}(\cdot)+(1-\alpha)^k P_{t-s_1-\cdots-s_k}\mu_k^{x_n}(\cdot),
		\end{aligned}
		\end{equation}
		where  $\text{supp }{{\nu}_i^{x_n}}\subset B(z,r),\,i=1,\dots,k.$\par

	\noindent $\mathbf{Step\;4.}$	Consequently, by \eqref{P-decompose}, for any $n\geq N_*$ and  $t\geq S_*$, $Q_t\delta_{x_n}$ can be decomposed as 
     \begin{equation}\label{Q-decompose}
    \begin{aligned}
     tQ_t\delta_{x_n}(\cdot)&=\int_{0}^{S_*}P_s\delta_{x_n}\d s+{\alpha}\int_{S_*}^{t} P_{s-s_1}\nu_1^{x_n}(\cdot)\d s+ \cdots\\
     &\quad+\alpha(1-\alpha)^{k-1}\int_{S_*}^{t}P_{s}\nu_k^{x_n}(\cdot)\d s+ {(1-\alpha)^k}\int_{S_*}^{t}P_{s}\mu_k^{x_n}(\cdot)\\
     &={S_*}Q_{S_*}\delta_{x_n}(\cdot)+{\alpha}((t-s_1)Q_{t-s_1}\nu_1^{x_n}-(S_*-s_1)Q_{S_*-s_1}\nu_1^{x_n})+ \cdots\\
     &\quad+\alpha(1-\alpha)^{k-1}(tQ_t\nu_k^{x_n}-S_*Q_{S_*}\nu_k^{x_n})+(1-\alpha)^{k}(tQ_t\mu_k^{x_n}-S_*Q_{S_*}\mu_k^{x_n}).
     \end{aligned}
    \end{equation}
    The same decomposition also holds for $Q_t\delta_{x_0}$.

    Meanwhile, using \eqref{T*}, it follows that
		\begin{equation*}
		|\langle f,Q_t\nu_i^{x_n}\rangle-Q_tf(x_0)|\leq\int_{\mathcal{X}} |Q_tf(y)-Q_tf(x_0)|\nu_i^{x_n}(\d y)\leq \epsilon/2\quad\forall\,n\geq N_*,\,i=1,\dots,k,\,t\geq T_*.
		\end{equation*}
		The same inequality also holds for $\nu_i^{x_0},i=1,\dots,k$. Thus it follows that
		\begin{equation*}
		|\langle f,Q_t\nu_i^{x_n}\rangle-\langle f,Q_t\nu_i^{x_0}\rangle|=|\langle Q_tf-Q_tf(x_0),\nu_i^{x_n}\rangle-\langle Q_tf-Q_tf(x_0),\nu_i^{x_0}\rangle|\leq\epsilon,
		\end{equation*}
	for any $n\geq N_*$, $i=1,\dots,k$, $t\geq T_*$. Thus we obtain that
		\begin{equation}\label{Q-decompose2}
		\limsup\limits_{n\rightarrow\infty}|\langle f,Q_{t_n}\nu_i^{x_n}\rangle-\langle f,Q_{t_n}\nu_i^{x_0}\rangle|\leq\epsilon\quad\forall\,i=1,\dots,k.
		\end{equation}
  
	Collecting \eqref{Q-decompose}, \eqref{Q-decompose2}, we conclude that
		\begin{equation*}
		\begin{aligned}
		3\epsilon&<\limsup\limits_{n\rightarrow\infty}|Q_{t_n}f(x_n)-Q_{t_n}f(x_0)|=\limsup\limits_{n\rightarrow\infty}|\langle f,Q_{t_n} \delta_{x_n}\rangle-\langle f,Q_{t_n} \delta_{x_0}\rangle| \\
		&\leq \alpha\limsup\limits_{n\rightarrow\infty}|\langle f, Q_{t_n}\nu_1^{x_n} \rangle-\langle f, Q_{t_n}\nu_1^{x_0} \rangle|+\cdots\\
        &\quad+\alpha(1-\alpha)^{k-1}\limsup\limits_{n\rightarrow\infty}|\langle f, Q_{t_n}\nu_{k-1}^{x_n}\rangle-\langle f, Q_{t_n}\nu_{k-1}^{x_0} \rangle| +2(1-\alpha)^k\|f\|_{\infty}\\
		&\leq (\alpha+\cdots+\alpha(1-\alpha)^l)\epsilon+\epsilon\\
		&\leq 2\epsilon,
		\end{aligned}
		\end{equation*}
		which is impossible. This completes the proof.

    \end{proof}

    \subsection{Proofs of Section \ref{Sec 3.2}}\label{prf sec3.2}

    Before we prove Proposition \ref{Prop Exist*}, let us first quote several lemmas from \cite{GL2015}, to which we will refer in the main proof.
	\begin{lemma}\label{Lemma I1}
		Let $\{A_n^\epsilon\}_{n\geq 1}$ be a sequence of mutually disjoint $\epsilon$-neighborhoods of $A_n.$ Then, for any
		compact set $C$, there exists $N>0$ such that
        \begin{equation*}
            C\cap A_n^{\epsilon/2}=\emptyset\quad\forall\,n\geq N.
        \end{equation*}
	
	\end{lemma} 
	\begin{lemma}\label{Lemma I2}
		Let $\{A_n^\epsilon\}_{n\geq 1}$ be a sequence of mutually disjoint $\epsilon$-neighborhoods of compact sets $A_n.$ Let $x\in \mathcal{X}$ and $t_i\geq 0,\;t_i\rightarrow\infty.$ Then for any $\eta>0$, there exists $N>0$ such that 
		\begin{equation*}
		\liminf\limits_{i\rightarrow\infty}Q_{t_i}(x,A_n^{\epsilon/4})\leq\eta\quad\forall\,n\geq N.
		\end{equation*}\par
	\end{lemma}
	
	\begin{lemma}\label{Lemma I3}
		Let $A_n\in\mathcal{B}(\mathcal{X})$ be a sequence of mutually disjoint sets. Let $x\in \mathcal{X}$ and $t_i\geq 0,\;t_i\rightarrow\infty.$ Then for any $\epsilon>0$, there exists $N>0$ such that
		\begin{equation*}
		\liminf\limits_{i\rightarrow\infty}Q_{t_i}(x,A_N)\leq\epsilon.
		\end{equation*}
		Hence, there exists a subsequence $\{t_{i_k}\}_{k\geq 0}$ with
		\begin{equation*}
		\limsup\limits_{k\rightarrow\infty}Q_{t_{i_k}}(x,A_N)\leq\epsilon.
		\end{equation*}	\par
	\end{lemma}

	\hspace*{\fill} \par
	
	We are now in a position to prove Proposition \ref{Prop Exist*}. The basic idea is mainly from \cite[Theorem 1.8.]{GL2015}, however, the delicate construction is unable to be applied directly to continuous setting, so here we take contradiction argument to finish the proof.
	
	\paragraph{Proof of Proposition \ref{Prop Exist*}:}
	\begin{proof}  We claim that $z\in\mathcal{T}.$  Assume, contrary to our claim that $\{Q_t(z,\cdot)\}_{t\geq 0}$ is not tight, then there exists some $\epsilon>0,$ a sequence of compact sets $\{K_i\}_{i\geq 1}$ and an increasing sequence of real numbers $\{t_i\}_{i\geq 1}$, $t_i\rightarrow\infty,$ such that
		\begin{align}
		Q_{t_i}(z,K_i)&\geq \epsilon,\quad\forall\,i\geq 1,\label{eq 5.12}\\	
		K_i^{\epsilon/2}\cap K_j^{\epsilon/2}&=\emptyset,\quad \forall\,i,j\geq 1,\;i\neq j.\label{eq 5.13}
		\end{align}\par
		
	\noindent 	$\mathbf{Step\;1}.$ We shall select a subsequence from the above data such that it is so sparse that contradicts (\ref{eq 5.12}). This step will be cut into three parts.\par
		Define
		\begin{equation*}
		f_i(y)=\rho(y,\mathcal{X}\setminus K_i^{\epsilon/4})/(\rho(y,\mathcal{X}\setminus K_i^{\epsilon/4})+\rho(y,K_i)),
		\end{equation*}
		which fulfills that $\|f_i\|_{\rm Lip}\leq 8/\epsilon,$ and $\mathbf{1}_{K_i}\leq f_i\leq\mathbf{1}_{K_i^{\epsilon/4}}.$\par
	\noindent 	$\mathbf{Part\;1.1}.$ Denote $z_0=z,\;s_0=r$ and $B_0=B(z_0,s_0).$ By (\ref{eq 3.4}), there exists some sequence $\{s_{i,0}\}_{i\geq 1}$ such that
		\begin{equation*}
		\alpha_0:=\lim\limits_{i\rightarrow\infty} Q_{s_{i,0}}(z,B_0)>0.
		\end{equation*}\par
		Lemma \ref{Lemma I2} yields that there exists $j_0$ sufficiently large such that
		\begin{equation}\label{eq 5.14}
		\begin{aligned}
		\epsilon\alpha_0/16&\geq\liminf\limits_{i\rightarrow\infty}Q_{s_{i,0}}(z,K_{j_0}^{\epsilon/4})=\liminf\limits_{i\rightarrow\infty}\int_{\mathcal{X}} Q_{t_{j_0}}(y,K_{j_0}^{\epsilon/4})Q_{s_{i,0}}(z,\d y)\\
		&\geq\liminf\limits_{i\rightarrow\infty}\int_{B_0} Q_{t_{j_0}}f_{j_0}(y)Q_{s_{i,0}}(z,\d y).
		\end{aligned}
		\end{equation}
		Due to the Feller property, define an open subset in $B_0$ as
		\begin{center}
			$A_0=\{y\in B_0:Q_{t_{j_0}}f_{j_0}(y)<\epsilon/8\}.$
		\end{center}
		It follows from (\ref{eq 5.14}) that
		\begin{equation*}
		\epsilon\alpha_0/16\geq\liminf\limits_{i\rightarrow\infty}\int_{B_0-A_0} Q_{t_{j_0}}f_{j_0}(y)Q_{s_{i,0}}(z,\d y),
		\end{equation*}
		which implies
		\begin{equation*}
		\liminf\limits_{i\rightarrow\infty}Q_{s_{i,0}}(z,B_0-A_0)\leq\alpha_0/2,
		\end{equation*}
		and thus
		\begin{equation}\label{eq 5.15}
		\limsup\limits_{i\rightarrow\infty}Q_{s_{i,0}}(z,A_0)\geq\alpha_0/2.
		\end{equation}
		So $A_0$ is nonempty, which contains a ball $B_1$ of radius less than $r/2,$
		\begin{equation*}
		Q_{t_{j_0}}f_{j_0}(y)\leq\epsilon/8,\quad\forall\,y\in \overline{B_1}.
		\end{equation*}

		Furthermore, by (\ref{eq 5.15}), we may fix some  $s>0 $ such that
		\begin{equation*}
		P_s(z,B_1)>0.
		\end{equation*} 
		By the Feller property, there exists $\delta>0$ such that for any $y\in B(z,\delta),$
		\begin{equation*}
		P_s(y,B_1)\geq\frac{1}{2}P_s(z,B_1)>0.
		\end{equation*} 
		From condition \eqref{eq 3.4}, we derive that
		\begin{equation*}
		\limsup\limits_{t\rightarrow\infty}Q_t(z,B_1)\geq\frac{1}{2}P_s(z,B_1)\cdot\limsup\limits_{t\rightarrow\infty}Q_t(z,B(z,\delta))>0.
		\end{equation*} 
		Inductively, $\forall\,k\geq 0,$ we can determine  $B_{k+1}\subset B_k$ of radius less than $r/2^{k+1}$ and $j_k\in\mathbb{N}$ such that 
		\begin{equation*}
		Q_{t_{j_k}}f_{j_k}(y)\leq\epsilon/8,\quad\forall\,y\in \overline{B_{k+1}}.
		\end{equation*} 
		Therefore, we can find a common $y_0\in\cap\overline{B_{k}}$ such that
		\begin{equation}\label{eq 5.16}
		Q_{t_{j_k}}f_{j_k}(y_0)\leq\epsilon/8,\quad\forall\,k\geq 0.
		\end{equation}
		\par
	\noindent 	$\mathbf{Part\;1.2}.$ By Lemma \ref{Lemma I2}, there exists a  subsequence $\{j_{k,1}\}_{k\geq 0 }\subset\{j_{k}\}_{k\geq 0 }$ such that
		\begin{equation*}
		\limsup\limits_{k\rightarrow\infty}Q_{t_{j_{k,1}}}f_{j_0}(y_0)\leq\epsilon/8.
		\end{equation*}
		For the same reason, there exists $\{j_{k,2}\}_{k\geq 0 }\subset\{j_{k,1}\}_{k\geq 0 }$ such that
		\begin{equation*}
		\limsup\limits_{k\rightarrow\infty}Q_{t_{j_{k,2}}}f_{j_{0,1}}(y_0)\leq\epsilon/(8\cdot 2).
		\end{equation*}
		By induction, we have $\{j_{k,l+1}\}_{k\geq 0 }\subset\{j_{k,l}\}_{k\geq 0 }$ satisfying
		\begin{equation*}
		\limsup\limits_{k\rightarrow\infty}Q_{t_{j_{k,l+1}}}f_{j_{0,l}}(y_0)\leq\epsilon/(8\cdot 2^l).
		\end{equation*}
		For simplicity of notation, still use $j_l$ instead of $j_{0,l}.$ Recall (\ref{eq 5.16}), we obtain 
		\begin{equation}\label{eq 5.17}
		Q_{t_{j_l}}f_{j_l}(y_0)\leq\epsilon/8,\;	\limsup\limits_{k\rightarrow\infty}Q_{t_{j_{k}}}f_{j_{l}}(y_0)\leq\epsilon/(8\cdot 2^l),\quad\forall\,l\geq 0.
		\end{equation}
		\par
	\noindent 	$\mathbf{Part\;1.3}.$ Noting that $K_j^{\epsilon/4}$ are disjoint mutually, there exists a big $u$ such that
		\begin{center}
			$Q_{t_{j_0}}\sum_{k\geq u}f_{j_k}(y_0)\leq\epsilon/8,$
		\end{center}
		Combining with  \eqref{eq 5.16} yields
		\begin{equation*}
		Q_{t_{j_0}}(f_{t_{j_0}}+\sum_{k\geq u}f({j_k}))(y_0)\leq\epsilon/4.
		\end{equation*}
		Let $\hat{j}_0 = j_0$ and $\hat{j}_1 = j_u.$ For the same reason, there exists $v>u$ such that
		\begin{equation*}
		Q_{t_{\hat{j}_1}}(f_{t_{\hat{j}_1}}+\sum_{k\geq v}f({j_k}))(y_0)\leq\epsilon/4.
		\end{equation*}
		Let $\hat{j}_2=j_v.$ By induction, we have subsequnce $\{\hat{j}_l\}_{l\geq 0}\subset\{j_l\}_{l\geq 0}$ such that for any $l\geq 0,$
		\begin{equation*}
		Q_{t_{\hat{j}_l}}\sum_{k\geq l}f_{\hat{j}_k}(y_0)\leq\epsilon/4.
		\end{equation*}
		For simplicity we still use $j_l$ instead of $\hat{j}_l.$ Combining with the second inequality in (\ref{eq 5.17}), we obtain
		\begin{equation}\label{eq 5.18}
		Q_{t_{j_l}}\sum_{k\geq l}f_{j_k}(y_0)\leq\epsilon/4,\;	\limsup\limits_{k\rightarrow\infty}Q_{t_{j_{k}}}f_{j_{l}}(y_0)\leq\epsilon/(8\cdot 2^l),\quad\forall\,l\geq 0.
		\end{equation}
		\par
	\noindent 	$\mathbf{Part\;1.4}.$ Based on the second inequality in (\ref{eq 5.18}), there exists a big $u$ such that
		\begin{equation*}
		Q_{t_{j_u}}f_{j_0}(y_0)\leq\epsilon/4,\;	\limsup\limits_{k\rightarrow\infty}Q_{t_{j_{k}}}(f_{j_{0}}+f_{j_{u}})(y_0)\leq(1+2^{-1})\epsilon/8.
		\end{equation*}
		For the same reason, there exists a sufficiently large $v>u$ such that
		\begin{equation*}
		Q_{t_{j_u}}(f_{j_0}+f_{j_u})(y_0)\leq\epsilon/4,\;	\limsup\limits_{k\rightarrow\infty}Q_{t_{j_{k}}}(f_{j_0}+f_{j_u}+f_{j_v})(y_0)\leq(1+2^{-1}+2^{-2})\epsilon/8.
		\end{equation*}
		Let $\check{j}_0=j_0,\check{j}_1=j_u,\check{j}_2=j_v.$ By induction, we have $\{\check{j}_l\}_{l\geq 0}\subset\{j_l\}_{l\geq 0}$ satisfying
		\begin{equation*}
		Q_{t_{\check{j}_l}}\sum_{0\leq k<l}f_{\check{j}_k}(y_0)\leq\epsilon/4.
		\end{equation*}
		Combining with the first inequality in (\ref{eq 5.18}) yields
		\begin{equation*}
		Q_{t_{\check{j}_l}}\sum_{k\geq 0}f_{\check{j}_k}(y_0)\leq\epsilon/2\quad\forall\,l\geq 0.
		\end{equation*}	
		For simplicity of notation, still use $j_l$ instead of $\check{j}_l,$ we have
		\begin{equation}\label{eq 5.19}
		Q_{t_{j_l}}\sum_{k\geq 0}f_{j_k}(y_0)\leq\epsilon/2\quad\forall\,l\geq 0.
		\end{equation}	
		\par
	\noindent 	$\mathbf{Step\;2}.$ Write $j_{k,0} = j_k.$ Let us repeat Step 1 by substituting $s_0$ for $s_1 = r/2,$ then obtain some
		$\{j_{k,0}\}_{k\geq 0}\subset\{j_k\}_{k\geq 0}$ and $y_1\in B(z, s_1)$ such that (similar to (\ref{eq 5.19}))
		\begin{equation*}
		Q_{t_{j_{l,1}}}\sum_{k\geq 		0}f_{j_{k,1}}(y_1)\leq\epsilon/2\quad\forall\,l\geq 0.
		\end{equation*}	
		By induction, we obtain the $p$-th subsequence $\{j_{k,p}\}_{k\geq 0}\subset\{j_{k,p-1}\}_{k\geq 0}$ and some $y_p\in\overline{B(z,s_p)}$ for $s_o=r/2^p$ such that
		\begin{equation*}
		Q_{t_{j_{l,p}}}\sum_{k\geq 		0}f_{j_{k,p}}(y_p)\leq\epsilon/2\quad\forall\,l\geq 0.
		\end{equation*}	
		Denote $\tilde{j}_p=j_{0,p},$ it follows that
		\begin{equation}\label{eq 5.20}
		Q_{t_{\tilde{j}_{l}}}\sum_{k\geq 		p}f_{\tilde{j}_{k}}(y_p)\leq\epsilon/2\quad\forall\,l\geq p.
		\end{equation}
		Note that $y_p\rightarrow z.$
		\par
	\noindent $\mathbf{Step\;3}.$ Let $j_0^*=\tilde{j}_0.$ The eventual continuity yields some $r_0^*$ such that for any $y\in B(z,r_0^*)$
		\begin{equation*}
		\limsup\limits_{t\rightarrow\infty}|Q_tf_{j_0^*}(z)-Q_tf_{j_0^*}(y)|\leq\epsilon/8.
		\end{equation*}
		Due to (\ref{eq 5.20}), choose $y_p\in B(z,r^*_0),$ denoted by $y^*_0.$ And for this $p,$ choose some $j^*_1\in\{\tilde{j}_k\}_{k\geq 0}$ with $j_1^*\geq\tilde{j}_p.$\par
		By induction, if we have $j_0^*,j_1^*,\dots,j_u^*,$ there exists $r_u^*$ such that for any $y\in B(z,r_u^*),$
		\begin{equation}\label{eq 5.21}
		\limsup\limits_{t\rightarrow\infty}|Q_t\sum_{l\leq u}f_{j_l^*}(z)-Q_t\sum_{l\leq u}f_{j_l^*}(y)|\leq\epsilon/8.
		\end{equation}	
		Choose $y_q\in B(z,r_u^*),$ denoted by $y_u^*,$ and then $j_{u+1}^*\in\{\tilde{j}_k\}$ with $j_{u+1}^*\geq\tilde{j}_q.$\par
		Consider the subsequence $\{j_k^*\}_{k\geq 0}\subset\{\tilde{j}_k\}_{k\geq 0}.$ Define $g=\sum_{k\geq 0} f_{j_k^*}\in L_b(\mathcal{X})$. Again, the eventual continuity yields some $r^*$ such that for any $y\in B(z,r^*),$
		\begin{equation}\label{eq 5.22}
		\limsup\limits_{t\rightarrow\infty}|Q_tg(z)-Q_tg(y)|\leq\epsilon/8.
		\end{equation}
		Fix $y_u^*\in B(z,r^*)$ for some big $u,$ denote $g_0=\sum_{0\leq l\leq u}f_{j_l^*}$ and $g_1=g-g_0.$ Then we have
		\begin{equation*}
		\begin{aligned}
		&\limsup\limits_{k\rightarrow\infty}|Q_{t_{j_k^*}}g_1(z)-Q_{t_{j_k^*}}g_1(y)|\\
		&\leq	
		\limsup\limits_{k\rightarrow\infty}|Q_{t_{j_k^*}}g(z)-Q_{t_{j_k^*}}g(y_u)|+	
		\limsup\limits_{k\rightarrow\infty}|Q_{t_{j_k^*}}g_0(z)-Q_{t_{j_k^*}}g_0(y_u)|,
		\end{aligned}
		\end{equation*}
		Owing to (\ref{eq 5.20})- (\ref{eq 5.22}), we have 
		\begin{equation}\label{eq 5.23}
		\begin{aligned}
		&\limsup\limits_{k\rightarrow\infty}Q_{t_{j_k^*}}g_1(z)\\
		&\leq\limsup\limits_{k\rightarrow\infty}Q_{t_{j_k^*}}g_1(y_u)+
		\limsup\limits_{k\rightarrow\infty}|Q_{t_{j_k^*}}g(z)-Q_{t_{j_k^*}}g(y_u)|+	
		\limsup\limits_{k\rightarrow\infty}|Q_{t_{j_k^*}}g_0(z)-Q_{t_{j_k^*}}g_0(y_u)|\\
		&\leq\epsilon/2+\epsilon/8+\epsilon/8\leq 3\epsilon/4.
		\end{aligned}
		\end{equation}
		On the other hand, recall (\ref{eq 5.12}) , we have for any $k>u$ 
		\begin{equation*}
		\epsilon\leq Q_{t_{j_k^*}}(z,K_{j_k^*})\leq Q_{t_{j_k^*}}f_{j_k^*}(z)\leq Q_{t_{j_k^*}}g_1(z),
		\end{equation*}
		which clearly contradicts (\ref{eq 5.23}). This completes the proof. 
    \end{proof}

    \paragraph{Proof of Theorem \ref{Thm 3}:}
    
    \begin{proof}  
    The implication that $(\runum{1})\Rightarrow(\runum{2})$ follows from \cite[Corollary 5.3]{SW2012}.  It remains to show that $(\runum{2})\Rightarrow(\runum{1})$. Taking Theorem \ref{Thm 2} into consideration, we only need to show $\mathcal{T}=\mathcal{X}$. It suffices to show that for any $x\in \mathcal{X}$ and $\epsilon>0,$ there exists some compact set $K\subset \mathcal{X}$ such that
		\begin{equation*}
		\liminf\limits_{t\rightarrow\infty}Q_t(x,K^\epsilon)\geq 1-\epsilon.
		\end{equation*}\par 
		The proof is divided into four steps:
  
      \noindent $\mathbf{Step\;1.}$ We first show that we can replace condition \eqref{eq 3.7} by a much stronger lower bound condition:
		\begin{equation}\label{eq 3.8}
		\inf\limits_{x\in \mathcal{X}}\liminf\limits_{t\rightarrow\infty}Q_t(x,B(z,\epsilon))>0\quad\forall\,\epsilon>0.
		\end{equation}\par 
		Fix $\epsilon>0,$ let $f(x)=\max\{1-\epsilon\rho(x,z),0\},$ then $\|f\|_{\infty}=1,\;\|f\|_{\rm Lip}\leq\epsilon^{-1},$ so $f\in L_b(\mathcal{X}).$  Further, $f$ satisfies that $\mathbf{1}_{\{B(z,\epsilon/2)\}}/2\leq f\leq \mathbf{1}_{\{B(z,\epsilon)\}}.$ \par 
		Let $\alpha=\mu_*(B(z,\epsilon/2))>0.$ Since $\{Q_(z,\cdot)\}_{t\geq 0}$ weakly converges to $\mu_*$ as $t\rightarrow\infty$, we have 
		\begin{equation}\label{eq 3.9}
		\liminf\limits_{t\rightarrow\infty}Q_t(z,B(z,\epsilon/2))\geq\alpha>0.
		\end{equation}
		Furthermore, $\{P_t\}_{t\geq 0}$ is Ces\`aro eventually continuous at $z,$ we may choose $\delta>0$ such that $\forall\,x\in B(z,\delta),$
		\begin{equation}\label{eq 3.10}
		\limsup\limits_{t\rightarrow\infty}|Q_tf(x)-Q_tf(z)|<\alpha/4.
		\end{equation}
		Combine (\ref{eq 3.9}) and (\ref{eq 3.10}), $\forall\,x\in B(z,\delta)$,
		\begin{equation}\label{eq 3.11}
		\begin{aligned}
		\liminf\limits_{t\rightarrow\infty}Q_t(x,B(z,\epsilon))\geq\frac{1}{2} \liminf\limits_{t\rightarrow\infty}Q_t(z,B(z,\epsilon/2))-\limsup\limits_{t\rightarrow\infty}|Q_tf(x)-Q_tf(z)|\geq \alpha/4.
		\end{aligned}
		\end{equation}
  
		By (\ref{eq 3.7}), there exists $\beta>0$ such that 
		\begin{equation*}
		\inf\limits_{x\in \mathcal{X}}\limsup\limits_{t\rightarrow\infty}Q_t(x,B(z,\delta))\geq\beta.
		\end{equation*}\par 
		Fix $x\in \mathcal{X}.$ Let $T>0$ be such that $Q_T(x,B(z,\delta))\geq\beta$. Define
		\begin{equation*}
		\nu:=\frac{Q_T\delta_{x}(\cdot\cap B(z,\delta))}{Q_T(x,B(z,\delta))},
		\end{equation*}
		then $\nu\in\mathcal{P}(\mathcal{X})$ with $\text{supp }\nu\subset B(z,\delta)$ and $Q_T(x,\cdot)\geq\beta\nu.$\par 
		Fatou's lemma implies that
		\begin{equation*}
		\liminf\limits_{t\rightarrow\infty}Q_t\nu(B(z,\epsilon))\geq\int_{X}	\liminf\limits_{t\rightarrow\infty}Q_t(y,B(x,\epsilon))\nu(\d y)\geq\alpha/4.
		\end{equation*}\par 
		Therefore, using \cite[Lemma 2]{KPS2010}, we have
		\begin{equation*}
		\liminf\limits_{t\rightarrow\infty}Q_t(x,B(x,\epsilon))=\liminf\limits_{t\rightarrow\infty}Q_tQ_T(x,B(x,\epsilon))\geq\alpha\beta/4>0\quad\forall\,x\in \mathcal{X}.
		\end{equation*}
  
        \noindent $\mathbf{Step\;2.}$  For any $\epsilon>0,$ there exists $\delta>0$ such that for any $\mu\in\mathcal{P}(\mathcal{X})$ with $\mu(B(z,\delta))=1$ there exists a compact set $K$ for which
		\begin{equation}\label{eq 3.12}
		\liminf\limits_{t\rightarrow\infty}Q_t\mu(K^\epsilon)\geq 1-\epsilon.
		\end{equation}\par 
		Fix $\epsilon>0.$ Noting that  $\{Q_t(z,\cdot)\}_{t\geq 0}$ is tight, we can find a compact set $K$ such that 
		\begin{equation*}
		\liminf\limits_{t\rightarrow\infty}Q_t(z,K^{\epsilon/2})\geq 1-\epsilon/2.
		\end{equation*}
		Now let $f(x)=\rho(x,\mathcal{X}\setminus K^\epsilon)/(\rho(x,K^{\epsilon/2})+\rho(x,\mathcal{X}\setminus K^\epsilon))\in L_b(\mathcal{X})$ be such that $\mathbf{1}_{K^{\epsilon/2}}\leq f\leq \mathbf{1}_{K^\epsilon}.$ By the Ces\`aro eventual continuity, we can find $\delta>0$ such that 
		\begin{equation*}
		\limsup\limits_{t\rightarrow\infty}|Q_tf(x)-Q_tf(z)|<\epsilon/2\quad\forall\,x\in B(z,\delta).
		\end{equation*}
		Then
		\begin{equation*}
		\begin{aligned}
		\liminf\limits_{t\rightarrow\infty}Q_t(x,K^{\epsilon})&\geq \liminf\limits_{t\rightarrow\infty}Q_tf(x)\geq\liminf\limits_{t\rightarrow\infty}Q_t(z,K^{\epsilon/2})-\limsup\limits_{t\rightarrow\infty}|Q_tf(x)-Q_tf(z)|\\
		&\geq 1-\epsilon
		\end{aligned}
		\end{equation*}
		for any $x\in B(z,\delta).$ Then by Fatou's lemma, for  $\mu\in\mathcal{P}(\mathcal{X})$ with $\mu(B(z,\delta))=1,$ (\ref{eq 3.12}) holds.
  
        \noindent $\mathbf{Step\;3.}$ For any $\epsilon>0$ and  $\mu\in\mathcal{P}(\mathcal{X}),$ there exists a compact set $K$ such that
		\begin{equation}\label{eq 3.13}
		\liminf\limits_{t\rightarrow\infty}Q_t\mu(K^\epsilon)\geq 1-\epsilon.
		\end{equation}\par 
		Fix $\epsilon>0$ and $\mu\in\mathcal{P}(\mathcal{X}).$ Let $\delta>0$ and compact set $K$ be given by Step 2. Define
		\begin{equation*}
		\begin{aligned}
		\gamma=\sup\{\alpha\geq 0:\exists\,N,\;t_1\dots,t_N,\;Q_{t_1,\dots,t_N}\mu\geq\alpha\nu,\;\nu\in\mathcal{P}(\mathcal{X}),\;\liminf\limits_{t\rightarrow\infty}Q_t\nu(K^\epsilon)\geq 1-\epsilon\},
		\end{aligned}
		\end{equation*}
		where $Q_{t_1,\dots,t_N}=Q_{t_1}Q_{t_2}\cdots Q_{t_N}.$ \par 
		We claim that $\gamma=1.$ Then for any $\alpha<1$ we can find $N$ and $t_1,\dots,t_N>0$ such that
		\begin{equation*}
		\begin{aligned}
		\liminf\limits_{t\rightarrow\infty}Q_t\mu(K^\epsilon)&=\liminf\limits_{t\rightarrow\infty}Q_{t,t_1,\dots,t_N}\mu(K^\epsilon)\\
		&\geq \alpha\liminf\limits_{t\rightarrow\infty}Q_t\nu(K^\epsilon)\geq\alpha(1-\epsilon).
		\end{aligned}
		\end{equation*}
		Hence (\ref{eq 3.13}) holds, finishing the proof. 
  
        \noindent $\mathbf{Step\;4.}$	Finally, we prove the claim. Assume, contrary to our claim, that $\gamma<1.$\par 
		Let $0<\eta<\delta.$ Let $\beta\in (0,1)$ be such that condition \eqref{eq 3.8} holds with $\epsilon$ replaced with $\eta.$ If $\gamma>0,$ choose $\alpha\in((\gamma-\beta/2)(1-\beta/2)^{-1},\gamma)\cap[0,1)$ and else choose  $\alpha=0.$ Then there exists $N\in\mathbb{N},\;t_1,\dots,t_N\geq 0$, and $\nu\in\mathcal{P}(\mathcal{X})$ such that
		\begin{equation*}
		Q_{t_1,\dots,t_N}\mu\geq\alpha\nu
		\end{equation*}
		and
		\begin{equation*}
		\liminf\limits_{t\rightarrow\infty}Q_t\nu(K^\epsilon)\geq 1-\epsilon.
		\end{equation*}\par 
		Set 
		\begin{equation*}
		\bar{\mu}=(1-\alpha)^{-1}(Q_{t_1,\dots,t_N}\mu-\alpha\nu)\in\mathcal{P}(\mathcal{X}).
		\end{equation*}
		Condition (\ref{eq 3.8}) and Fatou's lemma give that
		\begin{equation*}
		\liminf\limits_{t\rightarrow\infty}Q_t\bar{\mu}(B(z,\delta))\geq\beta,
		\end{equation*}
		so there exists a $t>0$ such that
		\begin{equation*}
		Q_t\bar{\mu}(B(z,\delta))\geq\beta/2.
		\end{equation*}
		Define
		\begin{equation*}
		\nu_1=\frac{Q_t\bar{\mu}(\cdot\cap B(z,\delta))}{	Q_t\bar{\mu}(B(z,\delta))}.
		\end{equation*}
		Then
		\begin{equation*}
		Q_t\bar{\mu}=(1-\alpha)^{-1}(Q_{t,t_1,\dots,t_N}\mu-\alpha Q_t\nu)\geq(\beta/2)\nu_1
		\end{equation*}
		and hence
		\begin{equation*}
		Q_{t,t_1,\dots,t_N}\mu\geq\alpha Q_t\nu+\beta/2(1-\alpha)\nu_1=[\alpha+\beta/2(1-\alpha)]\nu_2,
		\end{equation*}
		where $\nu_2=(\alpha+\beta/2(1-\alpha))^{-1}(\alpha Q_t\nu+\beta/2(1-\alpha)\nu_1)\in\mathcal{P}(\mathcal{X}).$ By the definition of $\nu$ and Step 2,
		\begin{equation*}
		\begin{aligned}
		\liminf\limits_{t\rightarrow\infty}Q_t\nu_2(K^\epsilon)&\geq(\alpha+\beta/2(1-\alpha))^{-1}[\alpha\liminf\limits_{t\rightarrow\infty}Q_t\nu(K^\epsilon)+\beta/2(1-\alpha)\liminf\limits_{t\rightarrow\infty}Q_t\nu_1(K^\epsilon)]\\
		&\geq 1-\epsilon.
		\end{aligned}
		\end{equation*}
		Observation that $\alpha+\beta/2(1-\alpha)>\gamma$ leads to contradiction.
    \end{proof}

    \paragraph{Proof of Proposition  \ref{Prop Sweep}:}
    
    \begin{proof}
    Assume, contrary to our claim, that there exists a compact set $K$ disjoint from $\mathcal{T},$  $\alpha>0$ and $\mu\in\mathcal{P}(\mathcal{X})$ such that
		\begin{equation*}\label{eq Sweep1}
		\limsup\limits_{t\rightarrow\infty}Q_t\mu(K)>\alpha.
		\end{equation*}
		As $\mathcal{T}$ is closed by Lemma \ref{Prop T4}, there exists $\eta>0$ such that $\mathcal{T}^\eta\cap K=\emptyset.$\par
		Set
		\begin{equation*}
		M:=\{\nu\in\mathcal{P}(\mathcal{X}):\text{ there exists }\xi<\eta\text{ such that }\liminf\limits_{t\rightarrow\infty}Q_t\nu(\mathcal{T}^\xi)>1-\alpha/2\}.
		\end{equation*}
		Note that $\delta_z\in M$ by Proposition \ref{Thm 2}, hence $M\neq\emptyset.$ Further, $M$ is convex and $Q_t(M)\subset M$ for any $t>0$ by in \cite[Lemma 2]{KPS2010}.\par
		Since $\delta_z\in M,$ there exists $\xi\in (0,\eta)$ and $\epsilon>0$ such that
		\begin{equation*}
		\liminf\limits_{t\rightarrow\infty}Q_t(z,\mathcal{T}^{\xi})\geq 1-\alpha/2+\epsilon>1-\alpha/2.
		\end{equation*}\par 
		Fix $\zeta>0$ such that $\xi+\zeta<\eta,$ and let $f\in L_b(\mathcal{X})$ such that $\mathbf{1}_{\mathcal{T}^{\xi}}\leq f\leq \mathbf{1}_{\mathcal{T}^{\xi+\zeta}}.$ Then by the Ces\`aro eventual continuity, there exists some $\delta>0$ such that
		\begin{equation*}
		\limsup\limits_{t\rightarrow\infty}|Q_tf(y)-Q_tf(z)|\leq\epsilon/2
		\end{equation*}
		for any $y\in B(z,\delta).$\par 
		Further, for $y\in B(z,\delta)$ we obtain
		\begin{equation*}
		\begin{aligned}
		\liminf\limits_{t\rightarrow\infty}Q_t(y,\mathcal{T}^{\xi+\zeta})&\geq \liminf\limits_{t\rightarrow\infty}Q_tf(y)\\
		&\geq \liminf\limits_{t\rightarrow\infty}Q_tf(z)-\limsup\limits_{t\rightarrow\infty}|Q_tf(y)-Q_tf(z)|\\
		&\geq 1-\alpha/2+\epsilon-\epsilon/2\\
		&>1-\alpha/2.
		\end{aligned}
		\end{equation*}
		Hence $\delta_y\in M$ for any $y\in B(z,\delta)$ and Fatou's lemma implies that $\nu\in M,$ provided that $\text{supp }\mu\subset B(z,\delta).$\par 
		Let $x\in K.$ (\ref{eq 3.5}) implies that there exists $t_x>0$ such that
		\begin{equation*}
		\alpha_x:=P_{t_x}(x,B(z,\delta/2))>0.
		\end{equation*}
		Further, using the Feller property, there exists an $r_x>0$ such that $P_{t_x}(y,B(z,\delta/2))>\alpha_x/2$ for any $y\in B(x,r_x).$ By compactness of $K,$ there exists $x_1,\dots,x_k\in K$ such that 
		\begin{equation*}
		K\subset\bigcup_{i=1}^{k}B(x_i,r_{x_i}).
		\end{equation*}\par 
		Define $\Theta:=\min_{1\leq i\leq k}\alpha_{x_i}/2$ and
		\begin{equation*}
		\gamma:=\sup\;\{\beta\geq 0:Q_{t_0}\mu\geq\beta\nu\text{ for some }\nu\in\mathcal{P}(\mathcal{X})\text{ and }t_0>0\}.
		\end{equation*} 
		Now let $\rho:=\Theta\alpha/(2k)>0,$ and choose $\nu\in M$ and $t_0>0$ such that $Q_{t_0}\mu\geq\beta\nu$ holds with $\beta>\gamma-\rho.$ Then, for any $t\geq 0,$ $Q_tQ_{t_0}\mu\geq\beta  Q_t\nu$ and $Q_t\nu\in M$ for any $t>0.$ Thus, we can choose $\nu\in M$ and $t_0$ in such a way that $Q_{t_0}\mu(K)>\alpha$ and $\nu(K)<\alpha/2.$ Consequently,
		\begin{equation*}
		(Q_{t_0}\mu-\beta\nu)(K)\geq\alpha-\alpha/2=\alpha/2.
		\end{equation*}
		Hence there exists some $j\in\{1,\dots,k\}$ such that
		\begin{equation*}
		(Q_{t_0}\mu-\beta\nu)(B(x_j,r_{x_j}))\geq\alpha/(2k).
		\end{equation*}
		Now
		\begin{equation*}
		\begin{aligned}
		P_{t_{x_j}}(Q_{t_0}\mu-\beta\nu)(B(z,\delta))&=\int_X P_{t_{x_j}}(x,B(z,\delta))(Q_{t_0}\mu-\beta\nu)(\d x)\\
		&\geq \int_{B(x_j,r_{x_j})} P_{t_{x_j}}(x,B(z,\delta))(Q_{t_0}\mu-\beta\nu)(\d x)\\
		&\geq \Theta\alpha/(2k)=\rho.
		\end{aligned}
		\end{equation*}\par
		Let 
		\begin{equation*}
		\tilde{\nu}:=\frac{(P_{t_{x_j}}Q_{t_0}\mu-\beta P_{t_{x_j}}\nu)(\cdot\cap B(z,\delta))}{(P_{t_{x_j}}Q_{t_0}\mu-\beta P_{t_{x_j}}\nu)(B(z,\delta))},
		\end{equation*}
		and it follows from $\text{supp }\tilde{\nu}\subset B(z,\delta)$ that $\tilde{\nu}\in M.$   Define
		\begin{equation*}
		\hat{\nu}:=\frac{\beta}{\beta+\rho}P_{t_{x_j}}\nu+\frac{\rho}{\beta+\rho}\tilde{\nu}.
		\end{equation*}
		As $P_{t_{x_j}}\nu,\;\tilde{\nu}\in M$ and $M$ is convex, thus $\hat{\nu}\in M.$ Furthermore,
		\begin{equation*}
		P_{t_{x_j+}t_0}\mu\geq (\beta+\rho)\hat{\nu},
		\end{equation*}
		which contradicts the fact that $\gamma<\beta+\rho$. This completes the proof.
    \end{proof}

    \subsection{Proofs of Section \ref{Sec 3.3}}
	
    \paragraph{Proof of Proposition \ref{Prop decompose}:}

    \begin{proof}
    
        $(\runum{1})$: Noting that $\mathcal{P}_{\rm erg}=\{\Phi(x):x\in\mathcal{T}_{\rm erg}\}$ by Theorem \ref{Prop T3*}, and that $\Phi$ is continuous, therefore the closedness of  $\mathcal{T}_{\rm erg}=\Phi^{-1}(\mathcal{P}_{\rm erg})$ in $\mathcal{X}$ follows form the closedness of $\mathcal{P}_{\rm erg}$ in $\mathcal{P}(\mathcal{X})$. It remains to show $\mathcal{P}_{\rm erg}$ is closed.  For $f\in C_b(\mathcal{X})$, let us define
	\begin{equation*}
	f^*(x)=\begin{cases}
	\langle f,\varepsilon_x\rangle,&\text{for }x\in\mathcal{T},\\
	0,&\text{for }x\notin\mathcal{T}. 
	\end{cases}
	\end{equation*}
        By Lemma \ref{Prop T1*}, for any $x\in\mathcal{T}$ and $f\in C_b(\mathcal{X})$, $\lim_{t\rightarrow\infty} Q_tf(x)=\langle f,\varepsilon_x\rangle=f^*(x)$.  By Proposition \ref{Prop T2*}, one has $\mu(\mathcal{T})=1$. Thus $\mu$ is ergodic if and only if $f^*=\langle f,\mu\rangle$ for $\mu$-almost surely for any $f\in C_b(\mathcal{X})$, or equivalently, 
        \begin{equation}\label{mu ergodic}
            \int_{\mathcal{T}}|f^*(x)-\langle f,\mu\rangle|^2\mu(\d x)=0.
        \end{equation}

        Let $\{\mu_n\}_{n\geq 1}\subset\mathcal{P}_{\rm erg}$ be such that         
        \begin{equation}\label{mu ergodic1}            
        \lim_{n\rightarrow\infty}\|\mu_n-\mu\|_L^*=0.
        \end{equation}
        It suffices to prove that $\mu\in\mathcal{P}_{\rm erg}$. Thanks to Lemma \ref{Prop T1*}, $\mu$ is an invariant measure for $\{P_t\}_{t\geq 0}$. In what follows, we show that \eqref{mu ergodic} holds for any $f\in C_b(\mathcal{X})$, which implies that $\mu\in\mathcal{P}_{\rm erg}$. Let $f\in C_b(\mathcal{X})$ be fixed. Making use of the ergodicity of $\mu_n$,
        \begin{equation}\label{mu ergodic2}
        \begin{aligned}
             \int_{\mathcal{T}}|f^*(x)-\langle f,\mu\rangle|^2\mu(\d x)&=\int_{\mathcal{T}}|f^*(x)-\langle f,\mu\rangle|^2\mu(\d x)-\int_{\mathcal{T}}|f^*(x)-\langle f,\mu_n\rangle|^2\mu_n(\d x)\\
             &=\int_{\mathcal{T}}(|f^*(x)-\langle f,\mu\rangle|^2-|f^*(x)-\langle f,\mu_n\rangle|^2)\mu_n(\d x)\\
             &\quad+\int_{\mathcal{T}}|f^*(x)-\langle f,\mu\rangle|^2(\mu(\d x)-\mu_n(\d x))\\
             &:=I_1+I_2.
        \end{aligned}           
        \end{equation}

        To deal with the first term $I_1$ on the right hand side of \eqref{mu ergodic2}, one has
        \begin{equation*}
            \begin{aligned}
                ||f^*(x)-\langle f,\mu\rangle|^2-|f^*(x)-\langle f,\mu_n\rangle|^2|&\leq 2|f^*(x)||\langle f,\mu\rangle-\langle f,\mu_n\rangle|+|\langle f,\mu\rangle^2-\langle f,\mu_n\rangle^2|\\
                &\leq 4\|f\|_{\infty}|\langle f,\mu\rangle-\langle f,\mu_n\rangle|.
            \end{aligned}
        \end{equation*}
        Thus \eqref{mu ergodic1} implies that
        \begin{equation}\label{mu ergodic3}
            \lim_{n\rightarrow\infty}|I_1|=0.
        \end{equation}

         For the second term $I_2$, as $\mathcal{T}$ is closed and $|f^*-\langle f,\mu\rangle|^2$ is a bounded and continuous function on $\mathcal{T}$,  the Tietze extension theorem implies that there exists a $g\in C_b(\mathcal{X})$ such that $g=|f^*-\langle f,\mu\rangle|^2$ on $\mathcal{T}$. Combining \eqref{mu ergodic1}, we deduce that
         \begin{equation}\label{mu ergodic4}
             \lim_{n\rightarrow\infty}|I_2|=\lim_{n\rightarrow\infty}|\langle f,\mu\rangle-\langle f,\mu_n\rangle|=0.
         \end{equation}

        In summary, plugging \eqref{mu ergodic3} and \eqref{mu ergodic4} into \eqref{mu ergodic2}, we conclude that \eqref{mu ergodic1} holds for any $f\in C_b(\mathcal{X})$, finishing the proof.

         $(\runum{2})$: Using the continuity of $\Phi$, it follows that $[x]=\Phi^{-1}(\{\varepsilon_x\})$ is closed. Moreover, by Theorem \ref{Prop T3*}, for any $y\in\supp\varepsilon_x$, $\{Q_t(y,\cdot)\}$ weakly converges to $\varepsilon_x$ as $t\rightarrow\infty$, and $\supp\varepsilon_x$ is closed. Thus $\supp\varepsilon_x\subset[x]$.

         $(\runum{3})$: This proof follows the ideas of \cite[Theorem 3.4]{SW2012}. By $(\runum{2})$, for any $\mu\in\mathcal{P}_{\rm erg}$,  $\mu=\varepsilon_x$ for any $x\in\supp\mu$. Thus we can write
         \begin{equation*}
             D=\{x\in\mathcal{T}_{\rm erg}:x\in\supp\varepsilon_x\}=\bigcap_{k\in\N}D_k,
         \end{equation*}
         where 
         \begin{equation*}
             D_k=\{x\in\mathcal{T}_{\rm erg}:\varepsilon_x(B(x,1/k))>0\}.
         \end{equation*}
         Next we show that $D_k':=\mathcal{T}_{\rm erg}\setminus D_k$ is closed for any $k\in\N$. Let $\{x_n\}_{n\geq 1}\subset D_k'$ converging to $x\in\mathcal{X}$. It suffices to show $x\in D_k'$. Noting that $\mathcal{T}_{\rm erg}$ is closed by $(\runum{1})$, $x\in\mathcal{T}_{\rm erg}$ and $\varepsilon_{x_n}\rightarrow\varepsilon_{x}$ as $n\rightarrow\infty$. For $N\in\N$, let
         \begin{equation*}
             A_N:=B(x,1/k)\cap(\bigcap_{n\geq N}B(x_n,1/k)).
         \end{equation*}
         Then one can check that $A_N$ is open, and $A_N$ increases to $B(x,1/k)$ as $N\rightarrow\infty$. In the mean time, it follows that
         \begin{equation*}
            \varepsilon_x(A_N)\leq \liminf\limits_{n\rightarrow\infty}\varepsilon_{x_n}(A_N)\leq \liminf\limits_{n\rightarrow\infty}\varepsilon_{x_n}(B(x_n,1/k))=0.
         \end{equation*}
         We therefore derive $\varepsilon_x(B(x,1/k))=0$, implying $x\in D_k'$. 
         Consequently, taking the closedness of $\mathcal{T}_{\rm erg}$ into account and writing $D$ as 
         \begin{equation*}
             D=\mathcal{T}_{\rm erg}\cap \bigcap_{k\in\N}(\mathcal{X}\setminus D_k'),
         \end{equation*}
         we show that $D$ is a $G_\delta$ set.
    \end{proof}

    \subsection{Proofs of  Section \ref{Sec 4}}
	
    \paragraph{Proof of Proposition \ref{IFS Prop 1}:}
    \begin{proof}
	    Denote $\{\Omega_n(t):\omega:\tau_n(\omega)\leq t<\tau_{n+1}(\omega)\},$ for $n\geq 1,\; \mathbf{i}=(i_1,\dots,i_n)\in I^n,$
	\begin{equation*}
	\begin{aligned}
	&\mathbf{x}_0(\mathbf{i}[0];x)=x,\\
	&\mathbf{x}_k(\mathbf{i}[k];x)=S(\triangle\tau_{k+1})(w_{i_n}(\mathbf{x}_{k-1}(\mathbf{i}[k-1];x))),\;1\leq k\leq n.\\
	\end{aligned}
	\end{equation*}
	Moreover, we set $\mathbf{p}_n(\mathbf{i};x)=\prod_{k=1}^{n}p_{i_k}(w_{\mathbf{i}[k]}(x)).$
	We have
	\begin{equation*}
	\begin{aligned}
	&\sum_{\mathbf{i}\in I^n}|\mathbf{p}_n(\mathbf{i};x)-\mathbf{p}_n(\mathbf{i};z)|\\
	&\quad\leq \sum_{\mathbf{i}\in I^n}\mathbf{p}_{n-1}(\mathbf{i}[n-1];x)|p_{i_n}(w_{\mathbf{i}}(x))-p_{i_n}(w_{\mathbf{i}}(z))|\\
     &\quad\quad+\sum_{\mathbf{i}\in I^{n}}|\mathbf{p}_{n-1}(\mathbf{i}[n-1];x)-\mathbf{p}_{n-1}(\mathbf{i}[n-1];z)|p_{i_n}(w_{\mathbf{i}}(x))\\
	&\quad\leq\sum_{\mathbf{i}\in I^{n-1}}\mathbf{p}_{n-1}(\mathbf{i}[n-1];x)\omega(\rho(\mathbf{x}_{n-1}(\mathbf{i};x),\mathbf{x}_{n-1}(\mathbf{i};z)))+\sum_{\mathbf{i}\in I^{n-1}}|\mathbf{p}_{n-1}(\mathbf{i};x)-\mathbf{p}_{n-1}(\mathbf{i};z)|\\
	&\quad\leq \omega(\sum_{\mathbf{i}\in I^{n-1}} \mathbf{p}_{n-1}(\mathbf{i}[n-1];x)\rho(\mathbf{x}_{n-1}(\mathbf{i};x),\mathbf{x}_{n-1}(\mathbf{i};z)))+\sum_{\mathbf{i}\in I^{n-1}}|\mathbf{p}_{n-1}(\mathbf{i};x)-\mathbf{p}_{n-1}(\mathbf{i};z)|\\
	&\quad\leq \sum_{k=1}^{n}\omega(J_{k}(x)\e^{\alpha\tau_k}\rho(x,z)).
	\end{aligned}
	\end{equation*}
	Then for $f\in L_b(\mathcal{X})$
	\begin{equation*}
	\begin{aligned}
	&\sum_{\mathbf{i}\in I^n}|f(S(t-\tau_n)(\mathbf{x}(\mathbf{i};x))\mathbf{p}_n(\mathbf{i};x)-f(S(t-\tau_n)(\mathbf{x}(\mathbf{i};z))\mathbf{p}_n(\mathbf{i};z)|\\
	&\quad\leq\|f\|_{\infty}\sum_{\mathbf{i}\in I^{n}}|\mathbf{p}_n(\mathbf{i};x)-\mathbf{p}_n(\mathbf{i};z)|+\|f\|_{\rm Lip}\sum_{\mathbf{i}\in I^{n}}\mathbf{p}_n(\mathbf{i};x)\rho(S(t-\tau_n)(\mathbf{x}(\mathbf{i};x),S(t-\tau_n)(\mathbf{x}(\mathbf{i};z))\\
	&\quad\leq\|f\|_\infty \sum_{k=1}^{n}\omega((J_{k}(x)\e^{\alpha\tau_k}\rho(x,z))+\|f\|_{\rm Lip}J_{n}(x)\e^{\alpha t}\rho(x,z).
	\end{aligned}
	\end{equation*}
	
	Thus we obtain that for $x\in B(z,\eta),$
	\begin{equation}\label{eq IFS J0}
	\begin{aligned}
	|P_tf(x)-&P_tf(z)|=|\mathbb{E}f(\Phi^x(t))-\mathbb{E}f(\Phi^z(t))|\\
	&\leq \sum_{n=M}^{\infty}\mathbb{E}(\mathbf{1}_{\Omega_n(t)}|f(S(t-\tau_n)(\Phi_n^x)-f(S(t-\tau_n)(\Phi_n^z)|)+2\|f\|_{\infty}P(t\leq\tau_{M})\\
	&\leq \sum_{n=M}^{\infty}\mathbb{E}(\mathbf{1}_{\Omega_n(t)}\sum_{\mathbf{i}\in I^n}|f(S(t-\tau_n)(\mathbf{x}(\mathbf{i};x))\mathbf{p}_n(\mathbf{i};x)-f(S(t-\tau_n)(\mathbf{x}(\mathbf{i};z))\mathbf{p}_n(\mathbf{i};z)|)\\
        &\quad+2\|f\|_{\infty}P(t\leq\tau_{N})\\
	&\leq \|f\|_\infty\sum_{n=M}^{\infty}\sum_{k=1}^{n}\mathbb{E}(\mathbf{1}_{\Omega_n(t)}\omega(J_{k}(x)\e^{\alpha\tau_k}\rho(x,z))+\|f\|_{\rm Lip}\sum_{n=M}^{\infty}J_n(x)\e^{\alpha t}\rho(x,z)\frac{(\lambda t)^n}{n!}\e^{-\lambda t}\\
	&\quad+2\|f\|_{\infty}P(t\leq\tau_{N})\\
    &:= J_1+J_2+J_3.
	\end{aligned}
	\end{equation}
    
	To estimate $J_1$, we first recall an useful inequality in \cite[Proposition 3]{BKS2014} as follows.
	\begin{equation}\label{eq IFS J4}
	    \mathbb{E}\sum_{n=k}^{\infty}\mathbf{1}_{\Omega_n(t)}\e^{\alpha\tau_k}\leq\frac{\lambda^k}{(\lambda-\alpha)^k},\quad k\geq 1.
	\end{equation} 
    By $\mathbf{B_2}$ and $(\ref{eq IFS J4}),$ we have that
	\begin{equation*}
	\begin{aligned}
	J_1&=\|f\|_\infty\sum_{k=1}^{M}\sum_{n=M}^{\infty}\mathbb{E}(\mathbf{1}_{\Omega_n(t)}\omega(J_{k}(x)\e^{\alpha\tau_k}\rho(x,z))+\|f\|_\infty\sum_{k=M+1}^{\infty}\sum_{n=k}^{\infty}\mathbb{E}(\mathbf{1}_{\Omega_n(t)}\omega(J_{k}(x)\e^{\alpha\tau_k}\rho(x,z))\\
	&\leq \|f\|_\infty\sum_{k=1}^{M}\omega(J_k(x)\rho(x,z))+\|f\|_\infty\sum_{k=M}^{\infty}\omega(J_k(x)\rho(x,z)\frac{\lambda^k}{(\lambda-\alpha)^k}).
	\end{aligned}
	\end{equation*}
	Then by $\mathbf{B_4},$ the right hand side is uniformly bounded, and hence 
	\begin{equation}\label{eq IFS J1}
	    \limsup\limits_{x\rightarrow z}J_1=0.
	\end{equation} 
	
	To estimate $J_2$, by (\ref{eq IFS 4}) it is easy to know that 
	\begin{equation*}
	    \limsup\limits_{n\rightarrow \infty}J_n(x)\frac{\lambda^n}{(\lambda-\alpha)^n}=0\quad\text{for }x\in B(z,\eta).
	\end{equation*} 
    Hence $\forall\,x\in B(z,\eta)$, there exists $K_x\geq M$ such that
	\begin{equation*}
	    J_n(x)\frac{\lambda^n}{(\lambda-\alpha)^n}\leq 1\quad\text{for }n\geq K_x.
	\end{equation*}\par 
	Further, we obtain that
	\begin{equation*}
	    \begin{aligned}
	    \sum_{n=M}^{\infty}J_n(x)\e^{\alpha t}\frac{(\lambda t)^n}{n!}\e^{-\lambda t}&= \sum_{n=M}^{K_x}J_n(x)\e^{\alpha t}\frac{(\lambda t)^n}{n!}\e^{-\lambda t}+ \sum_{n>K_x}J_n(x)\e^{\alpha t}\frac{(\lambda t)^n}{n!}\e^{-\lambda t}\\
	    &\leq \sum_{n=M}^{K_x}\frac{(\lambda t)^n}{n!}\e^{-(\lambda-\alpha) t}+\sum_{n>K_x}\frac{((\lambda-\alpha)t)^n}{n!}\e^{-(\lambda-\alpha)t}\\
	    &\leq \sum_{n=M}^{K_x}\frac{(\lambda t)^n}{n!}\e^{-(\lambda-\alpha) t}+1.
	    \end{aligned}
	\end{equation*}
	Hence
	\begin{equation*}
	    J_2\leq \|f\|_{\rm Lip}\rho(x,z)[\sum_{n=M}^{K_x}\frac{(\lambda t)^n}{n!}\e^{-(\lambda-\alpha) t}+1],
	\end{equation*}
	we have that
	\begin{equation}\label{eq IFS J2}
	    \limsup\limits_{x\rightarrow z}\limsup\limits_{t\rightarrow \infty} J_2=0.
	\end{equation}\par 
	Finally, it is clearly that 
	\begin{equation}\label{eq IFS J3}
	    \limsup\limits_{t\rightarrow \infty} J_3=0.
	\end{equation}

    Combining (\ref{eq IFS J0})-(\ref{eq IFS J3}),	we prove that $\{P_t\}_{t\geq 0}$ is eventually continuous at $z$ by (\ref{eq IFS 4}). 
    \end{proof}

    \paragraph{Proof of Theorem \ref{Thm IFS}: }
    
    \begin{proof}
    In view of Theorem \ref{Thm 4}, it suffices to check condition (\ref{eq 4.1}). Fixing any $\delta>0,$ we can take $f\in L_b(\mathcal{X})$ such that $\mathbf{1}_{B(z,\delta/2)}\leq f\leq \mathbf{1}_{B(z,\delta)}$ with $\|f\|_{\infty}\leq 1 $ and $\|f\|_{\rm Lip}\leq 4/\delta.$ Note that $P_tf(z)=1$ by assumption $\mathbf{C_2}$.\par 
	 By $\mathbf{C_1},$ for $x\in B(z,\delta),$
	\begin{equation*}
	\begin{aligned}
	1-P_t(x,B(z,\delta))
	&\leq |P_tf(z)-P_tf(x)|\\
	&\leq  \|f\|_\infty\sum_{k=1}^{M}\omega(J_k(x)\rho(x,z))+\|f\|_\infty\sum_{k=M}^{\infty}\omega\left(J_k(x)\rho(x,z)\frac{\lambda^k}{(\lambda-\alpha)^k}\right)\\
	&\quad+\|f\|_{\rm Lip}\rho(x,z)\left[\sum_{n=M}^{K_x}\frac{(\lambda t)^n}{n!}\e^{-(\lambda-\alpha) t}+1\right]\\
	&\quad+2\|f\|_{\infty}P(t\leq\tau_{M})\\
	&\leq 1-\gamma +\sum_{k=1}^{M}\omega(J_k(x)\rho(x,z))+4\delta^{-1}\rho(x,z)\left[\sum_{n=M}^{K_x}\frac{(\lambda t)^n}{n!}\e^{-(\lambda-\alpha) t}+1\right]\\
        &\quad +2P(t\leq\tau_{M}).
	\end{aligned}
	\end{equation*} 
    Let $\eta'=\eta'(\delta)\in (0,\eta)$ be sufficiently small such that
    \begin{equation*}
        \sum_{k=1}^{M}\omega(J_k(x)\rho(x,z))+4\delta^{-1}\eta'\leq \gamma/2\quad\text{for }x\in B(z,\eta').
    \end{equation*} 
    Then 
    \begin{equation*}
        \liminf\limits_{t\rightarrow\infty} P_t(x,B(z,\delta))\geq \gamma/2>0 \quad\text{for }x\in B(z,\eta').
    \end{equation*}\par 
    Finally, by (\ref{eq IFS 6}), taking $\epsilon=\eta',$ $\forall\,x\in \mathcal{X}$ there exists $\beta>0$ and $t_x$ such that 
    \begin{equation*}
        P_{t_x}(x,B(z,\eta'))\geq\beta>0.
    \end{equation*}\par 
    Finally, for $t$ sufficiently large, we obtain that
    \begin{equation*}
    \begin{aligned}
        P_t(x,B(z,\delta))&=\int_{y\in\mathcal{X}}P_{t-t_x}(y,B(z,\delta))P_{t_x}(x,\d y)\\
        &\geq P_{t_x}(x,B(z,\eta'))\inf\limits_{y\in B(z,\eta')}P_{t-t_x}(y,B(z,\delta))\\
        &\geq \beta\gamma/4>0.
    \end{aligned}
    \end{equation*}

	Therefore, condition (\ref{eq 4.1}) is satisfied for such $\{P_t\}_{t\geq 0}$, and Theorem \ref{Thm 4} implies that $\{P_t\}_{t\geq 0}$ is asymptotically stable.
	\end{proof}
	
	\paragraph{Proof of Feller property of IFS:}
	
    \begin{proposition}\label{IFS Prop Feller} The Markov semigroup $\{P_t\}_{t\geq 0}$ generated by $\{\Phi_t^x\}_{t\geq 0}$ is Feller.
	\end{proposition}
	\begin{proof}
	    Fix $f\in C_b(\mathcal{X}),\,x\in\mathcal{X},\,\epsilon>0.$ Noting that
	    \begin{equation*}
	        \lim\limits_{M\rightarrow\infty}\sum_{n=M}^{\infty}\frac{(\lambda t)^n}{n!}\e^{-\lambda t}=0,
	    \end{equation*}
	    we can choose $M$ sufficiently large such that
	    \begin{equation*}
	        2\|f\|_{\infty}\sum_{n=M}^{\infty}\frac{(\lambda t)^n}{n!}\e^{-\lambda t}\leq \epsilon.
	    \end{equation*}
	    Then we obatin
	    \begin{equation*}
	        \begin{aligned}
	        |P_tf(y)-P_tf(x)|&=|\mathbb{E}[f(\Phi_t^y)-f(\Phi_t^x)]|=\left|\sum_{n=0}^{\infty} \mathbb{E}\mathbf{1}_{\Omega_n(t)}[f(\Phi_t^x)-f(\Phi_t^y)]\right|\\
	        &\leq \sum_{n=0}^{M-1}\mathbb{E}\mathbf{1}_{\Omega_n(t)}|f(\Phi_t^x)-f(\Phi_t^y)|+\epsilon.
	        \end{aligned}
	    \end{equation*}
	       Hence by the dominated convergence theorem,
    \begin{equation*}
        \lim\limits_{y\rightarrow x}|P_tf(y)-P_tf(x)|\leq \epsilon,
    \end{equation*}
    which implies the Feller property, since $\epsilon$ is arbitrary.
	\end{proof}

    \subsection{Proofs of Section \ref{Sec 5}}\label{pr of beyond} 

    Let us consider a SDE on $\R_+$ given by
    \begin{equation}\label{eq r sde}
         \d r(t)=(a+\tfrac{1}{2}b^2)r(t)\d t-r^3(t)\d t+br(t)\d \beta_t,\quad r(0)=r_0\in \R_+,
    \end{equation}
    where $a\in\R$, $b\neq 0$ are constants, and $\{\beta_t\}_{t\in\R}$ is a standard (two-sided time) Brownian motion on a canonical space $(\Omega,\mathcal{F},\P,\{\theta_t\}_{t\in\R})$.  By definition, $\beta_t(\omega)=\omega(t)$ and $\theta_t\omega(\cdot)=\omega(\cdot-t)-\omega(t)$ for $\in\R$. Clearly $r\equiv0$ is an stationary solution of Equation \eqref{eq r sde}. Let $\{\hat{\beta}_t\}_{t\in\R}$ be another standard Brownian motion that is independent of $\beta$. In particular, one has
    \begin{equation*}
        \mathcal{D}(\{\beta_s-\beta_t:0\leq s\leq t\})=\mathcal{D}(\{\hat{\beta}_{s-t}:0\leq s\leq t\})\quad \forall\,0\leq t\leq \infty.
    \end{equation*}
    For $a\geq 0$, let 
    \begin{equation*}
        \lambda(\omega)=\left(2\int_{-\infty}^{0}\e^{2at+2b\hat{\beta}_t(\omega)}\d t\right)^{-1/2}.
    \end{equation*}
    
	\begin{lemma}\label{Lemma r-convergence}
	    In the case of $a<0$, it follows that
     \begin{equation*}
         \lim\limits_{t\rightarrow\infty}r(t)\overset{a.s.}{=}0\quad\forall\, r_0>0.
     \end{equation*}
      In the case of $a\geq 0$, $\mathcal{D}(\lambda)$ is an  invariant measure of Equation \eqref{eq r sde}. Moreover, 
      \begin{equation*}
         \lim\limits_{t\rightarrow\infty}r(t)\overset{d}{=}\lambda\quad\forall\,r_0>0,
     \end{equation*}
     and 
     \begin{equation*}
         \lambda\overset{a.s.}{=}0\quad\text{ for }a=0.
     \end{equation*}
	\end{lemma}
 
    \begin{proof} 
        Note that Equation \eqref{eq r sde} has solution in the form of 
        \begin{equation*}
            r(t)=\frac{r_0\e^{at+b\beta_t}}{\left(1+2(r_0)^2\int_{0}^{t}\e^{2as+2b\beta_s}\d s\right)^{1/2}}.
        \end{equation*}
        
       \noindent $\mathbf{Step\;1.}$ When $a<0$, for any $r_0>0$, one has
        \begin{equation*}
            r(t,\omega)\leq r_0\exp\left(a+b\tfrac{\beta_t(\omega)}{t}\right)\rightarrow 0\quad \text{ a.s. as } t\rightarrow\infty,
        \end{equation*}
        where we use the fact that $\lim\limits_{t\rightarrow\infty}\beta_t/t\overset{a.s.}{=}0$. 

        \noindent $\mathbf{Step\;2.}$ When $a\geq 0$, we first verify that $\lambda$ is well-defined.   When $a=0$, let $N(\omega):=\{t\leq 0:b\hat{\beta}_t(\omega)\leq 0\}$. Then 
        \begin{equation*}
            \int_{-\infty}^{0}e^{2b\hat{\beta}_t(\omega)}dt\geq \int_{-\infty}^{0}\mathbf{1}_{N(\omega)}(t)dt=\text{Leb }N(\omega).
        \end{equation*}
        By the arc sine law,
        \begin{equation*}
            \limsup\limits_{t\rightarrow-\infty}\frac{\text{Leb }\{s\geq t:b\hat{\beta}_t(\omega)\leq 0\}}{-t}\overset{a.s.}{=}1,
        \end{equation*}
        which implies that $\lambda\overset{a.s.}{=}0$ for $a=0$.

        For $a>0$, using the law of the iterated logarithm, we know that
        \begin{align*}
            \limsup\limits_{t\rightarrow-\infty}\frac{\hat{\beta}_t(\omega)}{\sqrt{2|t|\ln\ln |t|}}\overset{a.s.}{=}1,\quad \liminf\limits_{t\rightarrow-\infty}\frac{\hat{\beta}_t(\omega)}{\sqrt{2|t|\ln\ln |t|}}\overset{a.s.}{=}-1,
        \end{align*}
        which implies that $0<\lambda<\infty$ almost surely for $a>0$.

        \noindent $\mathbf{Step\;3.}$ Next is to check the invariance of  $\mathcal{D}(\lambda)$. It suffices to consider the case when $a>0$. Let 
        \begin{equation*}
            r^*(\omega)=\left(2\int_{-\infty}^{0}\e^{2at+2b\beta_t(\omega)}\d t\right)^{-1/2}.
        \end{equation*}
        Then one has
        \begin{align*}
            r^*(\theta_t\omega)&=\left(2\int_{-\infty}^{0}\e^{2as+2b\beta_s(\theta_t\omega)}\d t\right)^{-1/2}=\left(2\int_{-\infty}^{0}\e^{2as+2b(\beta_{s+t}(\omega)-\beta_{t}(\omega))}\d t\right)^{-1/2}\\
            &=\left(2\int_{-\infty}^{t}\e^{2a(s-t)+2b(\beta_{s}(\omega)-\beta_{t}(\omega))}\d t\right)^{-1/2}\\
            &=\e^{at+b\beta_t(\omega)}\left(2\int_{-\infty}^{0}\e^{2as+2b\beta_{s}(\omega)}\d t+2\int_{0}^{t}\e^{2as+2b\beta_{s}(\omega)}\d t\right)^{-1/2}\\
            &=\frac{r^*(\omega)\e^{at+b\beta_t(\omega)}}{\left(1+2(r^*(\omega))^2\int_{0}^{t}\e^{2as+2b\beta_{s}(\omega)}\d t\right)^{1/2}},
        \end{align*}
        which is exactly the solution of Equation \eqref{eq r sde} with initial condition $r(0)=r^*(\omega)$ at time $t$. Meanwhile, since $\mathcal{D}(\theta_t\omega)=\mathcal{D}(\omega)$ for any $t\geq 0$, it follows that
        \begin{equation*}
            \mathcal{D}(\lambda)=\mathcal{D}(r^*(\cdot))=\mathcal{D}(r^*(\theta_t\cdot)),
        \end{equation*}
        ensuring the invariance.

        \noindent $\mathbf{Step\;4.}$ It remains to derive the weak convergence. Indeed, it follows that for $a\geq 0$ and $r_0>0$,
        
        \begin{align*}
            r(t)&=\left((r_0)^{-2}\e^{-2at-2b\beta_t}+2\int_{0}^{t}\e^{2a(s-t)+2b(\beta_s-\beta_t)}\d s\right)^{-1/2}\\
            &\overset{d}{=}\left((r_0)^{-2}\e^{-2at+2b\hat{\beta}_{-t}}+2\int_{-t}^{0}\e^{2as+2b\hat{\beta}_s}\d s\right)^{-1/2}\overset{a.s.}{\longrightarrow} \lambda\quad\text{ as }t\rightarrow\infty,
        \end{align*}
        which implies the desired convergence property.
    \end{proof}

    \paragraph{Proof of Corollary \ref{Coro lorenz}:}
    \begin{proof}
     
        As $\supp\nu_*$ is a closed set in $\R^3$, it suffices to show that $\R^3\setminus H\subset \supp\nu_*$. To this end, let us first fix any $\hat{\boldsymbol{x}}_0\in\R^3\setminus H$ and $\epsilon>0$. Form the proof of \cite[Proposition 3.3]{CotiHairer}, one can deduce that for any $\boldsymbol{x}_0\in\R^3\setminus H$, there exists $T_{\boldsymbol{x}_0}$ such that for any $t\geq T_{\boldsymbol{x}_0}$, there exists $h_{\boldsymbol{x_0},t}\in C^1([0,t];\R)$ with $h_{\boldsymbol{x_0},t}(0)=0$ satisfying
        \begin{equation*}
            \|(x(t),y(t),z(t))-\hat{\boldsymbol{x}}_0\|<\tfrac{\epsilon}{2}.
        \end{equation*}
        Here $(x(\cdot),y(\cdot),z(\cdot))$ denoted the solution of Equation \eqref{eq lorenz2} with $h=h_{\boldsymbol{x_0},t}$. 

        Moreover, a standard perturbative argument shows that there exists $\delta>0$ such that for any $f\in C([0,t];\R)$ with $f(0)=0$, $\|f(\cdot)-h_{\boldsymbol{x}_0,t}(\cdot)\|_{C([0,t];\R)}\leq \delta$, the solution of Equation \eqref{eq lorenz2} with $h=f$ satisfying
        \begin{equation}\label{eq lorenz3}
            \|(x(t),y(t),z(t))-\hat{\boldsymbol{x}}_0\|<\epsilon.
        \end{equation}

       Consequently, by \eqref{eq lorenz3} and the invariance of $\nu_*$, we conclude that 
        \begin{align*}
            \nu_*(B(\hat{\boldsymbol{x}}_0,\epsilon))&=\int_{0}^{\infty}e^{-t}\int_{\R^3\setminus H}P_t(\boldsymbol{x}_0,B(\hat{\boldsymbol{x}}_0,\epsilon))\nu_*(\d \boldsymbol{x}_0)\d t\\
            &\geq\int_{\R^3\setminus H}\nu_*(\d \boldsymbol{x}_0) \int_{T_{\boldsymbol{x}_0}}^{\infty}e^{-t}\P(\|W_{\cdot}-h_{\boldsymbol{x}_0,t}(\cdot)\|_{C([0,t];\R)}\leq \delta)\d t\\
            &>0,
        \end{align*}
         completing the proof.
    \end{proof}

    \vspace{5mm}
    \noindent {\small {\bf Acknowledgments:} Fuzhou Gong and Yuan Liu are supported by National Key R$\&$D Program of China (No. 2020YFA0712700) and Key Laboratory of Random Complex Structures and Data Sciences, Mathematics and Systems Science, Chinese Academy of Sciences (No. 2008DP173182); Yong Liu is supported by CNNSF (No. 11731009, No. 12231002) and Center for Statistical Science, PKU; Yuan Liu is supported by (CNNSF 11688101). The authors thank Prof. Yiqian Wang, Mr. Jiehao Wan and Mr. Sixian Lai for discussing \cite{CotiHairer}.   }

    \bibliographystyle{abbrv}
    \bibliography{reference}

\begin{thebibliography}{10}

\bibitem{A1998}
L.~Arnold.
\newblock {\em Random dynamical systems}.
\newblock Springer, Berlin, 1998.

\bibitem{ArndV}
V.~I. Arnold.
\newblock {\em Mathematical methods of classical mechanics}, volume~60 of {\em Graduate Texts in Mathematics}.
\newblock Springer-Verlag, New York, second edition, 1989.
\newblock Translated from the Russian by K. Vogtmann and A. Weinstein.

\bibitem{A1976}
W.~Arveson.
\newblock {\em An invitation to $C^*$-algebras}.
\newblock Springer, New York, 1976.

\bibitem{BKS2014}
H.~Bessaih, R.~Kapica, and T.~Szarek.
\newblock Criterion on stability for {M}arkov processes applied to a model with jumps.
\newblock {\em Semigroup Forum}, 88(1):76--92, 2014.

\bibitem{ChenMF2005}
M.~Chen.
\newblock {\em Eigenvalues, inequalities, and ergodic theory}.
\newblock Probability and its Applications (New York). Springer-Verlag London, Ltd., London, 2005.

\bibitem{Chenmf1989}
M.~F. Chen and S.~F. Li.
\newblock Coupling methods for multidimensional diffusion processes.
\newblock {\em Ann. Probab.}, 17(1):151--177, 1989.

\bibitem{CotiHairer}
M.~Coti~Zelati and M.~Hairer.
\newblock A noise-induced transition in the {L}orenz system.
\newblock {\em Comm. Math. Phys.}, 383(3):2243--2274, 2021.

\bibitem{C2017}
D.~Czapla.
\newblock A criterion on asymptotic stability for partially equicontinuous {M}arkov operators.
\newblock {\em Stoch. Proc. Appl.}, 128(11):3656--3678, 2017.

\bibitem{CH2014}
D.~Czapla and K.~Horbacz.
\newblock Equicontinuity and stability properties of {M}arkov chains arising from iterated function systems on {P}olish spaces.
\newblock {\em Stoch. Anal. Appl.}, 32(1):1--29, 2014.

\bibitem{CHW-S2020}
D.~Czapla, K.~Horbacz, and H.~Wojew\'{o}dka-\'{S}ci\c{a}\.{z}ko.
\newblock Ergodic properties of some piecewise-deterministic markov process with application to gene expression modelling.
\newblock {\em Stochastic Process. Appl.}, 130:2851--2885, 2020.

\bibitem{DPZ1996}
G.~Da~Prato and J.~Zabczyk.
\newblock {\em {E}rgodicity for infinite dimensional systems}.
\newblock Cambridge University Press, Cambridge, 1996.

\bibitem{DongYL2016}
Y.~Dong.
\newblock Ergodicity of stochastic differential equations driven by {L\'evy} noise with local {Lipschitz} coefficient.
\newblock {\em Adv. Math. (China)}, 47(1):11--30, 2018.

\bibitem{DMPS2018}
R.~Douc, E.~Moulines, P.~Priouret, and P.~Soulier.
\newblock {\em Markov chains}.
\newblock Springer Series in Operations Research and Financial Engineering. Springer, Cham, 2018.

\bibitem{GHM2017}
N.~E. Glatt-Holtz, J.~C. Mattingly, and G.~Richards.
\newblock On unique ergodicity in nonlinear stochastic partial differential equations.
\newblock {\em J. Stat. Phys.}, 166(3):618--649, 2017.

\bibitem{GL2015}
F.~Gong and Y.~Liu.
\newblock Ergodicity and asymptotic stability of {F}eller semigroups on {P}olish metric spaces.
\newblock {\em Sci. China Math.}, 58(6):1235--1250, 2015.

\bibitem{GLLL2024}
F.~Gong, Y.~Liu, Y.~Liu, and Z.~Liu.
\newblock Asymtotic stability for non-equicontinuous {Markov} semigroups.
\newblock {\em Commun. Math. Stat.}, to appear.

\bibitem{HM2006}
M.~Hairer and J.~C. Mattingly.
\newblock Ergodicity of the 2{D} {N}avier-{S}tokes equations with degenerate stochastic forcing.
\newblock {\em Ann. of Math. (2)}, 164(3):993--1032, 2006.

\bibitem{HM2011}
M.~Hairer and J.~C. Mattingly.
\newblock A theory of hypoellipticity and unique ergodicity for semilinear stochastic {PDE}s.
\newblock {\em Electron. J. Probab.}, 16:658--738, 2011.

\bibitem{HSZ2017}
S.~C. Hille, T.~Szarek, and M.~A. Ziemlańska.
\newblock Equicontinuous families of {M}arkov operators in view of asymptotic stability.
\newblock {\em C. R. Math. Acad. Sci. Paris}, 355(12):1247--1251, 2017.

\bibitem{Hopf1948}
E.~Hopf.
\newblock A mathemetical example displaying features of turbulence.
\newblock {\em Communications on Pure and Appl. Math.}, 1, 1948.
\newblock See also in {\it Selected works of Eberhard Hopf with commentaries}. Edited by Cathleen S. Morawetz, James B. Serrin and Yakov G. Sinai. American Mathematical Society, Providence, RI, 2002.

\bibitem{Hopf1956}
E.~Hopf.
\newblock Repeated branching through loss of stability: an example.
\newblock In {\em Proceedings of the conference on differential equations (dedicated to A. Weinstein)}, pages 49--56. University of Maryland Book Store, College Park, Md., 1956.
\newblock See also in {\it Selected works of Eberhard Hopf with commentaries}. Edited by Cathleen S. Morawetz, James B. Serrin and Yakov G. Sinai. American Mathematical Society, Providence, RI, 2002.

\bibitem{H2002}
K.~Horbacz.
\newblock Randomly connected differential equations with {P}oisson type perturbations.
\newblock {\em Nonlinear Stud.}, 9(1):81--98, 2002.

\bibitem{HMS2005}
K.~Horbacz, J.~Myjak, and T.~Szarek.
\newblock On stability of some general random dynamical system.
\newblock {\em J. Stat. Phys.}, 119(1-2):35--60, 2005.

\bibitem{HMS2006}
K.~Horbacz, J.~Myjak, and T.~Szarek.
\newblock Stability of random dynamical system on banach spaces.
\newblock {\em Positivity}, 10(3):517--538, 2006.

\bibitem{J1964}
B.~Jamison.
\newblock Asymptotic behavior of successive iterates of continuous functions under a {M}arkov operator.
\newblock {\em J. Math. Anal. Appl.}, 9:203–214, 1964.

\bibitem{J1965}
B.~Jamison.
\newblock Ergodic decompositions induced by certain {M}arkov operators.
\newblock {\em Trans. Amer. Math. Soc.}, 117:451--468, 1965.

\bibitem{J2013+}
J.~Jaroszewska.
\newblock The asymptotic strong {Feller property does not imply the e-property of Markov-Feller} semigroups., 2013.
\newblock Preprint at \url{https://arxiv.org/pdf/1308.4967v1}.

\bibitem{J2013}
J.~Jaroszewska.
\newblock On asymptotic equicontinuity of {M}arkov transition functions.
\newblock {\em Stat. Probabil. Lett.}, 83(3):943--951, 2013.

\bibitem{KSS2012}
R.~Kapica, T.~Szarek, and M.~\'{S}leczka.
\newblock On a unique ergodicity of some {M}arkov processes.
\newblock {\em Potential Anal.}, 36(4):589--606, 2012.

\bibitem{KPS2010}
T.~Komorowski, S.~Peszat, and T.~Szarek.
\newblock On ergodicty of some {M}arkov processes.
\newblock {\em Ann. Probab.}, 38(4):1401--1443, 2010.

\bibitem{KW2021}
R.~Kukulski and H.~Wojew\'{o}dka-\'{S}ci\c{a}\.{z}ko.
\newblock The e-property of asymptotically stable {Markov} operators.
\newblock {\em Colloq. Math.}, 165(2):269--283, 2021.

\bibitem{KW2024}
R.~Kukulski and H.~Wojew\'{o}dka-\'{S}ci\c{a}\.{z}ko.
\newblock The e-property of asymptotically stable {Markov} semigroups.
\newblock {\em Results Math.}, 79(3):No. 112, 22 pp, 2024.

\bibitem{Kulik2018}
A.~Kulik.
\newblock {\em Ergodic behavior of {M}arkov processes}, volume~67 of {\em De Gruyter Studies in Mathematics}.
\newblock De Gruyter, Berlin, 2018.

\bibitem{KS2018}
A.~Kulik and M.~Scheutzow.
\newblock Generalized couplings and convergence of transition probabilities.
\newblock {\em Probab. Theory Related Fields}, 171(1):333--376, 2018.

\bibitem{Kup2011}
A.~Kupiainen.
\newblock Ergodicity of two dimensional turbulence (after {Hairer and Mattingly}).
\newblock In {\em Séminaire Bourbaki. Vol. 2009/2010. Exposés 1012–1026}. Astérisque 339, Exp. No. 1016, vii, 137-156, 2011.

\bibitem{Lasota1995}
A.~Lasata.
\newblock From fractals to stochastic differential equations.
\newblock In {\em Chaos---the interplay between stochastic and deterministic behaviour ({K}arpacz, 1995)}, volume 457 of {\em Lecture Notes in Phys.}, pages 235--255. Springer, Berlin, 1995.

\bibitem{LS2006}
A.~Lasota and T.~Szarek.
\newblock Lower bound technique in the theory of a stochastic differential equation.
\newblock {\em J. Differential Equations}, 231(2):513--533, 2006.

\bibitem{LT2003}
A.~Lasota and J.~Traple.
\newblock Invariant measures related with {P}oisson driven stochastic differential equation.
\newblock {\em Stochastic Process. Appl.}, 106(1):81--93, 2003.

\bibitem{LY1994}
A.~Lasota and J.~A. Yorke.
\newblock Lower bound technique for {M}arkov operators and iterated function systems.
\newblock {\em Random Comput. Dynam}, 2(1):41--77, 1994.

\bibitem{LL2024}
Y.~Liu and Z.~Liu.
\newblock Relation between the eventual continuity and the e-property.
\newblock {\em Acta Math. Appl. Sin. Engl. Ser.}, 40(1):1--16, 2024.

\bibitem{MT2009}
S.~P. Meyn and R.~L. Tweedie.
\newblock {\em Markov chains and stochastic stability}.
\newblock Cambridge University Press, Cambridge, {S}econd edition, 2009.

\bibitem{Moser}
J.~Moser.
\newblock On the theory of quasiperodic motions.
\newblock {\em SIAM Review}, 8(2):145--172, 1966.

\bibitem{PZZ2024}
X.~Peng, J.~Zhai, and T.~Zhang.
\newblock Ergodicity for {2D Navier-Stokes} equations with a degenerate pure jump noise, 2024.
\newblock Preprint at \url{https://arxiv.org/pdf/2405.00414}.

\bibitem{R1964}
M.~Rosenblatt.
\newblock Equicontinuous {M}arkov operators.
\newblock {\em Teor. Verojatnost. i Primenen.}, 9:205--222, 1964.

\bibitem{S2000+}
T.~Szarek.
\newblock Generic properties of learning systems.
\newblock {\em Ann. Polon. Math.}, 73(2):93--103, 2000.

\bibitem{S2000}
T.~Szarek.
\newblock The stability of {M}arkov operators on {P}olish spaces.
\newblock {\em Studia Math.}, 143(2):145--152, 2000.

\bibitem{S2003}
T.~Szarek.
\newblock Invariant measures for {M}arkov operators with application to function systems.
\newblock {\em Studia Math.}, 154(3):207 -- 222, 2003.

\bibitem{S2003+}
T.~Szarek.
\newblock Invariant measures for nonexpansive {M}arkov operators on {P}olish spaces.
\newblock {\em Diss. Math.}, 415:1--62, 2003.

\bibitem{S2006}
T.~Szarek.
\newblock {F}eller processes on nonlocally compact spaces.
\newblock {\em Ann. Probab.}, 34(5):1849 -- 1863, 2006.

\bibitem{SSU2010}
T.~Szarek, M.~Sleczka, and M.~Urba\'nski.
\newblock On stability of velocity vectors for some passive tracer models.
\newblock {\em Bull. Lond. Math. Soc.}, 42(5):923--936, 2010.

\bibitem{SW2012}
T.~Szarek and D.~T.~H. Worm.
\newblock Ergodic measures of {M}arkov semigroups with the e-property.
\newblock {\em Ergodic Theory Dynam. Systems}, 32(3):1117--1135, 2012.

\bibitem{Wangfy2014}
F.-Y. Wang.
\newblock {\em Analysis for diffusion processes on {R}iemannian manifolds}, volume~18 of {\em Advanced Series on Statistical Science \& Applied Probability}.
\newblock World Scientific Publishing Co. Pte. Ltd., Hackensack, NJ, 2014.

\bibitem{Zhai2024+}
J.~Wang, H.~Yang, and J.~Zhai.
\newblock Irreducibility of stochastic complex {G}inzburg-{L}andau equations driven by pure jump noise and its applications.
\newblock {\em Appl. Math. Optim.}, 89(2):Paper No. 47, 21, 2024.

\bibitem{Zhai2024}
J.~Wang, H.~Yang, J.~Zhai, and T.~Zhang.
\newblock Accessibility of {SPDE}s driven by pure jump noise and its applications.
\newblock {\em Proc. Amer. Math. Soc.}, 152(4):1755--1767, 2024.

\bibitem{WW2018}
S.~Wedrychowicz and A.~Wi\'{s}nicki.
\newblock On some results on the stability of {M}arkov operators.
\newblock {\em Studia Math.}, 241:41--55, 2018.

\bibitem{W2010}
D.~T.~H. Worm.
\newblock {\em Semigroups on spaces of measures}.
\newblock PhD thesis, Leiden University, 2010.

\bibitem{WH2010}
D.~T.~H. Worm and S.~C. Hille.
\newblock Equicontinuous families of {M}arkov operators on complete separable metric spaces with applications to ergodic decompositions and existence, uniqueness and stability of invariant measures.
\newblock {\em unpublished}, 2010.

\bibitem{WH2011}
D.~T.~H. Worm and S.~C. Hille.
\newblock An ergodic decomposition defined by regular jointly measurable {M}arkov semigroups on {P}olish spaces.
\newblock {\em Acta Appl. Math.}, 116(1):27--53, 2011.

\bibitem{WH2011-0}
D.~T.~H. Worm and S.~C. Hille.
\newblock Ergodic decompositions associated with regular {M}arkov operators on {P}olish spaces.
\newblock {\em Ergodic Theory Dynam. Systems}, 31(2):571--597, 2011.

\bibitem{Z2005}
R.~Zaharopol.
\newblock {\em Invariant probabilities of {M}arkov-{F}eller operators and their supports}.
\newblock Birkh\"{a}user Verlag, Basel, 2005.

\bibitem{Z2014}
R.~Zaharopol.
\newblock {\em Invariant probabilities of transition functions}.
\newblock Springer, Cham, 2014.

\end{thebibliography}

    \end{document}